\documentclass[final,onefignum,onetabnum]{siamonline250211}
\usepackage{enumitem}
\usepackage{graphicx}
\usepackage{subcaption}
\usepackage{caption}
\usepackage{float}
\usepackage{booktabs}

\usepackage{lipsum}
\usepackage{amsfonts}
\usepackage{amssymb}
\usepackage{amsmath}
\usepackage{graphicx}
\usepackage{epstopdf}
\usepackage{stmaryrd}
\usepackage{algorithmic}
\usepackage{tikz}
\usepackage{graphicx}
\ifpdf
  \DeclareGraphicsExtensions{.eps,.pdf,.png,.jpg}
\else
  \DeclareGraphicsExtensions{.eps}
\fi
\DeclareMathOperator*{\argmax}{arg\,max}
\DeclareMathOperator*{\argmin}{arg\,min}

\newsiamremark{remark}{Remark}
\newsiamremark{hypothesis}{Hypothesis}
\crefname{hypothesis}{Hypothesis}{Hypotheses}
\newsiamthm{claim}{Claim}
\newsiamremark{fact}{Fact}
\crefname{fact}{Fact}{Facts}

\headers{Geometry and factorization of multivariate Markov chains}{Michael C.H. Choi, Youjia Wang, and Geoffrey Wolfer}

\title{Geometry and factorization of multivariate Markov chains with applications to MCMC acceleration and approximate inference\thanks{Submitted to the editors in April 2025.}}

\author{Michael C.H. Choi\thanks{Department of Statistics and Data Science, National University of Singapore, Singapore (\email{mchchoi@nus.edu.sg}).}
\and Youjia Wang\thanks{Department of Statistics and Data Science, National University of Singapore, Singapore (\email{e1124868@u.nus.edu}).}
\and Geoffrey Wolfer\thanks{Department of Electrical Engineering and Computer Science, Tokyo University of Agriculture and Technology, Tokyo, Japan. Part of this research was conducted while the author was at RIKEN AIP and then at Waseda University. (\email{wolfer@go.tuat.ac.jp}).}}

\usepackage{amsopn}

\input{symbols}

\ifpdf
\hypersetup{
  pdftitle={Geometry and factorization of multivariate Markov chains with applications to MCMC acceleration and approximate inference},
  pdfauthor={Michael C.H. Choi, Youjia Wang, and Geoffrey Wolfer}
}
\fi

\begin{document}

\maketitle

\begin{abstract}
This paper analyzes the factorizability and geometry of transition matrices of multivariate Markov chains. Specifically, we demonstrate that the induced chains on factors of a product space can be regarded as information projections with respect to the Kullback-Leibler divergence. This perspective yields Han-Shearer type inequalities and submodularity of the entropy rate of Markov chains, as well as applications in the context of large deviations and mixing time comparison. As concrete algorithmic applications in Markov chain Monte Carlo (MCMC) {\color{black}and approximate inference}, we provide \textcolor{black}{three} illustrations based on lifted MCMC, swapping algorithm \textcolor{black}{and factored filtering} to demonstrate projection samplers improve mixing over the original samplers. The projection sampler based on the swapping algorithm resamples the highest-temperature coordinate at stationarity at each step, and we prove that such practice accelerates the mixing time by multiplicative factors related to the number of temperatures and the dimension of the underlying state space when compared with the original swapping algorithm. Through simple numerical experiments on a bimodal target distribution, we show that the projection samplers mix effectively, in contrast to lifted MCMC and the swapping algorithm, which mix less well. \textcolor{black}{In filtering, our proposed factored filtering scheme is able to scale to high dimensions with linear-in-dimension computational cost per step at the price of an approximation error that can be tracked using the distance to independence, compared with the exponential-in-dimension cost per step of the exact filter.}
\end{abstract}

\begin{keywords}
Markov chains, $f$-divergence, mutual information, Markov chain Monte Carlo, large deviations, matrix nearness problem, swapping algorithm, \textcolor{black}{filtering}, Han's inequality, Shearer's lemma, submodularity, information geometry
\end{keywords}

\begin{MSCcodes}
60F10, 60J10, 60J22, 94A15, 94A17
\end{MSCcodes}

\section{Introduction}
Consider two random variables $X,Y$ on a common finite state space $\Omega$. We denote by $\mathcal{P}(\Omega)$ to be the set of probability masses on $\Omega$. We write $p_{(X,Y)} \in \mathcal{P}(\Omega^2)$ to be the joint probability mass of $(X,Y)$ while $p_X$ (resp.~$p_Y$) $\in \mathcal{P}(\Omega)$ denotes the marginal probability mass of $X$ (resp.~$Y$) with respect to $p_{(X,Y)}$. Recall that the Kullback-Leibler (KL) divergence from $\nu$ to $\mu$ with $\mu,\nu \in \mathcal{P}(\Omega)$ is given by
\begin{align}\label{eq:widetildeDKL}
	\widetilde{D}_{KL}(\mu \| \nu) := \sum_{x \in \Omega} \mu(x) \ln \left(\dfrac{\mu(x)}{\nu(x)}\right),
\end{align}
where the usual convention of $0 \ln (0/0) := 0$ applies.
With the above setup in mind, the classical notion of mutual information between $X,Y$ is defined to be
\begin{align*}
	I(X;Y) := \widetilde{D}_{KL}(p_{(X,Y)} \| p_X \otimes p_Y),
\end{align*}
where the symbol $(\mu \otimes \nu) (x,y) := \mu(x) \nu(y)$ for all $x,y \in \Omega$ denotes the product distribution of $\mu$ and $\nu$. Note that $\mu \otimes \nu \in \mathcal{P}(\Omega^2)$. It can be shown, via the chain rule of the KL divergence \cite[Theorem $2.5.3$]{CT06}, that $p_X \otimes p_Y$ is the unique closest product distribution to $p_{(X,Y)}$ in the sense that
\begin{align}\label{eq:mutual}
	I(X;Y) = \min_{\mu,\nu \in \mathcal{P}(\Omega)} \widetilde{D}_{KL}(p_{(X,Y)} \| \mu  \otimes \nu).
\end{align}
In other words, the mutual information $I(X;Y)$ can broadly be interpreted as an entropic ``distance" to the closest product distribution, that is, an entropic distance to independence. From \eqref{eq:mutual}, we immediately see that
$I(X;Y)$ is non-negative, and vanishes if and only if $X,Y$ are independent \cite[equation $(2.90)$]{CT06}.

In the context of Markov chains, we consider two transition matrices $M,L$ on a common finite state space $\mathcal{X}$ and we write the set of all transition matrices on $\mathcal{X}$ to be $\mathcal{L}(\mathcal{X})$. Let $f:\mathbb{R}^+ \to \mathbb{R}$ be a convex function with $f(1) = 0$, and $\pi \in \mathcal{P}(\mathcal{X})$. Analogous to $f$-divergence between infinitesimal generators of continuous-time Markov chains in \cite[Proposition $1.5$]{DM09}, we define the $f$-divergence from $L$ to $M$ with respect to $\pi$ to be
\begin{align*}
	D_f^{\pi}(M\|L):=\sum_{x\in \mathcal{X}}\pi(x)\sum_{y\in\mathcal{X}}L(x,y)f\left(\dfrac{M(x,y)}{L(x,y)}\right),
\end{align*}
where several standard conventions apply in this definition, see Definition \ref{def:f-divergence} below. Note that $\pi$ is arbitrary and $M,L$ may or not admit $\pi$ as their respective stationary distribution. In the context of Markov chain Monte Carlo (MCMC), we can naturally choose $\pi$ as the target distribution and consider the $f$-divergence of two MCMC samplers from $L$ to $M$ with respect to this chosen $\pi$. In the special case of $f(t) = t \ln t$ that generates the KL divergence, we shall write $D^{\pi}_{KL}$, and this coincides with the KL divergence rate from $L$ to $M$ when these two admit $\pi$ as the stationary distribution. For $M \in \mathcal{L}(\mathcal{X}^{(1)})$ and $L \in \mathcal{L}(\mathcal{X}^{(2)})$, their  tensor product $M \otimes L \in \mathcal{L}(\mathcal{X}^{(1)} \times \mathcal{X}^{(2)})$ is defined to be
$$(M \otimes L)((x^1,x^2),(y^1,y^2)) := M(x^1,y^1)L(x^2,y^2),$$
where $x^i,y^i \in \mathcal{X}^{(i)}$ for $i = 1,2$. Suppose now the state space $\mathcal{X} = \mathcal{X}^{(1)} \times \ldots \times \mathcal{X}^{(d)}$ takes on a product form for $d \in \mathbb{N}$. Given a $P \in \mathcal{L}(\mathcal{X})$, what is the closest product Markov chain? To put this question in concrete applications, we can think of an interacting particle system with $d$ particles or agents, such as the voter model \cite{L04}, that is described by a transition matrix $P$. What is the dynamics of the closest independent system to $P$ in which each particle or agent evolves independently of each other? That is, we are interested in seeking a minimizer of 
\begin{align*}
	\mathbb{I}_f^{\pi}(P) := \min_{L_i \in \mathcal{L}(\mathcal{X}^{(i)}),\, \forall i \in \llbracket d \rrbracket} D_f^{\pi}(P \| \otimes_{i=1}^d L_i),
\end{align*}
which is analogous to the classical mutual information as in \eqref{eq:mutual}. Note that we write $\llbracket a,b \rrbracket := \{a, a+1, \ldots, b\}$ for $a,b \in \mathbb{Z}$ and $\llbracket d \rrbracket := \llbracket 1,d \rrbracket$ for $d \in \mathbb{N}$. We will first verify that $\mathbb{I}_f^{\pi}$ shares similar geometric properties with the distance to product distributions (i.e. mutual information between two random variables). As a result, $\mathbb{I}_f^{\pi}(P)$ can be interpreted as a distance to independence of a given $P$.

Orthogonality considerations allow us to identify and determine the closest product chain under KL divergence. This can be seen as a Markov chain version of the matrix nearness problem investigated in \cite{LL23,H89}, as we are seeking the closest product chain from a given $P$. In the dual case when $f$ generates the reverse KL divergence, we present a large deviation principle of Markov chains where $\mathbb{I}_f^{\pi}(P)$ plays a role in the exponent of large deviation probability. These results are presented in Section \ref{subsec:distanceind} and \ref{subsubsec:larged} below.

Note that in \cite{Lacker23} a similar problem has been investigated in the context of diffusion processes, where the author studies the closest independent diffusion process of a given multivariate diffusion process and identifies the associated Wasserstein gradient flow and consequences for the McKean-Vlasov equation. On the other hand in the present manuscript, we focus on the closest independent Markov chain problem and the underlying geometry induced by information divergences between transition matrices. 

We proceed to generalize these notions further. We introduce the leave-one-out, or more generally leave-$S$-out transition matrix, and investigate the factorizability of a transition matrix with respect to a partition or cliques of a given graph in Section \ref{subsec:leaveoneout} to \ref{subsec:clique}. Observing that leave-$S$-out transition matrices are instances of Markov chain decomposition \cite{JSTV04} or induced chains \cite{AF02}, we deduce comparison results for hitting and mixing time parameters such as spectral gap and log-Sobolev constant between $P$ and its information projections.

{\color{black}The connection between $P$ and its leave-$S$-out chains admits a natural interpretation as a Markov chain analogue of Rao--Blackwellization. In classical inference, Rao--Blackwellization replaces a given estimator by its conditional expectation with respect to a sufficient statistic, thereby improving statistical efficiency. In our setting, the leave-$S$-out transition matrices may be viewed as conditional expectations or suitably averaged versions of $P$, see Remark~\ref{rk:condite} below, which enjoy improved mixing properties when compared with $P$.}

Harnessing on these notions we design and propose a projection sampler based upon the swapping algorithm in Section \ref{sec:projectionsamplers}. The sampler can be considered as a starting-state-randomized swapping algorithm: at each step the first or the highest-temperature coordinate is refreshed according to its stationary distribution. We prove that such practice accelerates the mixing time with multiplicative factors related to the number of temperatures and the dimension of the underlying state space. This provides a concrete example where the notion of projection can be applied to improve the design of MCMC algorithms.

{\color{black}In addition to the design of improved MCMC samplers, in Section \ref{sec:factorfilter} we demonstate the usage of projected Markov chains in approximate inference. Specifically, we propose a factored particle filter where we replace the original Markov dynamics in the prediction step of filtering by its product approximation. We shall demonstrate the advantage of such factored filter: its computational cost is linear in the dimension while the original exact filter is exponential in dimension, at the trade-off of introducing approximation error. We illustrate numerically that the distance to independence $\mathbb{I}_f^{\pi}(P)$ where $f$ generates the Kullback-Leibler divergence can serve as a measure of this approximation error.}

\textcolor{black}{We conclude this introduction by providing a simple motivating example in the context of lifted MCMC, where projection can yield improved sampler.}

\subsection{Motivating examples: lifted samplers}\label{subsec:liftedMCMC}

As a simple illustration to demonstrate the idea of projection samplers, we consider lifted MCMC samplers followed by projection to further improve mixing.

Precisely, consider a Metropolis-Hastings chain with transition matrix $Q = Q(M, \pi^{(1)})$ on the state space $\mathcal{X}^{(1)}$, where $M$ is the proposal chain and $\pi^{(1)}$ is the target distribution that we seek to sample from. For simplicity in this example we shall consider $\mathcal{X}^{(1)} = \llbracket -n,n \rrbracket$ for $n \in \mathbb{N}$.

To add memory and to avoid diffusive-like behaviour in the dynamics, one acceleration method is to consider the lifted Metropolis-Hastings chain with transition matrix $P$ on the augmented state space $\mathcal{X} = \mathcal{X}^{(1)} \times \{-1,+1\}$, where the second coordinate can now be interpreted as a direction or velocity variable. Specifically, we consider $P$ to be of run-and-tumble type \cite{GGR21}. From an initial state of $(x,v) \in \mathcal{X}$, where
$\mathcal{X}=\llbracket -n,n\rrbracket \times \{-1,+1\}$, $P$ moves according to the following rules:
\begin{itemize}
	\item (Position move) With probability $a$, $P$ moves from $(x,v)$ to $(y,v)$ according to $Q$.
	
	\item {\color{black}(Directed Metropolis move) With probability $b$, $P$ proposes to move from $x$ to $x+v$. Define
	\[
	\alpha_v(x)
	:=
	\begin{cases}
		\displaystyle \min\left\{1,\frac{\pi^{(1)}(x+v)}{\pi^{(1)}(x)}\right\},
		& \text{if } x+v \in \llbracket -n,n\rrbracket,\\[1.2em]
		0,
		& \text{otherwise.}
	\end{cases}
	\]
	With probability $\alpha_v(x)$, the proposal is accepted and $P$ moves from $(x,v)$ to $(x+v,v)$. With probability $1-\alpha_v(x)$, the proposal is rejected and $P$ moves from $(x,v)$ to $(x,-v)$.}
	
	\item (Flipping the direction) With probability $c$, $P$ moves from $(x,v)$ to $(x,-v)$.
\end{itemize}
We suppose that $a,b,c>0$ and $a+b+c = 1$. \textcolor{black}{To avoid overburdening the introduction, we defer the proof of $\pi = \pi^{(1)} \otimes \mathcal{U}(\{-1,+1\})$-stationarity of the lifted sampler $P$ to Proposition~\ref{prop:pistationaritylifted}}, where $\mathcal{U}(\{-1,+1\})$ denotes the discrete uniform distribution on the two-point space $\{-1,+1\}$. Such $P$ can be understood as a simplified version of the kinetic random walks or Langevin diffusions \cite{M20}, run-and-tumble models \cite{GGR21} or Gustafson's guided walk samplers \cite{G98}.

\textcolor{black}{Note that we keep the explicit velocity-flip move with probability $c$ separate from the rejection-induced flip in the directed Metropolis move. These two mechanisms have different roles. The rejection-induced flip is part of the Metropolis correction that ensures stationarity of the directed move, and its probability $b(1-\alpha_v(x))$ is state-dependent. In contrast, the explicit flip with probability $c$ is a state-independent velocity refreshment mechanism. It provides an additional tuning parameter controlling the persistence of the velocity variable and helps avoid relying solely on rejected proposals to change direction. In the case of $c = 0$ and suppose the sampler is currently at a region where proposals are nearly always accepted, then the velocity may persist for a long time. While this might be desirable in some situations, it can also produce poor exploration behaviour of the sampler.}

We are interested in comparing the following three MCMC samplers:
\begin{itemize}
	\item Original $Q$. This is the baseline MH chain on the position space $\mathcal{X}^{(1)}$ that we wish to sample from and is being lifted to $P$.
	\item Lifted sampler $P$, and we discard the samples associated with the direction coordinate
	\item The keep-$\{1\}$-in projection sampler $P^{(1)}$, see Definition \ref{def:P-S} below. To simulate one step of $P^{(1)}$, we adopt the following procedure:
	\begin{enumerate}[label=(\roman*)]
		\item Starting from $x \in \mathcal{X}^{(1)}$, draw $v_1 \sim \mathcal{U}(\{-1,+1\})$;
		\item Draw $(y, v_2)\sim P\left((x, v_1),\cdot\right)$;
		\item Update $x \leftarrow y$. 
	\end{enumerate}
	Note that $P^{(1)}$ depends on the lifted sampler $P$, and similar velocity/momentum refreshment is a common technique used in Hamiltonian Monte Carlo, see e.g. \cite[Section $5.3$]{Neal2011HMC}.
\end{itemize}

Comparing $Q$ and $P^{(1)}$, the latter can be understood as a suitably perturbed version of the former with moves that might have a small probability of taking place in $Q$. For instance, with probability $b$, $P^{(1)}$ is able to move from $x$ to $x+v$, in which such proposal move might have a smaller probability in $Q$ than $b$.

Comparing $P$ and $P^{(1)}$, the intuitive rationale for the acceleration effect of $P$ lies in the added memory because of the inclusion of the velocity coordinate. For $P^{(1)}$, such memory seems to be lost as its velocity coordinate is randomized at each step according to $\mathcal{U}(\{-1,+1\})$.

The relationship between $P$ and $P^{(1)}$ can be intuitively interpreted as a Markov chain version of the Rao--Blackwellization. In estimation, Rao-Blackwellization proposes to consider estimators based on conditional expectations and sufficient statistics to improve the original estimators. In our context, $P^{(1)}$ can be understood as a conditional expectation of $P$, see Remark \ref{rk:condite} below. In this lifted sampler example, in $P^{(1)}$ one integrates out or refreshes the effect of the velocity component at each step.

The results established in this paper give that it is favourable to consider $P^{(1)}$ over $P$ for improved mixing: the KL divergence from $\Pi$ to $P$ is at least greater than or equal to that from $\Pi^{(1)}$ to $P^{(1)}$ ($\Pi$ or $\Pi^{(1)}$ are respectively the matrix where each row is $\pi$ or $\pi^{(1)}$), see Corollary \ref{cor:contraction}. In addition, the multiplicative spectral gap of $P^{(1)}$ is at least as good as that of $P$, see Corollary \ref{cor:mspectralgap}.

\subsubsection{Numerical experiments}\label{subsubsec:numlifted}

For reproducibility, the code used in our experiments is available at \url{https://github.com/mchchoi/factorization/tree/main}. We first state the parameters of the experiments:
\begin{itemize}
	\item $\pi^{(1)}(x) \propto 2^{|x|}$. Such bimodal V-shaped target distribution is commonly used to assess the performance of MCMC samplers, see e.g. \cite{DHM00,MZ03}. Notably there are two modes at $\pm n$ respectively.
	\item {\color{black}$n = 15$.}
	\item $M$, the proposal chain, moves from $x$ to $\min\{x+1,n\}$ and $\max\{x-1,-n\}$ with probability $1/2$, and $0$ otherwise.
	\item {\color{black}$a = c = \frac{1}{8}, b = \frac{3}{4}$.}
	\item {\color{black}All samplers are initialized at $15$, the mode on the right, and are simulated for $2,500,000$ steps. The initial velocity of the sampler $P$ is $-1$.}
\end{itemize}

The results are summarized and presented in Figure \ref{fig:Vshape}, Table \ref{tab:first} and \ref{tab:second}.

{\color{black}First, we note that $Q$ does not exhibit mixing: from the traceplot, histogram and empirical mean, it spent most time exploring the basin around the right mode at $15$ and rarely traverse to the left mode at $-15$ in the experiment.}

{\color{black}Second, from the traceplots and histograms we see that both $P$ and $P^{(1)}$ are able to hop between the two modes. One notable difference is that $P$ has more frequent hopping compared with $P^{(1)}$. While the histogram that these two generated are visually similar and arguably $P^{(1)}$ is a better fit than $P$, from Table \ref{tab:first} the empirical distribution generated by $P^{(1)}$ is closer to the ground truth $\pi^{(1)}$ than that generated by $P$. From Table \ref{tab:second} the empirical mean  generated by $P^{(1)}$ is also closer to the ground truth than that generated by $P$, while the empirical second moment of $P^{(1)}$ and $P$ are similar. These two tables seem to suggest that $P^{(1)}$ mixes better than the $x$-coordinate of $P$, thus offering empirical evidence that it is advantageous to use the projection sampler $P^{(1)}$ over either $P$ or $Q$ to sample from $\pi^{(1)}$. As a result, for the purpose of estimating $\pi^{(1)}(g)$ where $g$ is either $g(x) = x$ or $g(x) = x^2$, it is perhaps preferable to use samples generated by $P^{(1)}$ over that of $P$ or $Q$.}
\begin{figure}[htbp]
	\centering
	
	\begin{subfigure}{\textwidth}
		\centering
		\includegraphics[width=0.45\textwidth]{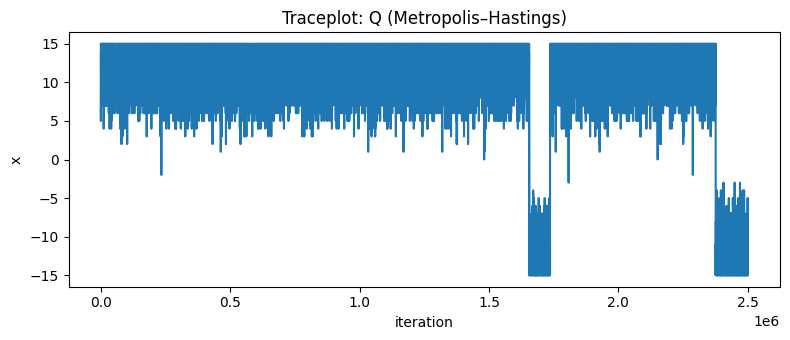}
		\includegraphics[width=0.45\textwidth]{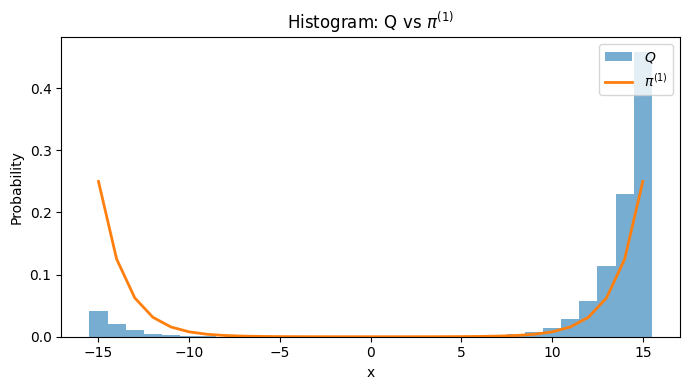}
		\caption{Traceplot and histogram of the trajectories of $Q$.}
	\end{subfigure}
	
	\begin{subfigure}{\textwidth}
		\centering
		\includegraphics[width=0.45\textwidth]{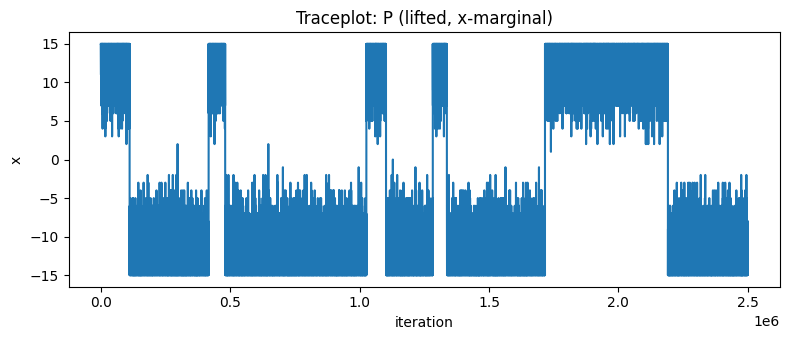}
		\includegraphics[width=0.45\textwidth]{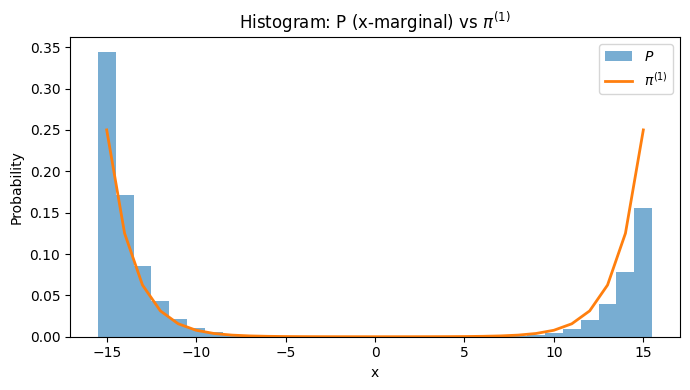}
		\caption{Traceplot and histogram of the trajectories of the $x$-coordinate of $P$.}
	\end{subfigure}
	
	\begin{subfigure}{\textwidth}
		\centering
		\includegraphics[width=0.45\textwidth]{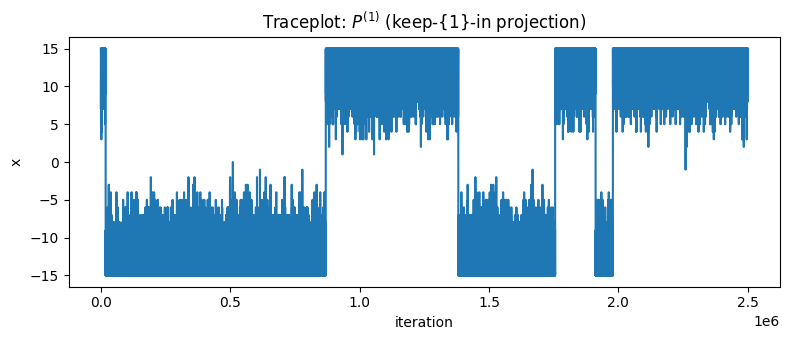}
		\includegraphics[width=0.45\textwidth]{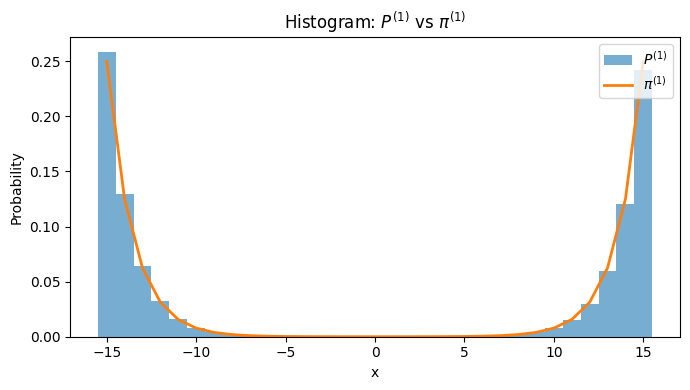}
		\caption{Traceplot and histogram of the trajectories of $P^{(1)}$.}
	\end{subfigure}
	
	\caption{Numerical experiments comparing the three samplers $Q, P, P^{(1)}$ with target distribution being the V-shaped $\pi^{(1)}(x) \propto 2^{|x|}$.}
	\label{fig:Vshape}
\end{figure}

\begin{table}[H]
	\centering
	{\color{black}
	\begin{tabular}{lcc}
		\toprule
		Sampler & $\widetilde{D}_{TV}(\widehat{\pi}^{(1)},\pi^{(1)})$ & $\widetilde{D}_{KL}(\widehat{\pi}^{(1)} \| \pi^{(1)})$ \\
		\midrule
		Q  & 0.42 & 0.41 \\
		$P$ ($x$-only) & 0.19 & 0.07 \\
		$P^{(1)}$ & 0.02 & 0.00 \\
		\bottomrule
	\end{tabular}}
	\caption{Comparison of total variation distance and KL divergence between $\widehat{\pi}^{(1)}$ and the ground truth $\pi^{(1)}$, where $\widehat{\pi}^{(1)}$ is the empirical distribution formed by the trajectories of the samplers. Recall that $\widetilde{D}_{TV}(\mu,\nu):= \tfrac{1}{2}\sum_x |\mu(x) - \nu(x)|$ and $\widetilde{D}_{KL}$ is defined in \eqref{eq:widetildeDKL}.}
	\label{tab:first}
\end{table}

\begin{table}[H]
	\centering
	{\color{black}
	\begin{tabular}{lcc}
		\toprule
		Sampler & Mean & Second moment \\
		\midrule
		Q  & 11.70 & 198.01 \\
		$P$ ($x$-only) & -5.25 & 198.09 \\
		$P^{(1)}$ & -0.48 & 198.14 \\
		Truth $\pi^{(1)}$ & 0 & 198.00 \\
		\bottomrule
	\end{tabular}}
	\caption{Comparison of the first and second moment between the samplers and the ground truth $\pi^{(1)}$.}
	\label{tab:second}
\end{table}

\subsection{Organization of the paper}

The rest of this paper is organized as follows. In Section \ref{sec:basicdef}, we first recall the notion of $f$-divergences between transition matrices and probability measures. We derive a few important properties of these divergences on finite product state spaces, which allow us to define an entropic distance to independence of a given multivariate $P$ in Section \ref{subsec:distanceind}. In Section \ref{subsubsec:larged}, we determine and identify the closest product chain under KL divergence, and present a large deviation principle in this context. We investigate the factorizability of $P$ with respect to partition or cliques of a given graph in Section \ref{subsec:leaveoneout} to \ref{subsec:clique}. In Section \ref{subsec:comparisonfunctionalconstants}, we compare mixing and hitting time parameters between $P$ and its information projections, while in Section \ref{subsec:submodular}, we show that several entropic functions that naturally arise in this paper are in fact submodular. To illustrate the applicability of projection chains, in Section \ref{sec:projectionsamplers} we propose a projection sampler and compare its mixing time with the original swapping algorithm, along with some simple numerical experiments in Section \ref{subsec:num}. {\color{black}As another application of projection chains, in Section \ref{sec:factorfilter} we propose a factored filtering scheme with linear-in-dimension computational cost per step when compared with the exact filter and provide related experiments.}

\section{Distance to independence and factorizability of Markov chains}\label{sec:basicdef}

On a finite state space $\mathcal{X}$, we define $\mathcal{L} = \mathcal{L}(\mathcal{X})$ as the set of transition matrices of discrete-time homogeneous Markov chains. We denote by $\mathcal{P}(\mathcal{X})$ to be the set of probability masses on $\mathcal{X}$. Let $\pi \in \mathcal{P}(\mathcal{X})$ be any given positive probability distribution \textcolor{black}{(i.e. $\pi$ satisfies $\min_x \pi(x) > 0$)}, and denote $\mathcal{L}(\pi)\subset \mathcal{L}$ as the set of $\pi$-reversible transition matrices on $\mathcal{X}$, where a transition matrix $P \in \mathcal{L}$ is said to be $\pi$-reversible if $\pi(x)P(x,y)=\pi(y)P(y,x)$ for all $x,y\in \mathcal{X}$. We also say that $P \in \mathcal{L}$ is $\pi$-stationary if it satisfies $\pi P = \pi$. Suppose that $P$ is $\pi$-stationary, then the $\pi$-dual or the time reversal of $P$, $P^*$, is defined to be
$P^*(x,y) := \frac{\pi(y)}{\pi(x)} P(y,x)$, for all $x,y\in \mathcal{X}$.

First, we give the definition of $f$-divergence of Markov chains and recall that of probability measures.

\begin{definition}[$f$-divergence of Markov chains and of probability measures]\label{def:f-divergence}
	
	Let $f:\mathbb{R}^+ \to \mathbb{R}$ be a convex function with $f(1)=0$. 
	For given $\pi \in \mathcal{P}(\mathcal{X})$ and transition matrices $M, L\in \mathcal{L}$, we define the $f$-divergence from $L$ to $M$ with respect to $\pi$ as 
	\begin{equation}\label{f divergence for matrices}
		D_f^{\pi}(M\|L):=\sum_{x\in \mathcal{X}}\pi(x)\sum_{y\in\mathcal{X}}L(x,y)f\left(\dfrac{M(x,y)}{L(x,y)}\right).
	\end{equation}
	For two probability measures $\mu,\nu \in \mathcal{P}(\Omega)$ with $\Omega$ finite, the $f$-divergence from $\nu$ to $\mu$ is defined to be 
	\begin{equation}\label{f divergence for probability measure}
		\widetilde D_f(\mu\|\nu):= \sum_{x \in \Omega} \nu(x)f\left(\frac{\mu(x)}{\nu(x)}\right),
	\end{equation}
	where we apply the usual convention that $0f(\frac{0}{0}):=0$ and $0f(\frac{a}{0}):= a f^{\prime}(+\infty)$ with $f^{\prime}(+\infty) := \lim_{x \to 0^+} x f(\frac{1}{x})$ for $a > 0$ in the two definitions above. We also adopt the convention that $0 \cdot \infty := 0$.
\end{definition}

In the special case of taking $f(t) = t \ln t$, we recover the KL divergence. In this case we shall write $D_{KL}^{\pi}$ and $\widetilde D_{KL}$ respectively. In particular, when $M,L$ are assumed to be $\pi$-stationary, we write $D(M\|L) := D_{KL}^{\pi}(M\| L)$ which can be interpreted as the KL divergence rate from $L$ to $M$, see \cite{R04}. Notably $f$-divergences of Markov chains have also been proposed for estimating the transition matrix from samples generated from Markov chains, for instance in \cite{HOP18}.

In the sequel, a majority of our focus is devoted to state space that takes on a product form, that is, $\mathcal{X} = \mathcal{X}^{(1)} \times \ldots \times \mathcal{X}^{(d)} =: \bigtimes_{i=1}^d \mathcal{X}^{(i)}$ for $d \in \mathbb{N}$. A transition matrix $P \in \mathcal{L}(\bigtimes_{i=1}^d \mathcal{X}^{(i)})$ is said to be of product form if there exists $M_i \in \mathcal{L}(\mathcal{X}^{(i)})$ for $i \in \llbracket d \rrbracket$ such that $P$ can be expressed as a tensor product of the form
$$P = \otimes_{i=1}^d M_i.$$ 
This notion of product chain has appeared in \cite[Exercise $12.7$]{levin2017markov} and differs from another slightly different ``product-type chain" in \cite{Mathe1998} or ``product chain" introduced in \cite[Section $12.4, 20.4$]{levin2017markov}. Analogously, a probability mass $\mu \in \mathcal{P}(\bigtimes_{i=1}^d \Omega^{(i)})$ is said to be of product form if there exists $\nu_i \in \mathcal{P}(\Omega^{(i)})$ for $i \in \llbracket d \rrbracket$ such that 
$$\mu = \otimes_{i=1}^d \nu_i.$$


\begin{remark}[On mutual information and interaction information]
	We remark that there exists a body of literature on various generalizations of mutual information to the case of $d > 2$ random variables, see for example \cite{GK17} and the references therein.
\end{remark}

Observe that in the special case when $\pi = \delta_x$ with $x = (x^1,\ldots,x^d)$, the Dirac point mass at $x$, we have
$D_f^{\pi}(M \| \otimes_{i=1}^d L_i) = \widetilde D_f(M((x^1,\ldots,x^d),\cdot) \| (\otimes_{i=1}^d L_i)((x^1,\ldots,x^d),\cdot))$. As such, we can interpret $D_f^{\pi}$ as a generalization of $\widetilde D_f$. We also note that $D_f^{\pi}$ can be written as a $\pi$-weighted average of $\widetilde D_f$, since we have
\begin{align*}
	D_f^{\pi}(M \| \otimes_{i=1}^d L_i) &= \sum_{x \in \mathcal{X}} \pi(x) \widetilde D_f(M(x,\cdot) \| (\otimes_{i=1}^d L_i)(x,\cdot)).
\end{align*}

Our next proposition summarizes some fundamental yet useful properties of $D_f^{\pi}$ from a product transition matrix to a given $M$:
\begin{proposition}\label{prop:Ifproperty}
	Denote the state space to be $\mathcal{X} = \bigtimes_{i=1}^d \mathcal{X}^{(i)}$. Let $\pi \in \mathcal{P}(\mathcal{X})$, $M \in \mathcal{L}(\mathcal{X})$ and $M_i, L_i \in \mathcal{L}(\mathcal{X}^{(i)})$ for $i \in \llbracket d \rrbracket$.
	\begin{enumerate}
		\item(Non-negativity)\label{it:Ifnonnegative} 
		\begin{align*}
			D_f^{\pi}(M \| \otimes_{i=1}^d L_i)\geq 0.
		\end{align*}
		Suppose that $\pi$ is a positive probability mass, that is, $\min_{x \in \mathcal{X}} \pi(x) > 0$. Then the equality holds if and only if $M = \otimes_{i=1}^d L_i$.
		
		\item(Convexity)\label{it:Ifconvex}
		Fix $ L_i \in \mathcal{L}(\mathcal{X}^{(i)})$ for $i \in \llbracket d \rrbracket$. The mapping 
		$$\mathcal{L}(\mathcal{X}) \ni M \mapsto D_f^{\pi}(M \| \otimes_{i=1}^d L_i)$$
		is convex in $M$.
		
		\item(Chain rule of KL divergence)\label{it:Ifchain} Let $\pi = \otimes_{i=1}^d \pi^{(i)}$ be a product distribution with $\pi^{(i)} \in \mathcal{P}(\mathcal{X}^{(i)})$. Then we have
		\begin{align*}
			D_{KL}^{\pi}(\otimes_{i=1}^d M_i \| \otimes_{i=1}^d L_i) = \sum_{i=1}^d D_{KL}^{\pi^{(i)}}(M_i \| L_i),
		\end{align*}
		where each $D_{KL}^{\pi^{(i)}}(M_i \| L_i)$ is weighted by $\pi^{(i)}$ for $i \in \llbracket d \rrbracket$.
		
		\item(Bounds of squared Hellinger distance)\label{it:Hellinger} Let $f(t) = (\sqrt{t}-1)^2$ that generates the squared Hellinger distance and $\pi = \otimes_{i=1}^d \pi^{(i)}$ be a product distribution with $\pi^{(i)} \in \mathcal{P}(\mathcal{X}^{(i)})$. We have
		\begin{align*}
			\max_{i \in \llbracket d \rrbracket} D^{\pi^{(i)}}_f (M_i \| L_i) \leq D^{\pi}_f(\otimes_{i=1}^d M_i \| \otimes_{i=1}^d L_i) \leq \sum_{i=1}^d D^{\pi^{(i)}}_f (M_i \| L_i).
		\end{align*}
		
		\item(Bisection property)\label{it:Ifbisection} Suppose that $M$ is $\pi$-stationary and $L_i$ is $\pi^{(i)}$-stationary, where $\pi = \otimes_{i=1}^d \pi^{(i)}$ is a product distribution and $\pi^{(i)} \in \mathcal{P}(\mathcal{X}^{(i)})$ for $i \in \llbracket d \rrbracket$. Then we have
		\begin{align*}
			D_f^{\pi}(M \| \otimes_{i=1}^d L_i) = D_f^{\pi}(M^* \| \otimes_{i=1}^d L^{*}_i).
		\end{align*}
		In particular, if $L_i \in \mathcal{L}(\pi^{(i)})$, then the above leads to 
		\begin{align*}
			D_f^{\pi}(M \| \otimes_{i=1}^d L_i) = D_f^{\pi}(M^* \| \otimes_{i=1}^d L_i).
		\end{align*}
	\end{enumerate}
\end{proposition}

\begin{remark}
	The Hellinger distance is commonly used to assess the convergence to equilibrium of product Markov chains, see for example \cite{CK18,levin2017markov}.
\end{remark}

\begin{remark}
	Note that Proposition \ref{prop:Ifproperty} item \eqref{it:Ifnonnegative}, \eqref{it:Ifconvex} and \eqref{it:Ifbisection} also hold when the second argument is not a product transition matrix.
\end{remark}

\begin{proof}
	For brevity, throughout this proof we write $x = (x^1,\ldots,x^d)$ and $y=(y^1,\ldots,y^d)$. We first prove item \eqref{it:Ifnonnegative}. Since $D_f^{\pi}$ is a $f$-divergence from $\otimes_{i=1}^d L_i$ to $M$ with respect to $\pi$, it is non-negative according to \cite{WC23}. Since $\pi$ is a positive probability mass, equality holds if and only if for all $x \in \mathcal{X}$ we have $\widetilde D_f(M(x,\cdot) ; (\otimes_{i=1}^d L_i)(x,\cdot)) = 0$ if and only if $M = \otimes_{i=1}^d L_i$ (see for instance \cite{polyanskiy2022information}).
	
	Next, we prove item \eqref{it:Ifconvex}. We see that
	\begin{align*}
		D_f^{\pi}(M \| \otimes_{i=1}^d L_i) &= \sum_{x \in \mathcal{X}} \pi(x) \widetilde D_f(M(x,\cdot) \| (\otimes_{i=1}^d L_i)(x,\cdot)).
	\end{align*}
	Since $M(x,\cdot) \mapsto \widetilde D_f(M(x,\cdot) \| (\otimes_{i=1}^d L_i)(x,\cdot))$ is convex and $D_f^{\pi}(M \| \otimes_{i=1}^d L_i)$ is a $\pi$-weighted sum of convex functions, it is convex.
	
	We proceed to prove item \eqref{it:Ifchain}. First, we consider the case where $L_i(x^i,y^i) > 0$ for all $i$ or $M_i(x^i,y^i) = 0$ whenever $L_i(x^i,y^i) = 0$, i.e. $M_i(x^i,\cdot) \ll L_i(x^i,\cdot)$ for all $i$. We see that
	\begin{align*}
		D_{KL}^\pi(\otimes_{i=1}^d M_i \| \otimes_{i=1}^d L_i) 
		&= \sum_{x,y \in \mathcal{X}} \pi(x) \prod_{j=1}^d M_j(x^j,y^j) \sum_{i=1}^d  \ln \left(\dfrac{M_i(x^i,y^i)}{L_i(x^i,y^i)}\right)\\
		&= \sum_{i=1}^d \sum_{x,y \in \mathcal{X}} \pi(x) \prod_{j=1}^d M_j(x^j,y^j)   \ln \left(\dfrac{M_i(x^i,y^i)}{L_i(x^i,y^i)}\right)\\
		&= \sum_{i=1}^d \sum_{x^i,y^i \in \mathcal{X}^{(i)}} \pi^{(i)}(x^i) M_i(x^i,y^i)   \ln \left(\dfrac{M_i(x^i,y^i)}{L_i(x^i,y^i)}\right) \\
		&= \sum_{i=1}^d D_{KL}^{\pi^{(i)}}(M_i \| L_i).
	\end{align*}
	Next, we consider the case where there exists $i$ such that $M_i(x^i,y^i) > 0$ yet $L_i(x^i,y^i) = 0$. Since we take $f(t) = t \ln t$ in KL divergence, we have $f^{\prime}(\infty) = \infty$, and hence both sides are $\infty$ in item \eqref{it:Ifchain}.
	
	Now, we prove item \eqref{it:Hellinger}. For the upper bound, we note that
	\begin{align*}
		D^{\pi}_f(\otimes_{i=1}^d M_i \| \otimes_{i=1}^d L_i) &= \sum_{x \in \mathcal{X}} \pi(x) \widetilde D_f( (\otimes_{i=1}^d M_i)(x,\cdot) \|  (\otimes_{i=1}^d L_i)(x,\cdot)) \\
		&\leq \sum_{x \in \mathcal{X}} \pi(x) \sum_{i=1}^d \widetilde D_f( M_i(x^i,\cdot) \|  L_i(x^i,\cdot)) \\ 
		&= \sum_{i=1}^d \sum_{x^i \in \mathcal{X}^{(i)}} \pi^{(i)}(x^i)  \widetilde D_f( M_i(x^i,\cdot) \|  L_i(x^i,\cdot)) \\
		&= \sum_{i=1}^d D^{\pi^{(i)}}_f (M_i \| L_i),
	\end{align*}
	where we apply \cite[Lemma $20.9$]{levin2017markov} in the inequality. On the other hand, the lower bound can be seen via
	\begin{align*}
		D^{\pi}_f(\otimes_{i=1}^d M_i \| \otimes_{i=1}^d L_i) &= \sum_{x \in \mathcal{X}} \pi(x) \widetilde D_f( (\otimes_{i=1}^d M_i)(x,\cdot) \|  (\otimes_{i=1}^d L_i)(x,\cdot)) \\
		&\geq \sum_{x \in \mathcal{X}} \pi(x) \max_{i \in \llbracket d \rrbracket} \widetilde D_f( M_i(x^i,\cdot) \| L_i(x^i,\cdot)) \\ 
		&\geq \sum_{x \in \mathcal{X}} \pi(x) \widetilde D_f( M_i(x^i,\cdot) \|  L_i(x^i,\cdot)) \\ 
		&= D^{\pi^{(i)}}_f (M_i \| L_i),
	\end{align*}
	where the first inequality follows from \cite[Proposition $2.3$]{CK18}. The desired result follows by taking maximum over $i \in \llbracket d \rrbracket$.
	
	Finally, for item \eqref{it:Ifbisection}, we note that the proof is similar to \cite[Section IIIA]{CW24} and is therefore omitted.
\end{proof}

\subsection{Distance to independence and the closest product chain}\label{subsec:distanceind}

Given a Markov chain with transition matrix $P$ on a finite product state space, how far away is it from being a product chain? In other words, what is the (information-theoretic) ``distance", in a broad sense, to independence? One possible way to measure this distance is by means of projection. Throughout this section, unless otherwise specified we shall consider a product state space of the form $\mathcal{X} = \bigtimes_{i=1}^d \mathcal{X}^{(i)}$. We now define

\begin{definition}[Distance to independence of $P$ with respect to $D_{f}^{\pi}$]
	Given $P \in \mathcal{L}(\mathcal{X})$, we define the distance to independence of $P$ with respect to $D_{f}^{\pi}$ to be
	\begin{align}\label{def:IfP}
		\mathbb{I}_f^{\pi}(P) := \min_{L_i \in \mathcal{L}(\mathcal{X}^{(i)}),\, \forall i \in \llbracket d \rrbracket} D_f^{\pi}(P \| \otimes_{i=1}^d L_i).
	\end{align}
	In particular, when we take $f(t) = t \ln t$ that generates the KL divergence, we write $\mathbb{I}^{\pi}(P) := \mathbb{I}_f^{\pi}(P)$ in this case. If $P$ is $\pi$-stationary, then we also write $\mathbb{I}(P) = \mathbb{I}^{\pi}(P)$.
\end{definition}

Note that since $L \mapsto D_f^{\pi}(M \| L)$ is continuous and the set $\bigtimes_{i=1}^d \mathcal{L}(\mathcal{X}^{(i)})$ is compact, the minimization problem \eqref{def:IfP} is always attained. The next result states that this distance to independence of $P$ is zero if and only if $P$ is a product chain under suitable assumptions on $\pi$. This is analogous to the property that if two random variables are independent, then their correlation is zero.

\begin{proposition}
	Assume the same setting as in Proposition \ref{prop:Ifproperty}. Let $\pi \in \mathcal{P}(\mathcal{X})$ and $P \in \mathcal{L}(\mathcal{X})$. We have
	\begin{align*}
		\mathbb{I}_f^{\pi}(P) \geq 0.
	\end{align*}
	Suppose that $\pi$ is a positive probability mass, that is, $\min_{x \in \mathcal{X}} \pi(x) > 0$. Then the equality holds if and only if $P$ is a product chain.
\end{proposition}

\begin{proof}
	The non-negativity is clear from Proposition \ref{prop:Ifproperty}. If $P$ is a product chain, then clearly $\mathbb{I}^\pi_f(P) = 0$. For the other direction, if $\mathbb{I}^\pi_f(P) = 0$, then $D_f^{\pi}(P \| \otimes_{i=1}^d L_i) = 0$ for some $L_i$ since the minimization in \eqref{def:IfP} is exactly attained. By Proposition \ref{prop:Ifproperty}, $P = \otimes_{i=1}^d L_i$.
\end{proof}

Let us now recall the notion of edge measure of a Markov chain (see e.g. \cite[equation (7.5)]{levin2017markov}). Let $\pi \in \mathcal{P}(\mathcal{X})$ and $P \in \mathcal{L}(\mathcal{X})$. The edge measure $\pi \boxtimes P \in \mathcal{P}(\mathcal{X} \times \mathcal{X})$ is defined to be, for $x,y \in \mathcal{X}$, $$(\pi \boxtimes P) (x,y) := \pi(x) P(x,y).$$ Note that this edge measure encodes the probability of observing a consecutive pair generated from the chain with transition matrix $P$ starting from the distribution $\pi$.

We proceed to define the $i$th marginal transition matrix of $P$ with respect to $\pi$:

\begin{definition}[$P^{(i)}_{\pi}$: the $i$th marginal transition matrix of $P$ with respect to $\pi$]\label{def:Pi}
	Assume the same setting as in Proposition \ref{prop:Ifproperty}. Let $\pi \in \mathcal{P}(\mathcal{X})$ be a positive probability mass, $P \in \mathcal{L}(\mathcal{X})$ and $i \in \llbracket d \rrbracket$. For any $(x^i,y^i) \in \mathcal{X}^{(i)} \times \mathcal{X}^{(i)}$, we define
	\begin{align*}
		P^{(i)}_\pi(x^i,y^i) &:= \dfrac{\sum_{j=1;~j \neq i}^d \sum_{(x^j,y^j) \in \mathcal{X}^{(j)} \times \mathcal{X}^{(j)}} \pi(x^1,\ldots,x^d) P((x^1,\ldots,x^d),(y^1,\ldots,y^d))}{\sum_{j=1;~j \neq i}^d\sum_{x^j \in \mathcal{X}^{(j)}}\pi(x^1,\ldots,x^d)} \\
		&= \dfrac{\sum_{j=1;~j \neq i}^d \sum_{(x^j,y^j) \in \mathcal{X}^{(j)} \times \mathcal{X}^{(j)}} (\pi \boxtimes P)((x^1,\ldots,x^d),(y^1,\ldots,y^d))}{\pi^{(i)}(x^i)},
	\end{align*}
	where $\pi^{(i)}$ is the $i$th marginal probability mass of $\pi$. Note that $P^{(i)}_{\pi} \in \mathcal{L}(\mathcal{X}^{(i)})$. When $P$ is $\pi$-stationary, we omit the subscript and write $P^{(i)}$.
\end{definition}

First, we note that $P^{(i)}_\pi$ can be understood as a special case of the keep-$S$-in transition matrix to be introduced in Definition \ref{def:P-S} below, which is a further special case of various notions of ``projection chains" investigated in \cite{JSTV04,AF02,CLP99}.
Second, we see that if $P$ is $\pi$-stationary, then $P^{(i)}_{\pi}$ is $\pi^{(i)}$-stationary. Similarly, if $P$ is $\pi$-reversible, then $P^{(i)}_{\pi}$ is $\pi^{(i)}$-reversible. As a result, $\otimes_{i=1}^d P^{(i)}_{\pi}$ is thus $\otimes_{i=1}^d \pi^{(i)}$-stationary, and hence in general $\otimes_{i=1}^d P^{(i)}_{\pi}$ is not $\pi$-stationary. In the case where $\pi = \otimes_{i=1}^{d} \pi^{(i)}$ is a product stationary distribution, then $\otimes_{i=1}^d P^{(i)}_{\pi}$ is $\pi$-stationary. A generalization of the above discussions can be found in Proposition \ref{prop:PSergodic} below.


For a concrete example of $P_{\pi}^{(i)}$, we point to the example of the swapping algorithm where we calculate explicitly the marginal transition matrices in Section \ref{sec:projectionsamplers}.

In our next result, we state that under the KL divergence and positivity of $\pi$, a Pythagorean identity holds and it implies that the product chain with transition matrix $\otimes_{i=1}^d P^{(i)}_\pi$ is the unique closest product chain to $P$:

\begin{theorem}\label{thm:closest}
	Assume the same setting as in Proposition \ref{prop:Ifproperty}. Let $\pi \in \mathcal{P}(\mathcal{X})$, $P \in \mathcal{L}(\mathcal{X})$ and $L_i \in \mathcal{L}(\mathcal{X}^{(i)})$ for $i \in \llbracket d \rrbracket$. We then have
	\begin{enumerate}
		\item(Pythagorean identity of $D^{\pi}_{KL}$)\label{it:Py}
		Let $\pi$ be a positive probability mass. We have
		\begin{align*}
			D^{\pi}_{KL}(P \| \otimes_{i=1}^d L_i) 
			&= D_{KL}^{\pi}(P \| \otimes_{i=1}^d P^{(i)}_{\pi}) + D_{KL}^{\pi}(\otimes_{i=1}^d P^{(i)}_\pi \| \otimes_{i=1}^d L_i) \\
			&= D^{\pi}_{KL}(P \| \otimes_{i=1}^d P^{(i)}_{\pi}) + \sum_{i=1}^d D^{\pi^{(i)}}_{KL}(P^{(i)}_{\pi} \| L_i),
		\end{align*}
		where each $D_{KL}^{\pi^{(i)}}(P^{(i)}_\pi \| L_i)$ is weighted by $\pi^{(i)}$, the $i$th marginal distribution of $\pi$. In particular, the unique minimizer that solves \eqref{def:IfP} is given by $\otimes_{i=1}^d P^{(i)}_\pi$, that is,
		\begin{align*}
			\mathbb{I}^{\pi}(P) = D^{\pi}_{KL}(P \| \otimes_{i=1}^d P^{(i)}_\pi).
		\end{align*}
		
		\item(Bisection property)\label{it:bisection} Suppose that $P$ is $\pi$-stationary, where $\pi = \otimes_{i=1}^d \pi^{(i)}$ is a product distribution and $\pi^{(i)} \in \mathcal{P}(\mathcal{X}^{(i)})$ for $i \in \llbracket d \rrbracket$. We have
		\begin{align*}
			\mathbb{I}^{\pi}(P) = \mathbb{I}^{\pi}(P^*).
		\end{align*}
		In other words, the distance to independence of $P$ with respect to $D_{KL}^{\pi}$ and that of its time-reversal $P^*$ is the same.
	\end{enumerate}
	
\end{theorem}

\begin{remark}[Distance to independence as KL divergence rate from the closest product chain of $P$ to $P$]
	Suppose that $P$ is $\pi$-stationary, and $(\mathbf{X}_n)_{n \in \mathbb{N}} = (X^1_n,X^2_n,\ldots,X^d_n)_{n \in \mathbb{N}}$ (resp.~$(\mathbf{Y}_n)_{n \in \mathbb{N}} = (Y^1_n,Y^2_n,\ldots,Y^d_n)_{n \in \mathbb{N}}$) is the discrete-time homogeneous Markov chain with transition matrix $P$ (resp.~$\otimes_{i=1}^d P^{(i)}_\pi$). Then, the distance to independence of $P$ can be written as
	\begin{align*}
		\mathbb{I}(P) = D(P \| \otimes_{i=1}^d P^{(i)}_\pi) = \lim_{n \to \infty} \dfrac{1}{n} \widetilde{D}_{KL}(\mu_n \| \nu_n),
	\end{align*}
	where $\mathbf{X}_n \sim \mu_n, \mathbf{Y}_n \sim \nu_n$ and the right hand side can be interpreted as the KL divergence rate from the closest product chain to $P$. The rightmost expression in the equality above is also known as the mutual information rate \cite{CBM18,HGG98}.
\end{remark}

\begin{proof}
	We first prove item \eqref{it:Py}. We first consider the case where there exists $x,y \in \mathcal{X}$ such that $P(x,y) > 0$ yet $(\otimes_{i=1}^d L_i)(x,y) = 0$. This implies that $P^{(i)}_\pi(x^i,y^i) > 0$ for all $i$. Thus, both $D_{KL}^{\pi}(P \| \otimes_{i=1}^d L_i) = D_{KL}^{\pi^{(i)}}(P^{(i)}_\pi \| L_i) = +\infty$ and the identity holds. Next, we consider the case with $P(x,\cdot) \ll (\otimes_{i=1}^d L^{(i)})(x,\cdot)$ for all $x$. We see that
	\begin{align*}
		D_{KL}^{\pi}(P \| \otimes_{i=1}^d L_i) &= D_{KL}^{\pi}(P \| \otimes_{i=1}^d P^{(i)}_\pi) + \sum_{x,y} \pi(x) P(x,y) \ln\left(\dfrac{(\otimes_{i=1}^d P^{(i)}_\pi)(x,y)}{(\otimes_{i=1}^d L_i)(x,y)}\right) \\
		&= D_{KL}^{\pi}(P \| \otimes_{i=1}^d P^{(i)}_\pi) + \sum_{i=1}^d \sum_{(x^i,y^i)}  \pi^{(i)}(x^i) \sum_{(x^j,y^j);~j\neq i} \dfrac{\pi(x) P(x,y)}{\pi^{(i)}(x^i)} \ln\left(\dfrac{P^{(i)}_\pi(x^i,y^i)}{L_i(x^i,y^i)}\right) \\
		&= D_{KL}^{\pi}(P \| \otimes_{i=1}^d P^{(i)}_\pi) + \sum_{i=1}^d \sum_{(x^i,y^i)}  \pi^{(i)}(x^i) P^{(i)}_\pi(x^i,y^i) \ln\left(\dfrac{P^{(i)}_\pi(x^i,y^i)}{L_i(x^i,y^i)}\right) \\
		&= D_{KL}^{\pi}(P \| \otimes_{i=1}^d P^{(i)}_\pi) + \sum_{i=1}^d D_{KL}^{\pi^{(i)}}(P^{(i)}_\pi \| L_i).
	\end{align*}
	In view of Proposition \ref{prop:Ifproperty}, 
	$$D_{KL}^{\pi}(P \| \otimes_{i=1}^d L_i) \geq D_{KL}^{\pi}(P \| \otimes_{i=1}^d P^{(i)}_\pi)$$
	and equality holds if and only if $D_{KL}^{\pi}(P^{(i)}_\pi \| L_i) = 0$ for all $i$ if and only if $P^{(i)}_\pi = L_i$ for all $i$.
	
	Next, we prove item \eqref{it:bisection}. First we see that $P^{(i)}_\pi$ is $\pi^{(i)}$-stationary with $P^{(i)*}_\pi = P^{*(i)}_\pi$. Thus, by item \eqref{it:Py} and Proposition \ref{prop:Ifproperty} item \eqref{it:Ifbisection}, we have
	\begin{align*}
		\mathbb{I}^{\pi}(P) = D_{KL}^{\pi}(P \| \otimes_{i=1}^d P^{(i)}_\pi) = D_{KL}^{\pi}(P^* \| \otimes_{i=1}^d P^{(i)*}_\pi) = D_{KL}^{\pi}(P^* \| \otimes_{i=1}^d P^{*(i)}_\pi) = \mathbb{I}^{\pi}(P^*).
	\end{align*}
\end{proof}

One possible application of the Pythagorean identity lies in assessing the convergence to equilibrium of $P$. This is in part motivated by \cite[Section $10$]{HKLPSW19} which suggests looking into ``Markov chains with factored transition kernels with a few factors". Suppose that $P$ is ergodic (i.e. irreducible and aperiodic) with a product form stationary distribution $\pi = \otimes_{i=1}^d \pi^{(i)}$. Let $\Pi \in \mathcal{L}(\mathcal{X})$ be the matrix where each row is $\pi$, and $\Pi^{(i)} \in \mathcal{L}(\mathcal{X}^{(i)})$ be a matrix where each row is $\pi^{(i)}$ for all $i$. We thus see that
\begin{align*}
	D_{KL}^{\pi}(P^n \| \Pi) \leq \max_{x \in \mathcal{X}} \widetilde{D}_{KL}(P^n(x,\cdot)\|\pi).
\end{align*}
On the other hand, we can lower bound $D_{KL}^{\pi}(P^n\|\Pi)$ via the KL divergence from $\Pi^{(i)}$ to $P^{n(i)}_\pi$ using the Pythagorean identity in Theorem \ref{thm:closest}:
\begin{align*}
	D_{KL}^{\pi}(P^n \| \Pi) \geq \max_{i \in \llbracket d \rrbracket} D_{KL}^{\pi^{(i)}}(P^{n(i)}_\pi \| \Pi^{(i)}).
\end{align*}

As a result this yields, for $\varepsilon > 0$, 
\begin{align*}
	t_{mix}(\varepsilon) \geq \max_{i \in \llbracket d \rrbracket} t_{mix}^{(i)}(\varepsilon),
\end{align*}
where $t_{mix}(\varepsilon) := \inf\{n \in \mathbb{N};\max_{x \in \mathcal{X}} \widetilde{D}_{KL}(P^n(x,\cdot)\|\pi) < \varepsilon\}$ is the worst-case KL divergence mixing time of $P$ and $t_{mix}^{(i)}(\varepsilon) := \inf\{n \in \mathbb{N}; D_{KL}^{\pi^{(i)}}(P^{n(i)}_\pi \| \Pi^{(i)}) < \varepsilon\}$ is an average-case KL divergence mixing time of the $i$th marginal transition matrix of $P^n$, namely $P^{n(i)}_{\pi}$. The interpretation is that the first time for $P$ to be $\varepsilon$ close to $\pi$ in the sense of KL divergence is at least larger than the worst average-case marginal transition matrix KL divergence mixing time. In Section \ref{subsec:comparisonfunctionalconstants}, we shall compare ergodicity constants, such as the spectral gap and log-Sobolev constant, between $P$ and its information projections.

\subsubsection{The closest product chain with prescribed marginals and a large deviation principle of Markov chains}\label{subsubsec:larged}

Fix $i \in \llbracket d \rrbracket$ and suppose we are prescribed with $L_j \in \mathcal{L}(\mathcal{X}^{(j)})$ for all $j \in \llbracket d \rrbracket$, $j \neq i$ and a transition matrix $P \in \mathcal{L}(\mathcal{X})$. We consider the problem of finding the closest product chain of the form $(\otimes_{j=1}^{i-1} L_j) \otimes L \otimes (\otimes_{j=i+1}^{d} L_j)$. In other words, we are interested in seeking a minimizer of
\begin{align*}
	L^{(i)}_* = L_*^{(i)}(P,L_1,\ldots,L_{i-1},L_{i+1},\ldots,L_d,f,\pi) \in \argmin_{L \in \mathcal{L}(\mathcal{X}^{(i)})} D_f^{\pi}(P \| (\otimes_{j=1}^{i-1} L_j) \otimes L \otimes (\otimes_{j=i+1}^{d} L_j)).
\end{align*}

In view of the previous subsection, in the special case where $L_j = P^{(j)}_\pi$ for all $j \neq i$, the $j$th marginal transition matrix of $P$ as introduced in Definition \ref{def:Pi}, it seems natural to guess that $L_*^{(i)}$ is $P^{(i)}_\pi$. Our next result shows that, depending on the choice of $f$, $L_*^{(i)}$ can in fact be weighted averages of $L_j$ and $P$ in a broad sense. Therefore, the seemingly natural product chain with transition matrix $\otimes_{i=1}^d P^{(i)}_\pi$ is not necessarily the closest product chain under some information divergences.

\begin{theorem}\label{thm:prescribed}
	Fix $i \in \llbracket d \rrbracket$ and suppose we are prescribed with $L_j \in \mathcal{L}(\mathcal{X}^{(j)})$ for all $j \in \llbracket d \rrbracket$, $j \neq i$ and a transition matrix $P \in \mathcal{L}(\mathcal{X})$. Let $\pi$ be a positive probability mass.
	
	\begin{enumerate}
		\item(Reverse KL divergence)\label{it:reverseKL} Let $f(t) = -\ln t$ that generates the reverse KL divergence. The unique $L_*^{(i)}$ is given by, for $x^i,y^i \in \mathcal{X}^{(i)}$,
		\begin{align}\label{eq:reverseKLLi}
			L_*^{(i)}(x^i,y^i) \propto \prod_{x^{(-i)},y^{(-i)}} P(x,y)^{\frac{\pi(x) \mathbf{Z}(x^{(-i)},y^{(-i)})}{Z(x^i,y^i)}},
		\end{align}
		where $x^{(-i)} := (x^1,\ldots,x^{i-1},x^{i+1},\ldots,x^d)$, $\mathbf{Z}(x^{(-i)},y^{(-i)}) := \prod_{j=1;~ j \neq i}^n L_j(x^j,y^j)$ and $Z(x^i,y^i) = \sum_{x^{(-i)},y^{(-i)}} \pi(x) \mathbf{Z}(x^{(-i)},y^{(-i)})$. 
		
		\item($\alpha$-divergence)\label{it:alpha} Let $f(t) = \frac{1}{\alpha - 1}(t^{\alpha}-1)$ that generates the $\alpha$-divergence with $\alpha \in (0,1) \cup (1,\infty)$. The unique $L_*^{(i)}$ is given by, for $x^i,y^i \in \mathcal{X}^{(i)}$,
		\begin{align*}
			L_*^{(i)}(x^i,y^i) \propto \left(\sum_{x^{(-i)},y^{(-i)}} \pi(x) \left( \prod_{j=1;~ j \neq i}^d L_j(x^j,y^j)\right)^{1-\alpha} P(x,y)^{\alpha}\right)^{1/\alpha}.
		\end{align*}
		
		\item(KL divergence)\label{it:KL} Let $f(t) = t \ln t$ that generates the KL divergence. The unique $L_*^{(i)}$ is given by
		\begin{align*}
			L_*^{(i)} = P^{(i)}_\pi.
		\end{align*}
		Note that $L_*^{(i)}$ depends on $P$ and $\pi$ and does not depend on $L_j$ with $j \neq i$.
	\end{enumerate}
\end{theorem}

\begin{proof}
	We first prove item \eqref{it:reverseKL}. We write down
	\begin{align*}
		D_f^{\pi}(P &\| (\otimes_{j=1}^{i-1} L_j) \otimes L \otimes (\otimes_{j=i+1}^{d} L_j)) \\
		&= \sum_{x,y} \pi(x) L(x^i,y^i) \prod_{j=1;~ j \neq i}^d L_j(x^j,y^j) \ln \left(\dfrac{L(x^i,y^i) \prod_{j=1;~ j \neq i}^d L_j(x^j,y^j)}{P(x,y)}\right).
	\end{align*}
	Differentiating the above with respect to $L(x^i,y^i)$ and noting that $z^i$ is chosen such that $L(x^i,z^i) = 1 - \sum_{y^i \in \mathcal{X}^i;~ y^i \neq z^i} L(x^i,y^i)$, we set the derivative to be zero to give
	\begin{align*}
		\sum_{x^{(-i)},y^{(-i)}} \pi(x)\prod_{j=1;~ j \neq i}^d L_j(x^j,y^j) \ln \left(\dfrac{L(x^i,y^i)P(x,(y^1,\ldots,z^i,\ldots,y^d))}{L(x^i,z^i)P(x,y)}\right) = 0.
	\end{align*}
	Using $\mathbf{Z}(x^{(-i)},y^{(-i)}) = \prod_{j=1;~ j \neq i}^d L_j(x^j,y^j)$ and $Z(x^i,y^i) = \sum_{x^{(-i)},y^{(-i)}} \pi(x) \mathbf{Z}(x^{(-i)},y^{(-i)})$, we then see that
	$$L_*^{(i)}(x^i,y^i) \propto \prod_{x^{(-i)},y^{(-i)}} P(x,y)^{\frac{\pi(x) \mathbf{Z}(x^{(-i)},y^{(-i)})}{Z(x^i,y^i)}}.$$
	
	Next, we prove item \eqref{it:alpha}. In this case we have
	\begin{align*}
		D_f^{\pi}(P &\| (\otimes_{j=1}^{i-1} L_j) \otimes L \otimes (\otimes_{j=i+1}^{d} L_j)) \\
		&= \dfrac{1}{\alpha - 1}\sum_{x,y} \pi(x) \left(L(x^i,y^i) \prod_{j=1;~ j \neq i}^d L_j(x^j,y^j)\right)^{1-\alpha} P(x,y)^{\alpha}  - 1.
	\end{align*}
	We then differentiate the above with respect to $L(x^i,y^i)$ and note that $z^i$ satisfies $L(x^i,z^i) = 1 - \sum_{y^i \in \mathcal{X}^i;~ y^i \neq z^i} L(x^i,y^i)$. Setting the derivative to be zero leads to
	$$L_*^{(i)}(x^i,y^i) \propto \left(\sum_{x^{(-i)},y^{(-i)}} \pi(x) \left( \prod_{j=1;~ j \neq i}^d L_j(x^j,y^j)\right)^{1-\alpha} P(x,y)^{\alpha}\right)^{1/\alpha}.$$
	
	Finally, we prove item \eqref{it:KL}. We first note that
	\begin{align*}
		D_f^{\pi}(P &\| (\otimes_{j=1}^{i-1} L_j) \otimes L \otimes (\otimes_{j=i+1}^{d} L_j)) \\
		&= \sum_{x,y} \pi(x) P(x,y) \ln \left(\dfrac{P(x,y)}{L(x^i,y^i) \prod_{j=1;~ j \neq i}^d L_j(x^j,y^j)}\right).
	\end{align*}
	Differentiating the above with respect to $L(x^i,y^i)$ and noting that $z^i$ is chosen such that $L(x^i,z^i) = 1 - \sum_{y^i \in \mathcal{X}^i;~ y^i \neq z^i} L(x^i,y^i)$, we set the derivative to be zero to give
	\begin{align*}
		L_*^{(i)}(x^i,y^i) \propto \sum_{x^{(-i)},y^{(-i)}} \pi(x) P(x,y),
	\end{align*}
	that is, $L_*^{(i)} = P^{(i)}_{\pi}$.
\end{proof}

One application of Theorem \ref{thm:prescribed} lies in the large deviation analysis and Sanov's theorem of Markov chains, in which we apply the results obtained in \cite{FF02}. We refer readers to \cite{N85,V14,DZ10} for related literature on large deviations of Markov chains.

Let $X = (X_n)_{n \in \mathbb{N}_0}$ be the Markov chain with transition matrix $P$. Define the pair empirical measure of $X$ to be
\begin{align}\label{eq:pairempirical}
	E_n := \dfrac{1}{n} \left( \sum_{i=1}^{n-1} \delta_{(X_i,X_{i+1})} + \delta_{(X_n,X_{1})} \right).
\end{align} 

\begin{theorem}[A Sanov's theorem for pair empirical measure of Markov chains]\label{thm:prescribedlarged}
	Fix $i \in \llbracket d \rrbracket$. Let $\pi = \otimes_{l=1}^d \pi^{(l)}$. Suppose we are prescribed with $\pi^{(j)}$-stationary $L_j \in \mathcal{L}(\mathcal{X}^{(j)})$ for all $j \in \llbracket d \rrbracket$, $j \neq i$ and a $\pi$-stationary $P \in \mathcal{L}(\mathcal{X})$. Let $K_i$ be the set
	\begin{align*}
		K_i &= K_i(L_1,\ldots,L_{i-1},L_{i+1},\ldots,L_{d}) \\
		&:= \{ (\otimes_{j=1}^{i-1} L_j) \otimes M \otimes (\otimes_{j=i+1}^{d} L_j);~ \textrm{$\pi^{(i)}$-stationary}\quad M \in \mathcal{L}(\mathcal{X}^{(i)}) \},
	\end{align*}
	and $f(t) = - \ln t$ that generates the reverse KL divergence. We have
	\begin{align*}
		\limsup_{n \to \infty} \dfrac{1}{n} \ln \mathbb{P}(E_n \in K_i) \leq - D_f^{\pi}(P \| (\otimes_{j=1}^{i-1} L_j) \otimes L^{(i)}_* \otimes (\otimes_{j=i+1}^{d} L_j)),
	\end{align*}
	where we recall that $E_n$ is the pair empirical measure as introduced in \eqref{eq:pairempirical} and $L^{(i)}_*$ is given in \eqref{eq:reverseKLLi}. Note that the above result holds without any restriction on the initial distribution of the chain $X$.
\end{theorem}

\begin{proof}
	The plan is to invoke Theorem $1.1$ in \cite{FF02}. Let us first assume that $K$ is a subset of the set of balanced measures. Then by \cite{FF02} and Theorem \ref{thm:prescribed}, these yield
	\begin{align*}
		\limsup_{n \to \infty} \dfrac{1}{n} \ln \mathbb{P}(E_n \in K_i) \leq - \inf_{L \in K} D_f^{\pi}(P \| L) = - D_f^{\pi}(P \| (\otimes_{j=1}^{i-1} L_j) \otimes L^{(i)}_* \otimes (\otimes_{j=i+1}^{d} L_j)).
	\end{align*}
	
	It remains to verify that $K_i$ is a subset of the set of balanced measures, in which we readily see that
	$$(\pi \boxtimes (\otimes_{j=1}^{i-1} L_j) \otimes M \otimes (\otimes_{j=i+1}^{d} L_j)) (\mathcal{X},\cdot) = \pi(\cdot) = (\pi \boxtimes (\otimes_{j=1}^{i-1} L_j) \otimes M \otimes (\otimes_{j=i+1}^{d} L_j)) (\cdot,\mathcal{X}).$$
	
\end{proof}

\subsubsection{A coordinate descent algorithm for finding the closest product chain}

Let us recall that in Theorem \ref{thm:closest}, we have shown $\otimes_{i=1}^d P^{(i)}_\pi$ is the unique closest product chain to a given $P$. In other choices of $f$-divergences such as the reverse KL divergence or the $\alpha$-divergence, we did not manage to derive a closed form formula for the closest product chain. In these cases, if we have the closed form of the closest product chain with prescribed marginals as in Theorem \ref{thm:prescribed}, we can derive a coordinate descent algorithm to find approximately the closest product chain.

Suppose that the algorithm is initiated with $L^{0}_i \in \mathcal{L}(\mathcal{X}^{(i)})$ for $i \in \llbracket d \rrbracket$. At iteration $l \in \mathbb{N}$ and for each $i \in \llbracket d \rrbracket$, we compute that
\begin{align*}
	L^{l}_i = L^l_{i}(P,L^{l}_1,\ldots,L^{l}_{i-1},L^{l-1}_{i+1},\ldots,L^{l-1}_{d},f,\pi) \in \argmin_{L \in \mathcal{L}(\mathcal{X}^{(i)})} D_f^{\pi}(P \| (\otimes_{j=1}^{i-1} L^{l}_j) \otimes L \otimes (\otimes_{j=i+1}^{d} L^{l-1}_{j})).
\end{align*}
In the case of reverse KL divergence or $\alpha$-divergence, these are computed in Theorem \ref{thm:prescribed}. The sequence $(\otimes_{i=1}^d L^{l}_i)_{l \in \mathbb{N}}$ satisfies
$$D_f^{\pi}(P \| \otimes_{i=1}^{d} L^{l}_i) \geq D_f^{\pi}(P \| \otimes_{i=1}^{d} L^{l+1}_{i} )$$ 
for $l \geq 0$.

\subsection{Leave-$S$-out transition matrices and Han-Shearer type inequalities for KL divergence of Markov chains}\label{subsec:leaveoneout}


In this section, we consider marginalizing a subset $S \subseteq \llbracket d \rrbracket$ of the $d$ coordinates of a multivariate transition matrix $P$ on a product state space $\mathcal{X} = \bigtimes_{j=1}^d \mathcal{X}^{(j)}$. In doing so, we introduce the leave-$S$-out and keep-$S$-in transition matrices, as well as deriving Han-Shearer type inequalities for Markov chains.

The leave-$S$-out state space is defined to be $\mathcal{X}^{(-S)} := \bigtimes_{j=1;j \notin S}^d \mathcal{X}^{(j)}$, while the keep-$S$-in state space is defined to be $\mathcal{X}^{(S)} := \bigtimes_{j=1;j \in S}^d \mathcal{X}^{(j)}$. Let $x^j \in \mathcal{X}^{(j)}$ for all $j \in \llbracket d \rrbracket$, and we write $x^{(-S)} := (x^{j})_{j \notin S} \in \mathcal{X}^{(-S)}$, $x^{(S)} := (x^{j})_{j \in S} \in \mathcal{X}^{(S)}$. Let $\mu,\otimes_{j=1}^d \nu_j \in \mathcal{P}(\mathcal{X})$. The leave-$S$-out distribution of $\mu$ is given by
$$\mu^{(-S)}(x^{(-S)}) := \textcolor{black}{\sum_{x^i \in \mathcal{X}^{(i)} \, \textrm{for each}\, i \in S}} \mu(x^1,\ldots,x^d).$$
In particular, this yields
$$(\otimes_{j=1}^d \nu_j)^{(-S)} = \otimes_{j=1;~j \notin S}^d \nu_j.$$
The keep-$S$-in distribution of $\mu$ is the leave-$\llbracket d \rrbracket \backslash S$-out distribution of $\mu$, that is,
$$\mu^{(S)}(x^{(S)}) := \mu^{(-\llbracket d \rrbracket \backslash S)}(x^{(-\llbracket d \rrbracket \backslash S)}) = \textcolor{black}{\sum_{x^i \in \mathcal{X}^{(i)} \, \textrm{for each}\, i \notin S}} \mu(x^1,\ldots,x^d).$$
This leads to
$$(\otimes_{j=1}^d \nu_j)^{(S)} = \otimes_{j=1;~j \in S}^d \nu_j.$$
In the special case of a singleton $j \in \llbracket d \rrbracket$, we write that
\begin{align*}
	\mathcal{X}^{(-j)} &= \mathcal{X}^{(-\{j\})}, \quad \mathcal{X}^{(j)} = \mathcal{X}^{(\{j\})}, \\
	\mu^{(-j)} &= \mu^{(-\{j\})}, \quad \mu^{(j)} = \mu^{(\{j\})}.
\end{align*}

We remark that these leave-one-out and more generally leave-$S$-out distributions are widely used in forming the jackknife estimator in mathematical statistics and cross-validation in the training of machine learning algorithms.

Consider a sequence $(S_i)_{i=1}^n$ with $S_i \subseteq \llbracket d \rrbracket$, where each $j \in \llbracket d \rrbracket$ belongs to at least $r$ of $S_i$. The Shearer's lemma of entropy \cite[Theorem $1.8$]{polyanskiy2022information} is given by
\begin{align}\label{eq:Shearerentropy}
	H(X_1,\ldots,X_d) \leq \dfrac{1}{r} \sum_{i=1}^n H((X_l)_{l \in S_i}),
\end{align}
where $(X_l)_{l \in S_i} \sim \mu^{(S_i)}$ and $H((X_l)_{l \in S_i}) := - \sum_{x^{(S_i)}} \mu^{(S_i)}(x^{(S_i)}) \ln  \mu^{(S_i)}(x^{(S_i)})$ is the entropy of $\mu^{(S_i)}$, while the KL divergence version of the Shearer's lemma \cite[Corollary $2.8$]{GLSS15} is stated as, for $\mu, \nu = \otimes_{j=1}^d \nu_j \in \mathcal{P}(\mathcal{X})$,
\begin{align}\label{eq:Shearerpm}
	\widetilde{D}_{KL}(\mu \| \otimes_{j=1}^d \nu_j) \geq \dfrac{1}{r} \sum_{i=1}^n \widetilde{D}_{KL}(\mu^{(S_i)} \| \nu^{(S_i)}).
\end{align}
In the special case of taking $S_i = \llbracket d \rrbracket \backslash \{i\}$ and $n = d$ so that $r = d - 1$, we recover the Han's inequality for KL divergence between discrete probability masses \cite[Theorem $4.9$]{BLM13}:
\begin{align}\label{eq:Hanspm}
	\widetilde{D}_{KL}(\mu \| \otimes_{j=1}^d \nu_j) \geq \dfrac{1}{d-1} \sum_{i=1}^n \widetilde{D}_{KL}(\mu^{(-i)} \| \nu^{(-i)}).
\end{align}

Next, we introduce the leave-$S$-out transition matrix:

\begin{definition}[$P^{(-S)}_\pi$ and $P^{(S)}_\pi$: the leave-$S$-out and keep-$S$-in transition matrix of $P$ with respect to $\pi$]\label{def:P-S}
	Let $\pi \in \mathcal{P}(\mathcal{X})$ be a positive probability mass, $P \in \mathcal{L}(\mathcal{X})$ and $S \subseteq \llbracket d \rrbracket$. For any $(x^{(-S)},y^{(-S)}) \in \mathcal{X}^{(-S)} \times \mathcal{X}^{(-S)}$, we define
	\begin{align*}
		P^{(-S)}_\pi(x^{(-S)},y^{(-S)}) &:= \dfrac{\sum_{(x^{(S)},y^{(S)}) \in \mathcal{X}^{(S)} \times \mathcal{X}^{(S)}} \pi(x^1,\ldots,x^d) P((x^1,\ldots,x^d),(y^1,\ldots,y^d))}{\sum_{x^{(S)} \in \mathcal{X}^{(S)}}\pi(x^1,\ldots,x^d)} \\
		&= \dfrac{(\pi \boxtimes P)^{(-S)}(x^{(-S)},y^{(-S)})}{\pi^{(-S)}(x^{(-S)})}.
	\end{align*}
	Note that $P^{(-S)}_\pi \in \mathcal{L}(\mathcal{X}^{(-S)})$. The keep-$S$-in transition matrix of $P$ with respect to $\pi$ is
	$$P^{(S)}_\pi := P^{(-\llbracket d \rrbracket \backslash S)}_\pi \in \mathcal{L}(\mathcal{X}^{(S)}).$$
	In the special case of $S = \{i\}$ for $i \in \llbracket d \rrbracket$, we write
	\begin{align*}
		P^{(-i)}_\pi = P^{(-\{i\})}_\pi, \quad P^{(i)}_\pi = P^{(\{i\})}_\pi,
	\end{align*}
	and call these to be respectively the leave-$i$-out and keep-$i$-in transition matrix of $P$ with respect to $\pi$. When $P$ is $\pi$-stationary, we omit the subscript and write $P^{(-S)},P^{(S)}$.
\end{definition}

\begin{remark}[$P^{(-S)}_\pi$ and $P^{(S)}_\pi$ as conditional expectations and a simulation procedure]\label{rk:condite}
	In this remark, we show that $P^{(-S)}_\pi$ can be understood as conditional expectations. Precisely, let $\pi(\cdot|x^{(-S)})$ denote the conditional probability mass of $\pi$ given the coordinates $x^{(-S)} \in \mathcal{X}^{(-S)}$, where for all $x^{(S)} \in \mathcal{X}^{(S)}$, we have 
	$$\pi(x^{(S)}|x^{(-S)}) = \dfrac{\pi(x^1,\ldots,x^d)}{\pi^{(-S)}(x^{(-S)})}.$$
	In view of Definition \ref{def:P-S}, we arrive at, for any  $(x^{(-S)},y^{(-S)}) \in \mathcal{X}^{(-S)} \times \mathcal{X}^{(-S)}$,
	\begin{align*}
		P^{(-S)}_\pi(x^{(-S)},y^{(-S)}) &= \sum_{(x^{(S)},y^{(S)}) \in \mathcal{X}^{(S)} \times \mathcal{X}^{(S)}} \pi(x^{(S)}|x^{(-S)}) P((x^1,\ldots,x^d),(y^1,\ldots,y^d)).
	\end{align*}
	Viewing these projections as conditional expectations is therefore in line with conditional expectations in the drift term of independent projections in \cite{Lacker23}.
	In addition, this observation allows us to simulate one step of the projection chain associated with $P^{(-S)}_\pi$. Suppose that $\pi(\cdot|x^{(-S)})$ can be sampled from. Starting from the initial state $x^{(-S)}$, we first draw a random $x^{(S)}$ according to $\pi(\cdot|x^{(-S)})$, followed by one step of $P$ from $(x^{(S)},x^{(-S)})$ to $(y^{(S)},y^{(-S)})$. This is applied in Section \ref{sec:projectionsamplers} to design a projection sampler for the swapping algorithm.
\end{remark}

	We note that when $P$ is ergodic and admits $\pi$ as stationary distribution, the keep-$S$-in transition matrix $P^{(S)}$ can be viewed as a special case of the ``projection chain" of $P$ in \cite{PS17,MR02,JSTV04}, or as an ``induced chain" of $P$ in \cite[Section $4.6$]{AF02}. The latter is also known as a ``collapsed chain" in \cite{CLP99}, which can be understood as the opposite of the lifting procedure in MCMC.
	
	Let $\Omega_{x^{(S)}} := \{x^{(S)}\} \times \mathcal{X}^{(-S)}$. We then see that $(\Omega_{x^{(S)}})_{x^{(S)} \in \mathcal{X}^{(S)}}$ is a partition of the state space, that is, $\mathcal{X} = \bigtimes_{i=1}^d \mathcal{X}^{(i)} = \cup_{x^{(S)} \in \mathcal{X}^{(S)}} \Omega_{x^{(S)}}$. Let $\overline{P}$ be the ``projection chain" of $P$ with respect to $(\Omega_{x^{(S)}})_{x^{(S)} \in \mathcal{X}^{(S)}}$ in the sense of \cite{JSTV04}, which is defined to be
	\begin{align*}
		\overline{P}(x^{(S)},y^{(S)}) &:= \dfrac{\sum_{x \in \Omega_{x^{(S)}},y \in \Omega_{y^{(S)}}} \pi(x) P(x,y)}{\sum_{x \in \Omega_{x^{(S)}}}\pi(x)}. 
	\end{align*}
	This yields $\overline{P} = P^{(S)}$.
	
	
	We summarize the ergodicity properties of $P^{(S)}$ appeared in \cite{AF02,JSTV04}. We can understand that these properties are inherited from the counterpart properties of the original $P$:
	\begin{proposition}\label{prop:PSergodic}
		Let $P \in \mathcal{L}(\mathcal{X})$, $S \subseteq \llbracket d \rrbracket$ and $\pi \in \mathcal{P}(\mathcal{X})$ be a positive probability mass. We have
		\begin{enumerate}
			\item If $P$ is $\pi$-stationary, then $P^{(S)}_{\pi}$ is $\pi^{(S)}$-stationary.
			
			\item If $P$ is $\pi$-reversible, then $P^{(S)}_{\pi}$ is $\pi^{(S)}$-reversible.
			
			\item If $P$ is $\pi$-stationary, then $P^{(S)}_{\pi} \otimes P^{(-S)}_{\pi}$ is $\pi^{(S)} \otimes \pi^{(-S)}$-stationary. In particular, if $\pi = \otimes_{i=1}^{d} \pi^{(i)}$ is a product stationary distribution, then $P^{(S)}_{\pi} \otimes P^{(-S)}_{\pi}$ is $\pi$-stationary.
			
			\item If $P$ is lazy, that is, $P(x,x) \geq 1/2$ for all $x \in \mathcal{X}$, then $P^{(S)}_{\pi}$ is lazy.
			
			\item If $P$ is ergodic, then $P^{(S)}_{\pi}$ is ergodic.
		\end{enumerate}
	\end{proposition}
	%
	%
	%
	
	
	In the context of MCMC, there are often situations in which we are only interested in sampling from a subset $S$ out of the $d$ coordinates of the stationary distribution $\pi$ of a sampler. Thus, $P^{(S)}_{\pi}$ offers a natural projection sampler to approximately sample from $\pi^{(S)}$. \textcolor{black}{We also recall that in Section \ref{subsec:liftedMCMC} we have presented numerical evidence to support the use of projection samplers as motivation of this paper.}
	
	As a concrete example, we can consider the run-and-tumble Markov chains \cite{GGR21} where the algorithm maintains both the positions and directions of $d$ particles. In such setting, we are often interested in sampling from the stationary distribution of the positions of the particles only and discard the samples from the directions. Another concrete example would be the swapping algorithm that we shall discuss in Section \ref{sec:projectionsamplers}, where one maintains a system of Markov chains over a range of temperatures. The swapping algorithm is designed to sample from the Boltzmann-Gibbs distribution at the lowest temperature of the algorithm and samples that correspond to higher temperatures are often discarded. A third example is the auxiliary MCMC methods \cite{H98}, where one artificially add auxiliary variables in a MCMC algorithm to improve convergence, and at the end of simulation the samples of these auxiliary variables are discarded. In these algorithms, it is therefore natural to consider $P^{(S)}_{\pi}$ as a candidate sampler for $\pi^{(S)}$. We shall compare hitting and mixing time parameters between $P$ and $P^{(S)}_{\pi}$ in Section \ref{subsec:comparisonfunctionalconstants}. In addition, we implement these ideas and propose an improved projection sampler based on the swapping algorithm and analyze its mixing time in Section \ref{subsubsec:projectionsampler}.





Now, let us recall the partition lemma for KL divergence of probability masses \cite[Lemma $13.1.3$]{Bremaud2017}, which is a consequence of the log-sum inequality. For $\mu, \nu \in \mathcal{P}(\mathcal{X})$ and $\emptyset \neq S \subseteq \llbracket d \rrbracket$, the partition lemma gives
\begin{align*}
	\widetilde{D}_{KL}(\mu \| \nu) \geq \widetilde{D}_{KL}(\mu^{(S)} \| \nu^{(S)}).
\end{align*}
Note that this result holds independent of whether $\mu$ or $\nu$ is a product probability mass. We now state the partition lemma for KL divergence of Markov chains. It will be used in Section \ref{subsec:comparisonfunctionalconstants} to prove some monotonicity results, and in Section \ref{subsec:submodular} to demonstrate that the entropy rate is a submodular function.
\begin{theorem}[Partition lemma for KL divergence of Markov chains]\label{thm:partitionMC}
	Let $\pi \in \mathcal{P}(\mathcal{X})$, $P, L \in \mathcal{L}(\mathcal{X})$ and suppose $\emptyset \neq S \subseteq \llbracket d \rrbracket$. We have
	\begin{align*}
		D_{KL}^{\pi}(P \| L) \geq D_{KL}^{\pi^{(S)}}(P^{(S)}_\pi \| L^{(S)}_\pi).
	\end{align*}
	Note that this result holds independently of whether $\pi$ is a product probability mass or $P$ or $L$ is a product transition matrix.
\end{theorem}
\begin{proof}
	Observe that the space $\mathcal{X}^2$ can be partitioned as disjoint unions of $\Omega_{x^{(S)},y^{(S)}}$, which are given by
	\begin{align*}
		\mathcal{X}^2 &= \bigcup_{x^{(S)},y^{(S)} \in \mathcal{X}^{(S)}} \Omega_{x^{(S)},y^{(S)}}, \quad \Omega_{x^{(S)},y^{(S)}} := \bigcup_{x^{(-S)},y^{(-S)} \in \mathcal{X}^{(-S)}}((x^{(S)},x^{(-S)}),(y^{(S)},y^{(-S)})).
	\end{align*}
	This leads to
	\begin{align*}
		D_{KL}^{\pi}(P \| L) &= \sum_{x^{(S)},y^{(S)} \in \mathcal{X}^{(S)}} \sum_{x^{(-S)},y^{(-S)} \in \mathcal{X}^{(-S)}} \pi(x) P(x,y) \ln \left(\dfrac{\pi(x)P(x,y)}{\pi(x)L(x,y)}\right) \\
		&\geq \sum_{x^{(S)},y^{(S)} \in \mathcal{X}^{(S)}} \left(\sum_{x^{(-S)},y^{(-S)} \in \mathcal{X}^{(-S)}} \pi(x) P(x,y)\right) \ln \left(\dfrac{\sum_{x^{(-S)},y^{(-S)} \in \mathcal{X}^{(-S)}}\pi(x)P(x,y) }{\sum_{x^{(-S)},y^{(-S)} \in \mathcal{X}^{(-S)}}\pi(x)L(x,y) }\right) \\
		&= \sum_{x^{(S)},y^{(S)} \in \mathcal{X}^{(S)}}\pi^{(S)}(x^{(S)}) P^{(S)}_\pi(x^{(S)},y^{(S)}) \ln \left(\dfrac{\pi^{(S)}(x^{(S)}) P^{(S)}_\pi(x^{(S)},y^{(S)})}{\pi^{(S)}(x^{(S)}) L^{(S)}_\pi(x^{(S)},y^{(S)})}\right) \\
		&= D_{KL}^{\pi^{(S)}}(P^{(S)}_\pi \| L^{(S)}_\pi),
	\end{align*}
	where we apply the log-sum inequality.
\end{proof}

Next, we state the Shearer's lemma for KL divergence of Markov chains:

\begin{theorem}[Shearer's lemma for KL divergence of Markov chains]\label{thm:ShearerMC}
	Let $\pi = \otimes_{j=1}^d \pi^{(j)} \in \mathcal{P}(\mathcal{X})$ be a positive product distribution, $P, L = \otimes_{j=1}^d L_j \in \mathcal{L}(\mathcal{X})$ and $L_j \in \mathcal{L}(\mathcal{X}^{(j)})$ for $j \in \llbracket d \rrbracket$. Given a sequence $(S_i)_{i=1}^n$ with $S_i \subseteq \llbracket d \rrbracket$, where each $j \in \llbracket d \rrbracket$ belongs to at least $r$ of $S_i$. We have
	\begin{align*}
		D_{KL}^{\pi}(P \| L) \geq \dfrac{1}{r} \sum_{i=1}^n D_{KL}^{\pi^{(S_i)}}(P^{(S_i)}_\pi \| L^{(S_i)}).
	\end{align*}
	Note that $L^{(S_i)} = \otimes_{j \in S_i} L_j$. 
\end{theorem}

\begin{remark}[Han's inequality for KL divergence of Markov chains]
	In the special case of taking $S_i = \llbracket d \rrbracket \backslash \{i\}$ and $n = d$ so that $r = d - 1$, we obtain a Han's inequality of the form
	\begin{align*}
		D_{KL}^{\pi}(P \| L) \geq \dfrac{1}{d-1} \sum_{i=1}^n D_{KL}^{\pi^{(-i)}}(P^{(-i)}_\pi \| L^{(-i)}).
	\end{align*}
\end{remark}
\begin{proof}
	Using the Shearer's lemma for KL divergence of probability masses in \eqref{eq:Shearerpm}, we see that
	\begin{align*}
		D_{KL}^{\pi}(P \| \otimes_{j=1}^d L_j) &= \widetilde{D}_{KL}(\pi \boxtimes P \| \otimes_{j=1}^d (\pi^{(j)} \boxtimes L_j)) \\
		&\geq \dfrac{1}{r} \sum_{i=1}^n \widetilde{D}_{KL}((\pi \boxtimes P)^{(S_i)} \| \otimes_{j \in S_i} (\pi^{(j)} \boxtimes L_j))  \\
		&= \dfrac{1}{r} \sum_{i=1}^n D_{KL}^{\pi^{(S_i)}}(P^{(S_i)}_\pi \| \otimes_{j \in S_i} L_j).
	\end{align*}
\end{proof}

For $j \in S_i$, if we consider the $j$th marginal transition matrix of $P^{(S_i)}_\pi$ with respect to $\pi^{(S_i)}$, we compute that to be
\begin{align}\label{eq:PSij}
	(P^{(S_i)}_\pi)^{(j)}_{\pi^{(S_i)}} = P^{(j)}_\pi.
\end{align}
Using again Theorem \ref{thm:closest} leads to
\begin{align*}
	\mathbb{I}^{\pi}(P) &= D_{KL}^{\pi}(P \| \otimes_{j=1}^d P^{(j)}_\pi), \\
	\mathbb{I}^{\pi^{(S_i)}}(P^{(S_i)}_\pi) &= D_{KL}^{\pi^{(S_i)}}(P^{(S_i)}_\pi \| \otimes_{j \in S_i} (P^{(S_i)}_\pi)^{(j)}_{\pi^{(S_i)}}) = D_{KL}^{\pi^{(S_i)}}(P^{(S_i)}_\pi \| \otimes_{j \in S_i} P^{(j)}_\pi).
\end{align*}
Applying these two equalities to Theorem \ref{thm:ShearerMC} by taking $L_j = P^{(j)}_\pi$ therein leads to

\begin{corollary}[Shearer's lemma for distance to independence of $P$ with respect to $\pi$]\label{cor:Shearer}
	Let $\pi = \otimes_{j=1}^d \pi^{(j)} \in \mathcal{P}(\mathcal{X})$ be a positive product distribution, $P, L = \otimes_{j=1}^d L_j \in \mathcal{L}(\mathcal{X})$ and $L_j \in \mathcal{L}(\mathcal{X}^{(j)})$ for $j \in \llbracket d \rrbracket$. Let $(S_i)_{i=1}^n$ be a sequence with $S_i \subseteq \llbracket d \rrbracket$, where each $j \in \llbracket d \rrbracket$ belongs to at least $r$ of $S_i$. We have
	\begin{align*}
		\mathbb{I}^{\pi}(P) \geq \dfrac{1}{r} \sum_{i=1}^n \mathbb{I}^{\pi^{(S_i)}}(P^{(S_i)}_\pi).
	\end{align*}
	Equality holds if $P$ is itself a product chain where both sides equal to zero.
\end{corollary}

\begin{remark}[Han's inequality for distance to independence of $P$ with respect to $\pi$]
	In the special case of taking $S_i = \llbracket d \rrbracket \backslash \{i\}$ and $n = d$ so that $r = d - 1$, we obtain a Han's inequality of the form
	\begin{align*}
		\mathbb{I}^{\pi}(P) \geq \dfrac{1}{d-1} \sum_{i=1}^d \mathbb{I}^{\pi^{(-i)}}(P^{(-i)}_\pi).
	\end{align*}
\end{remark}

Intuitively, we can understand that the distance to independence of $P$ with respect to $\pi$ is at least greater than the ``average" distances to independence of the keep-$S_i$-in transition matrices.

\subsection{$(S_i)_{i=1}^n$-factorizable transition matrices and the distance to $(S_i)_{i=1}^n$-factorizability}\label{subsec:partitionfactor}

Throughout this subsection, we consider a mutually exclusive partition of $\llbracket d \rrbracket$ that we denote by $(S_i)_{i=1}^n$. We also assume that $|S_i| \geq 1$ for all $i$, and hence $n \in \llbracket d \rrbracket$. Intuitively, we can understand that $\llbracket d \rrbracket$ is partitioned into $n$ blocks with each block being $S_i$. The aim of this subsection is to find the closest $(S_i)_{i=1}^n$-factorizable transition matrix to a given multivariate $P$ and hence to compute the distance to $(S_i)_{i=1}^n$-factorizability of $P$. In other words, we are looking for transition matrices which are independent across different blocks but possibly dependent within each block of coordinates.

We now define the set of $(S_i)_{i=1}^n$-factorizable transition matrices.
\begin{definition}[$(S_i)_{i=1}^n$-factorizable transition matrices]\label{def:Sifactor}
	Consider a mutually exclusive partition $(S_i)_{i=1}^n$ of $\llbracket d \rrbracket$ with $|S_i| \geq 1$. A transition matrix $P \in \mathcal{L}(\mathcal{X})$ is said to be $(S_i)_{i=1}^n$-factorizable if there exists $L_i \in \mathcal{L}(\mathcal{X}^{(S_i)})$ such that $P$ can be written as
	\begin{align*}
		P = \otimes_{i=1}^n L_i.
	\end{align*}
	We write $\mathcal{L}_{\otimes_{i=1}^n S_i}(\mathcal{X})$ to be the set of all $(S_i)_{i=1}^n$-factorizable transition matrices. In particular, we write  $ \mathcal{L}_{\otimes}(\mathcal{X}) := \mathcal{L}_{\otimes_{i=1}^d \{i\}}(\mathcal{X})$ to be the set of product transition matrices on $\mathcal{X}$.
\end{definition}

It is obvious to see that $\mathcal{L}_{\otimes}(\mathcal{X}) \subseteq \mathcal{L}_{\otimes_{i=1}^n S_i}(\mathcal{X})$ for any choice of the partition $(S_i)_{i=1}^n$.

Next, we consider the information projection of $P$ onto the space $\mathcal{L}_{\otimes_{i=1}^n S_i}(\mathcal{X})$. We develop a Pythagorean identity in this case:

\begin{theorem}(Pythagorean identity)\label{thm:Pythfact}
	Consider a mutually exclusive partition $(S_i)_{i=1}^n$ of $\llbracket d \rrbracket$ with $|S_i| \geq 1$. Let $\pi \in \mathcal{P}(\mathcal{X})$ be a positive probability mass, $P \in \mathcal{L}(\mathcal{X})$, $L_i \in \mathcal{L}(\mathcal{X}^{(S_i)})$ for $i \in \llbracket n \rrbracket$. We then have
	\begin{align*}
		D_{KL}^{\pi}(P \| \otimes_{i=1}^n L_i) &= D_{KL}^{\pi}(P \| \otimes_{i=1}^n P^{(S_i)}_\pi) + D_{KL}^{\pi}(\otimes_{i=1}^n P^{(S_i)}_\pi \| \otimes_{i=1}^n L_i) \\
		&= D_{KL}^{\pi}(P \| \otimes_{i=1}^n P^{(S_i)}_\pi) + \sum_{i=1}^n D_{KL}^{\pi^{(S_i)}}(P^{(S_i)}_\pi \| L_i),
	\end{align*}
	where we recall that $P^{(S_i)}_\pi$ is the keep-$S_i$-in transition matrices of $P$ with respect to $\pi$ in Definition \ref{def:P-S}, while $D_{KL}^{\pi^{(S_i)}}(P^{(S_i)}_\pi \| L_i)$ is weighted by $\pi^{(S_i)}$, the keep-$S_i$-in marginal distribution of $\pi$. In other words, $\otimes_{i=1}^n P^{(S_i)}_\pi$, the tensor product of the keep-$S_i$-in transition matrices, is the unique minimizer of 
	\begin{align*}
		\min_{L_i \in \mathcal{L}(\mathcal{X}^{(S_i)})} D_{KL}^{\pi}(P \| \otimes_{i=1}^n L_i) = D_{KL}^{\pi}(P \| \otimes_{i=1}^n P^{(S_i)}_\pi).
	\end{align*}
\end{theorem}

\begin{remark}
	In fact we have already seen a special case earlier. If we take $n = d$ and $S_i = \{i\}$, we recover the Pythagorean identity in Theorem \ref{thm:closest}.
\end{remark}

\begin{proof}
	\begin{align*}
		D_{KL}^{\pi}(P \| \otimes_{i=1}^n L_i) &= D_{KL}^{\pi}(P \| \otimes_{i=1}^n P^{(S_i)}_\pi) + \sum_{x,y} \pi(x) P(x,y) \ln\left(\dfrac{(\otimes_{i=1}^n P^{(S_i)}_\pi)(x,y)}{(\otimes_{i=1}^n L_i)(x,y)}\right) \\
		&= D_{KL}^{\pi}(P \| \otimes_{i=1}^n P^{(S_i)}_\pi) + \sum_{i=1}^n \sum_{x^{(S_i)},y^{(S_i)}}  \pi^{(S_i)}(x^{(S_i)}) P^{(S_i)}_\pi(x^{(S_i)},y^{(S_i)}) \ln\left(\dfrac{P^{(S_i)}_\pi(x^{(S_i)},y^{(S_i)})}{L_i(x^{(S_i)},y^{(S_i)})}\right) \\
		&= D_{KL}^{\pi}(P \| \otimes_{i=1}^n P^{(S_i)}_\pi) + \sum_{i=1}^n D_{KL}^{\pi^{(S_i)}}(P^{(S_i)}_\pi \| L_i).
	\end{align*}
\end{proof}

Using Theorem \ref{thm:Pythfact} we now introduce a distance to the space $\mathcal{L}_{\otimes_{i=1}^n S_i}(\mathcal{X})$ of a given multivariate $P$:

\begin{definition}[Distance to $(S_i)_{i=1}^n$-factorizability of $P$ with respect to $\pi$]
	Consider a mutually exclusive partition $(S_i)_{i=1}^n$ of $\llbracket d \rrbracket$ with $|S_i| \geq 1$. Let $\pi \in \mathcal{P}(\mathcal{X})$ be a positive probability mass and $P \in \mathcal{L}(\mathcal{X})$. We define
	\begin{align*}
		\mathbb{I}^{\pi}(P, \mathcal{L}_{\otimes_{i=1}^n S_i}(\mathcal{X})) := \min_{L_i \in \mathcal{L}(\mathcal{X}^{(S_i)})} D_{KL}^{\pi}(P \| \otimes_{i=1}^n L_i) = D_{KL}^{\pi}(P \| \otimes_{i=1}^n P^{(S_i)}_\pi).
	\end{align*}
	In particular, if we take $S_i = \{i\}$ and $n = d$, we recover the distance to independence of $P$ with respect to $\pi$ as introduced in \eqref{def:IfP}:
	\begin{align*}
		\mathbb{I}^{\pi}(P) = \mathbb{I}^{\pi}(P, \mathcal{L}_{\otimes}(\mathcal{X})) = D_{KL}^{\pi}(P \| \otimes_{i=1}^d P^{(i)}_\pi).
	\end{align*}
\end{definition}

Finally, the above results give an interesting decomposition of the distance to independence of $P$ in terms of the distance to $(S_i)_{i=1}^n$-factorizability of $P$.

\begin{corollary}[Decomposition of the distance to independence of $P$]
	Consider a mutually exclusive partition $(S_i)_{i=1}^n$ of $\llbracket d \rrbracket$ with $|S_i| \geq 1$. Let $\pi \in \mathcal{P}(\mathcal{X})$ be a positive probability mass and $P \in \mathcal{L}(\mathcal{X})$. We have	
	\begin{align*}
		\underbrace{\mathbb{I}^{\pi}(P)}_{\text{distance to independence of $P$}} &= \underbrace{\mathbb{I}^{\pi}(P, \mathcal{L}_{\otimes_{i=1}^n S_i}(\mathcal{X}))}_{\text{distance to $(S_i)_{i=1}^n$-factorizability of $P$}} + \underbrace{\mathbb{I}^{\pi}(\otimes_{i=1}^n P^{(S_i)}_\pi)}_{\text{distance to independence of $\otimes_{i=1}^n P^{(S_i)}_\pi$}} \\
		&= \underbrace{\mathbb{I}^{\pi}(P, \mathcal{L}_{\otimes_{i=1}^n S_i}(\mathcal{X}))}_{\text{distance to $(S_i)_{i=1}^n$-factorizability of $P$}} + \sum_{i=1}^n \underbrace{\mathbb{I}^{\pi^{(S_i)}}(P^{(S_i)}_\pi)}_{\text{distance to independence of $P^{(S_i)}_\pi$}}.		
	\end{align*}
\end{corollary}

\begin{proof}
	Recall that $P^{(i)}_\pi$ is the $i$th marginal transition matrix of $P$ with respect to $\pi$. For arbitrary $(S_i)_{i=1}^n$, we take $L_i = \otimes_{j \in S_i} P^{(j)}_\pi$ in Theorem \ref{thm:Pythfact}. The desired result follows by recalling that with these choices,
	\begin{align*}
		D_{KL}^{\pi}(P \| \otimes_{i=1}^n L_i) = \mathbb{I}^{\pi}(P), \quad D_{KL}^{\pi^{(S_i)}}(P^{(S_i)}_\pi \| L_i) = \mathbb{I}^{\pi^{(S_i)}}(P^{(S_i)}_\pi).
	\end{align*}
\end{proof}

\tikzset{every picture/.style={line width=0.75pt}} 

\begin{figure}[H]\label{fig:product}
	\begin{center}
		\begin{tikzpicture}[x=0.75pt,y=0.75pt,yscale=-1,xscale=1]
			\draw   (101,198.5) .. controls (101,158.34) and (162.56,132) .. (213.83,132) .. controls (265.1,132) and (326.67,158.34) .. (326.67,198.5) .. controls (326.67,238.66) and (265.1,265) .. (213.83,265) .. controls (162.56,265) and (101,238.66) .. (101,198.5) -- cycle ;
			\draw   (218,193.5) .. controls (218,156.16) and (290.31,124) .. (354.83,124) .. controls (419.36,124) and (491.67,156.16) .. (491.67,193.5) .. controls (491.67,230.84) and (419.36,263) .. (354.83,263) .. controls (290.31,263) and (218,230.84) .. (218,193.5) -- cycle ;
			\draw    (270.67,196) .. controls (271.32,152.88) and (266.85,128.01) .. (276.08,81.85) ;
			\draw [shift={(276.67,79)}, rotate = 101.77] [fill={rgb, 255:red, 0; green, 0; blue, 0 }  ][line width=0.08]  [draw opacity=0] (10.72,-5.15) -- (0,0) -- (10.72,5.15) -- (7.12,0) -- cycle    ;
			\draw    (192.67,197) .. controls (193.33,153.44) and (220.12,119.68) .. (271.11,77.29) ;
			\draw [shift={(272.67,76)}, rotate = 140.41] [fill={rgb, 255:red, 0; green, 0; blue, 0 }  ][line width=0.08]  [draw opacity=0] (10.72,-5.15) -- (0,0) -- (10.72,5.15) -- (7.12,0) -- cycle    ;
			\draw   (275,72.5) .. controls (275,71.12) and (276.12,70) .. (277.5,70) .. controls (278.88,70) and (280,71.12) .. (280,72.5) .. controls (280,73.88) and (278.88,75) .. (277.5,75) .. controls (276.12,75) and (275,73.88) .. (275,72.5) -- cycle ;
			\draw    (359.67,196) .. controls (360.32,152.66) and (329.61,99.62) .. (285.68,75.1) ;
			\draw [shift={(283.67,74)}, rotate = 28.07] [fill={rgb, 255:red, 0; green, 0; blue, 0 }  ][line width=0.08]  [draw opacity=0] (10.72,-5.15) -- (0,0) -- (10.72,5.15) -- (7.12,0) -- cycle    ;
			\draw   (268,200.5) .. controls (268,199.12) and (269.12,198) .. (270.5,198) .. controls (271.88,198) and (273,199.12) .. (273,200.5) .. controls (273,201.88) and (271.88,203) .. (270.5,203) .. controls (269.12,203) and (268,201.88) .. (268,200.5) -- cycle ;
			\draw   (190,201.5) .. controls (190,200.12) and (191.12,199) .. (192.5,199) .. controls (193.88,199) and (195,200.12) .. (195,201.5) .. controls (195,202.88) and (193.88,204) .. (192.5,204) .. controls (191.12,204) and (190,202.88) .. (190,201.5) -- cycle ;
			\draw   (357,200.5) .. controls (357,199.12) and (358.12,198) .. (359.5,198) .. controls (360.88,198) and (362,199.12) .. (362,200.5) .. controls (362,201.88) and (360.88,203) .. (359.5,203) .. controls (358.12,203) and (357,201.88) .. (357,200.5) -- cycle ;
			\draw    (265.67,201) -- (199,201) ;
			\draw [shift={(196,201)}, rotate = 360] [fill={rgb, 255:red, 0; green, 0; blue, 0 }  ][line width=0.08]  [draw opacity=0] (10.72,-5.15) -- (0,0) -- (10.72,5.15) -- (7.12,0) -- cycle    ;
			
			\draw    (270.5,198) -- (354.83,200) ;
			\draw [shift={(357.83,200)}, rotate = 536.5699999999999] [fill={rgb, 255:red, 0; green, 0; blue, 0 }  ][line width=0.08]  [draw opacity=0] (10.72,-5.15) -- (0,0) -- (10.72,5.15) -- (7.12,0) -- cycle    ;
			
			\draw  [color={rgb, 255:red, 208; green, 2; blue, 27 }  ,draw opacity=1 ] (266.91,181) -- (266.91,193.9) -- (254,193.9) ;
			\draw  [color={rgb, 255:red, 208; green, 2; blue, 27 }  ,draw opacity=1 ] (354.83,180.82) -- (354.83,193.73) -- (341.93,193.73) ;
			\draw  [color={rgb, 255:red, 208; green, 2; blue, 27 }  ,draw opacity=1 ] (210.83,195.5) -- (197.93,195.5) -- (197.93,182.59) ;
			\draw  [color={rgb, 255:red, 208; green, 2; blue, 27 }  ,draw opacity=1 ] (288.13,193.9) -- (275.23,193.9) -- (274.77,181) ;
			
			\draw (280,49) node [anchor=north west][inner sep=0.75pt]   [align=left] {$P$};
			\draw (377,75) node [anchor=north west][inner sep=0.75pt]   [align=left] {$(\calL(\calX), D^{\pi}_{KL})$};
			\draw (146,207) node [anchor=north west][inner sep=0.75pt]   [align=left] {$\otimes_{i=1}^{n}P^{(S_i)}_\pi$};
			\draw (50,246) node [anchor=north west][inner sep=0.75pt]   [align=left] {$\calL_{\otimes_{i=1}^{n} S_i}(\calX)$};
			\draw (248,207) node [anchor=north west][inner sep=0.75pt]   [align=left] {$\otimes_{i=1}^{d}P^{(i)}_\pi$};
			\draw (455,239) node [anchor=north west][inner sep=0.75pt]   [align=left] {$\calL_{\otimes_{i=1}^{l} V_i}(\calX)$};
			\draw (360,207) node [anchor=north west][inner sep=0.75pt]   [align=left] {$\otimes_{i=1}^{l}P^{(V_i)}_\pi$};
			
		\end{tikzpicture}
		
		\caption{Visualizations of the set $\calL_{\otimes_{i=1}^{n} S_i}(\calX)$ and $\calL_{\otimes_{i=1}^{l} V_i}(\calX)$, where $(V_i)_{i=1}^l$ and $(S_i)_{i=1}^n$ are two partitions of $\llbracket d \rrbracket$. Note that $\mathcal{L}_{\otimes}(\mathcal{X}) \subseteq \calL_{\otimes_{i=1}^{n} S_i}(\calX) \cap \calL_{\otimes_{i=1}^{l} V_i}(\calX)$, and all the arrows are based upon the divergence $D^{\pi}_{KL}$. The Pythagorean identity of $P \in \mathcal{L}(\mathcal{X})$ is stated in Theorem \ref{thm:Pythfact}.}
	\end{center}
\end{figure}

\subsection{$(C_i)_{i=1}^n$-factorizable transition matrices with respect to a graph}\label{subsec:clique}

In the literature of graphical model and Markov random field, factorization of probability masses with respect to the cliques of a graph has been investigated, and culminates in the Hammersley-Clifford theorem, see for instance \cite[Chapter $11$]{W19}. In this vein, in this subsection we consider the problem of factorizability of a transition matrix with respect to the cliques of a graph.

Let us consider a given undirected graph $G = (V = \llbracket d \rrbracket, E)$. A set $C \subseteq \llbracket d \rrbracket$ is said to be a clique of $G$ if any pair of two vertices in $C$ are connected by an edge in $E$. Let $(C_i)_{i=1}^n$ be a set of cliques of the graph $G$, with possibly overlapping vertices, such that $\cup_{i=1}^n C_i = \llbracket d \rrbracket$.

We are ready to define the set of $(C_i)_{i=1}^n$-factorizable transition matrices with respect to the graph $G$:
\begin{definition}[$(C_i)_{i=1}^n$-factorizable transition matrices with respect to a graph $G$]\label{def:Cifactor}
	Let $(C_i)_{i=1}^n$ be a set of cliques of a graph $G$ on $\llbracket d \rrbracket$ such that $\cup_{i=1}^n C_i = \llbracket d \rrbracket$. A transition matrix $P \in \mathcal{L}(\mathcal{X})$ is said to be $(C_i)_{i=1}^n$-factorizable with respect to the graph $G$ and $(L_i)_{i=1}^n$ if there exists $L_i \in \mathcal{L}(\mathcal{X}^{(C_i)})$ such that $P$ can be written as, for $x,y \in \mathcal{X}$,
	\begin{align*}
		P(x,y) &\propto \prod_{i=1}^n L_i(x^{(C_i)}, y^{(C_i)}) \\
		&= \dfrac{1}{\sum_{y \in \mathcal{X}} \prod_{i=1}^n L_i(x^{(C_i)}, y^{(C_i)}) } \prod_{i=1}^n L_i(x^{(C_i)}, y^{(C_i)}) \\
		&=: \dfrac{1}{Z(x,(L_i)_{i=1}^n)} \prod_{i=1}^n L_i(x^{(C_i)}, y^{(C_i)}),
	\end{align*}
	where $Z(x,(L_i)_{i=1}^n)$ is the normalization constant of $P$ at $x$. We write $\mathcal{L}_{\otimes_{i=1}^n C_i}^G(\mathcal{X})$ to be the set of all $(C_i)_{i=1}^n$-factorizable transition matrices with respect to $G$. In particular, we note that $ \mathcal{L}_{\otimes}(\mathcal{X}) = \mathcal{L}_{\otimes_{i=1}^d \{i\}}^G(\mathcal{X})$ for any graph $G$.
\end{definition}

Let $(S_i)_{i=1}^n$ be a mutually exclusive partition of $\llbracket d \rrbracket$ as in Section \ref{subsec:partitionfactor}, and we choose a graph $G$ with $n$ disjoint blocks and the members of each $S_i$ form a clique within. With these choices we see that $\mathcal{L}_{\otimes_{i=1}^n S_i}(\mathcal{X}) = \mathcal{L}_{\otimes_{i=1}^n C_i}^G(\mathcal{X})$, and hence Definition \ref{def:Cifactor} is a generalization of Definition \ref{def:Sifactor}. These concepts are illustrated in Figure \ref{fig:CiSi}.

\begin{figure}[H]
	\centering
	\begin{minipage}{.5\textwidth}
		\centering
		\tikzset{every picture/.style={line width=0.75pt}} 
		
		\begin{tikzpicture}[x=0.75pt,y=0.75pt,yscale=-1,xscale=1]
			
			\draw   (149.5,270.25) .. controls (149.5,262.38) and (155.88,256) .. (163.75,256) .. controls (171.62,256) and (178,262.38) .. (178,270.25) .. controls (178,278.12) and (171.62,284.5) .. (163.75,284.5) .. controls (155.88,284.5) and (149.5,278.12) .. (149.5,270.25) -- cycle ;
			\draw   (149.5,367.25) .. controls (149.5,359.38) and (155.88,353) .. (163.75,353) .. controls (171.62,353) and (178,359.38) .. (178,367.25) .. controls (178,375.12) and (171.62,381.5) .. (163.75,381.5) .. controls (155.88,381.5) and (149.5,375.12) .. (149.5,367.25) -- cycle ;
			\draw   (207.5,317.25) .. controls (207.5,309.38) and (213.88,303) .. (221.75,303) .. controls (229.62,303) and (236,309.38) .. (236,317.25) .. controls (236,325.12) and (229.62,331.5) .. (221.75,331.5) .. controls (213.88,331.5) and (207.5,325.12) .. (207.5,317.25) -- cycle ;
			\draw   (267.5,271.25) .. controls (267.5,263.38) and (273.88,257) .. (281.75,257) .. controls (289.62,257) and (296,263.38) .. (296,271.25) .. controls (296,279.12) and (289.62,285.5) .. (281.75,285.5) .. controls (273.88,285.5) and (267.5,279.12) .. (267.5,271.25) -- cycle ;
			\draw   (267.5,367.25) .. controls (267.5,359.38) and (273.88,353) .. (281.75,353) .. controls (289.62,353) and (296,359.38) .. (296,367.25) .. controls (296,375.12) and (289.62,381.5) .. (281.75,381.5) .. controls (273.88,381.5) and (267.5,375.12) .. (267.5,367.25) -- cycle ;
			\draw    (176,278) -- (210,309.5) ;
			\draw    (229,327.5) -- (269,360.5) ;
			\draw    (163.75,284.5) -- (163.75,353) ;
			\draw    (281.75,285.5) -- (281.75,353) ;
			\draw    (176,359.25) -- (210,328.5) ;
			\draw    (233,308) -- (270,279.5) ;
			\draw  [dash pattern={on 4.5pt off 4.5pt}] (124,319.5) .. controls (124,278.63) and (149.97,245.5) .. (182,245.5) .. controls (214.03,245.5) and (240,278.63) .. (240,319.5) .. controls (240,360.37) and (214.03,393.5) .. (182,393.5) .. controls (149.97,393.5) and (124,360.37) .. (124,319.5) -- cycle ;
			\draw  [dash pattern={on 4.5pt off 4.5pt}] (201,320) .. controls (201,277.75) and (228.31,243.5) .. (262,243.5) .. controls (295.69,243.5) and (323,277.75) .. (323,320) .. controls (323,362.25) and (295.69,396.5) .. (262,396.5) .. controls (228.31,396.5) and (201,362.25) .. (201,320) -- cycle ;
			
			\draw (158,263) node [anchor=north west][inner sep=0.75pt]   [align=center] {$1$};
			\draw (158,360) node [anchor=north west][inner sep=0.75pt]   [align=left] {$2$};
			\draw (216,310) node [anchor=north west][inner sep=0.75pt]   [align=left] {$3$};
			\draw (276,360) node [anchor=north west][inner sep=0.75pt]    {$4$};
			\draw (276,263) node [anchor=north west][inner sep=0.75pt]   [align=left] {$5$};
			\draw (123,380) node [anchor=north west][inner sep=0.75pt]   [align=left] {$C_{1}$};
			\draw (302,382) node [anchor=north west][inner sep=0.75pt]   [align=left] {$C_{2}$};

		\end{tikzpicture}
	\end{minipage}%
	\begin{minipage}{.5\textwidth}
		\centering
		\tikzset{every picture/.style={line width=0.75pt}} 
		
		\begin{tikzpicture}[x=0.75pt,y=0.75pt,yscale=-1,xscale=1]
			
			\draw   (425.5,266.25) .. controls (425.5,258.38) and (431.88,252) .. (439.75,252) .. controls (447.62,252) and (454,258.38) .. (454,266.25) .. controls (454,274.12) and (447.62,280.5) .. (439.75,280.5) .. controls (431.88,280.5) and (425.5,274.12) .. (425.5,266.25) -- cycle ;
			\draw   (425.5,363.25) .. controls (425.5,355.38) and (431.88,349) .. (439.75,349) .. controls (447.62,349) and (454,355.38) .. (454,363.25) .. controls (454,371.12) and (447.62,377.5) .. (439.75,377.5) .. controls (431.88,377.5) and (425.5,371.12) .. (425.5,363.25) -- cycle ;
			\draw   (483.5,313.25) .. controls (483.5,305.38) and (489.88,299) .. (497.75,299) .. controls (505.62,299) and (512,305.38) .. (512,313.25) .. controls (512,321.12) and (505.62,327.5) .. (497.75,327.5) .. controls (489.88,327.5) and (483.5,321.12) .. (483.5,313.25) -- cycle ;
			\draw   (543.5,267.25) .. controls (543.5,259.38) and (549.88,253) .. (557.75,253) .. controls (565.62,253) and (572,259.38) .. (572,267.25) .. controls (572,275.12) and (565.62,281.5) .. (557.75,281.5) .. controls (549.88,281.5) and (543.5,275.12) .. (543.5,267.25) -- cycle ;
			\draw   (543.5,363.25) .. controls (543.5,355.38) and (549.88,349) .. (557.75,349) .. controls (565.62,349) and (572,355.38) .. (572,363.25) .. controls (572,371.12) and (565.62,377.5) .. (557.75,377.5) .. controls (549.88,377.5) and (543.5,371.12) .. (543.5,363.25) -- cycle ;
			\draw    (452,274) -- (486,305.5) ;
			\draw    (505,323.5) -- (545,356.5) ;
			\draw    (439.75,280.5) -- (439.75,349) ;
			\draw    (557.75,281.5) -- (557.75,349) ;
			\draw    (452,355.25) -- (486,324.5) ;
			\draw    (509,304) -- (546,275.5) ;
			\draw  [dash pattern={on 4.5pt off 4.5pt}] (400,315.5) .. controls (400,274.63) and (426.42,241.5) .. (459,241.5) .. controls (491.58,241.5) and (518,274.63) .. (518,315.5) .. controls (518,356.37) and (491.58,389.5) .. (459,389.5) .. controls (426.42,389.5) and (400,356.37) .. (400,315.5) -- cycle ;
			\draw  [dash pattern={on 4.5pt off 4.5pt}] (529,314) .. controls (529,272.85) and (541.98,239.5) .. (558,239.5) .. controls (574.02,239.5) and (587,272.85) .. (587,314) .. controls (587,355.15) and (574.02,388.5) .. (558,388.5) .. controls (541.98,388.5) and (529,355.15) .. (529,314) -- cycle ;
			
			\draw (434,260) node [anchor=north west][inner sep=0.75pt]   [align=left] {$1$};
			\draw (434,356) node [anchor=north west][inner sep=0.75pt]   [align=left] {$2$};
			\draw (492,306) node [anchor=north west][inner sep=0.75pt]   [align=left] {$3$};
			\draw (552,356) node [anchor=north west][inner sep=0.75pt]   [align=left] {$4$};
			\draw (552,260) node [anchor=north west][inner sep=0.75pt]   [align=left] {$5$};
			\draw (403,378) node [anchor=north west][inner sep=0.75pt]   [align=left] {$S_{1}$};
			\draw (580,374) node [anchor=north west][inner sep=0.75pt]   [align=left] {$S_{2}$};

		\end{tikzpicture}
	\end{minipage}
	\caption{Illustrations of $C_1,C_2,S_1,S_2$ on a given $5$-node graph with $d = 5$. We take $C_1 = \{1,2,3\}$, $C_2 = \{3,4,5\}$, $S_1 = \{1,2,3\}$ and $S_2 = \{4,5\}$. Note that these sets are all cliques of the graph and $S_1$,$S_2$ together form a partition of $\llbracket 5 \rrbracket$.}
	\label{fig:CiSi}
\end{figure}

To seek the closest $(C_i)_{i=1}^n$-factorizable transition matrix with respect to the graph $G$ and KL divergence, one such candidate is given by

\begin{definition}
	Let $\pi \in \mathcal{P}(\mathcal{X})$ be a positive probability mass and $P \in \mathcal{L}(\mathcal{X})$. We define, for all $x,y \in \mathcal{X}$,
	\begin{align*}
		\mathbf{P}_\pi(x,y) &\propto \prod_{i=1}^n P^{(C_i)}_{\pi}(x^{(C_i)}, y^{(C_i)}) \\
		&= \dfrac{1}{Z(x,(P^{(C_i)}_\pi)_{i=1}^n)} \prod_{i=1}^n P^{(C_i)}_\pi(x^{(C_i)}, y^{(C_i)}),
	\end{align*}
	where we recall that $P^{(C_i)}_\pi$ is the keep-$C_i$-in transition matrix with respect to $\pi$ and $Z(x,(P^{(C_i)}_\pi)_{i=1}^n)$ is introduced in Definition \ref{def:Cifactor}. Note that $\mathbf{P}_\pi$ depends on $(C_i)_{i=1}^n$, $P$ and $\pi$, and $\mathbf{P}_\pi \in \mathcal{L}_{\otimes_{i=1}^n C_i}^G(\mathcal{X})$.
\end{definition}

We prove a Pythagorean inequality, which allows us to conclude that $\mathbf{P}_\pi$ is indeed the closest $(C_i)_{i=1}^n$-factorizable transition matrix with normalization constants greater than or equal to that of $\mathbf{P}_\pi$, with respect to the graph $G$ and KL divergence.

\begin{theorem}(Pythagorean inequality)\label{thm:Pythfact2}
	Let $(C_i)_{i=1}^n$ be a set of cliques of a graph $G$ on $\llbracket d \rrbracket$ such that $\cup_{i=1}^n C_i = \llbracket d \rrbracket$. Let $\pi \in \mathcal{P}(\mathcal{X})$ be a positive probability mass, $P \in \mathcal{L}(\mathcal{X})$, $M \in \mathcal{L}_{\otimes_{i=1}^n C_i}^G(\mathcal{X})$, $L_i \in \mathcal{L}(\mathcal{X}^{(C_i)})$ for $i \in \llbracket n \rrbracket$ such that $M$ is $(C_i)_{i=1}^n$-factorizable with respect to the graph $G$ and $(L_i)_{i=1}^n$. For all $M$ such that $Z(x,(L_i)_{i=1}^n) \geq Z(x,(P^{(C_i)}_\pi)_{i=1}^n)$ for all $x$, we have
	\begin{align*}
		D_{KL}^{\pi}(P \| M) &\geq D_{KL}^{\pi}(P \| \mathbf{P}_\pi) + \sum_{i=1}^n D_{KL}^{\pi^{(C_i)}}(P^{(C_i)}_\pi \| L_i),
	\end{align*}
	where $P^{(C_i)}_\pi$ is the keep-$C_i$-in transition matrices of $P$ with respect to $\pi$ in Definition \ref{def:P-S}, while $D^{\pi^{(C_i)}}_{KL}(P^{(C_i)}_\pi;L_i)$ is weighted by $\pi^{(C_i)}$, the keep-$C_i$-in marginal distribution of $\pi$. In other words, $\mathbf{P}_\pi$ is the unique minimizer of 
	\begin{align*}
		\min_{M \in \mathcal{L}_{\otimes_{i=1}^n C_i}^G(\mathcal{X});~Z(x,(L_i)_{i=1}^n) \geq Z(x,(P^{(C_i)}_\pi)_{i=1}^n) \, \forall x} D_{KL}^{\pi}(P \| M) = D_{KL}^{\pi}(P \| \mathbf{P}_\pi).
	\end{align*}
\end{theorem}

\begin{proof}
	\begin{align*}
		D_{KL}^{\pi}(P \| M) &= D_{KL}^{\pi}(P \| \mathbf{P}_\pi) + \sum_{x,y} \pi(x) P(x,y) \ln\left(\dfrac{Z(x,(L_i)_{i=1}^n)}{Z(x,(P^{(C_i)}_\pi)_{i=1}^n)}\dfrac{(\otimes_{i=1}^n P^{(S_i)}_\pi)(x,y)}{(\otimes_{i=1}^n L_i)(x,y)}\right) \\
		&\geq D_{KL}^{\pi}(P \| \mathbf{P}_\pi) + \sum_{i=1}^n \sum_{x^{(C_i)},y^{(C_i)}}  \pi^{(C_i)}(x^{(C_i)}) P^{(C_i)}_\pi(x^{(C_i)},y^{(C_i)}) \ln\left(\dfrac{P^{(C_i)}_\pi(x^{(C_i)},y^{(C_i)})}{L_i(x^{(C_i)},y^{(C_i)})}\right) \\
		&= D_{KL}^{\pi}(P \| \mathbf{P}_\pi) + \sum_{i=1}^n D_{KL}^{\pi^{(C_i)}}(P^{(C_i)}_\pi \|L_i).
	\end{align*}
\end{proof}

\subsection{Comparisons of mixing and hitting time parameters between $P$ and its information projections}\label{subsec:comparisonfunctionalconstants}

Let $S \subseteq \llbracket d \rrbracket$. In Section \ref{subsec:partitionfactor}, we have seen that $P^{(S)} \otimes P^{(-S)}$, the tensor product of the keep-$S$-in and leave-$S$-out transition matrix of a given $\pi$-stationary $P \in \mathcal{L}(\mathcal{X})$, arises naturally as an information projection of $P$ onto the space $\mathcal{L}_{S \otimes \llbracket d \rrbracket \backslash S}(\mathcal{X})$. The objective of this subsection is to investigate the relationship of hitting and mixing time parameters such as commute time, right spectral gap, log-Sobolev constant and Cheeger's constant between $P$ and its information projections. These parameters play important roles in bounding the hitting or mixing time of $P$, see for instance \cite{levin2017markov,SaloffCoste1997,Montenegro2006,AF02} and the references therein.

To this end, let us fix the notations. Throughout this subsection, we assume that $X= (X_n)_{n \in \mathbb{N}}$ is an ergodic $\pi$-reversible Markov chain with transition matrix $P \in \mathcal{L}(\pi)$. In view of Proposition \ref{prop:PSergodic}, we thus see that $X^{(S)} := (X_n^{(S)})_{n \in \mathbb{N}}$ is also an ergodic $\pi^{(S)}$-reversible Markov chain with transition matrix $P^{(S)}$. We write, for $f: \mathcal{X} \to \mathbb{R}$,
\begin{align*}
	\mathbb{E}_{\pi}(f) &:= \sum_{x \in \mathcal{X}} \pi(x) f(x), \quad
	\mathrm{Var}_{\pi}(f) := \sum_{x \in \mathcal{X}} \pi(x) (f(x) - \mathbb{E}_{\pi}(f))^2, \\
	\mathrm{Ent}_{\pi}(f) &:= \mathbb{E}_{\pi} \left(f \ln \dfrac{f}{\mathbb{E}_{\pi}(f)}\right), \quad
	\mathcal{D}_{\pi}(f,f) := \dfrac{1}{2} \sum_{x,y \in \mathcal{X}} \pi(x) P(x,y) (f(x) - f(y))^2,
\end{align*}
to be, respectively, the expectation, variance, entropy of $f$ with respect to $\pi$ and the Dirichlet form of $P$ with respect to $f, \pi$. We are interested in the following list of hitting and mixing time parameters associated with $X$ and $X^{(S)}$:
\begin{itemize}
	\item(Right spectral gap, relaxation time and log-Sobolev constant)
	The right spectral gap and the log-Sobolev constant of $P$ are defined respectively to be
	\begin{align*}
		\gamma(P) := \inf_{f;~ \mathrm{Var}_{\pi}(f) \neq 0} \dfrac{\mathcal{D}_{\pi}(f,f)}{\mathrm{Var}_{\pi}(f)}, \quad \alpha(P) := \inf_{f;~ \mathrm{Ent}_{\pi}(f^2) \neq 0} \dfrac{\mathcal{D}_{\pi}(f,f)}{\mathrm{Ent}_{\pi}(f^2)}.
	\end{align*}
	Note that since $P$ is assumed to be reversible and ergodic, the right spectral gap can be written as $\gamma(P) = 1 - \lambda_2(P)$, where $\lambda_2(P) < 1$ is the second largest eigenvalue of $P$. For $S \subseteq \llbracket d \rrbracket$, we shall analogously consider the right spectral gap $\gamma(P^{(S)})$ and the log-Sobolev constant $\alpha(P^{(S)})$ of $P^{(S)}$ by replacing $\pi$ in the definitions above with $\pi^{(S)}$ and $\mathcal{X}$ with $\mathcal{X}^{(S)}$. The relaxation time of the continuized chain of $P$ is defined to be $t_{rel}(P) := 1/\gamma(P)$. 
	
	\item(Cheeger's constant)
	Let $\emptyset \neq A \subseteq \mathcal{X} = \bigtimes_{i=1}^d \mathcal{X}^{(i)}$ and $A^c := \mathcal{X}\backslash A$. We define
	\begin{align*}
		\Phi_A(P) := \dfrac{(\pi \boxtimes P)(A, A^c)}{\pi(A)},
	\end{align*}
	and hence the Cheeger's constant of $P$ \cite[Chapter $7.2$]{levin2017markov} is defined to be
	\begin{align*}
		\Phi(P) := \min_{A \subset \mathcal{X};~ 0 < \pi(A) \leq 1/2} \Phi_A(P).
	\end{align*}
	Analogously, we define the Cheeger's constant of the keep-$S$-in transition matrix $P^{(S)}$. Let $\emptyset \neq A^{(S)} \subseteq \mathcal{X}^{(S)}$ and $A^{(S)c} = \mathcal{X}^{(S)} \backslash A^{(S)}$. We then have
	\begin{align*}
		\Phi_{A^{(S)}}(P^{(S)}) &= \dfrac{(\pi^{(S)} \boxtimes P^{(S)})(A^{(S)}, A^{(S)c})}{\pi^{(S)}(A^{(S)})}, \\
		\Phi(P^{(S)}) &= \min_{A^{(S)} \subset \mathcal{X}^{(S)};~ 0 < \pi^{(S)}(A^{(S)}) \leq 1/2} \Phi_{A^{(S)}}(P^{(S)}).
	\end{align*}
	
	\item(Commute time and average hitting time) Let $x,y \in \mathcal{X}$. The hitting time to the state $x$ (resp.~$x^{(S)}$) of the Markov chain $X$ (resp.~$X^{(S)}$) are defined to be
	\begin{align*}
		\tau_x(P) := \inf \{n \geq 0;~ X_n = x\},\quad \tau_{x^{(S)}}(P^{(S)}) = \inf \{n \geq 0;~ X_n^{(S)} = x^{(S)}\},
	\end{align*}
	where the usual convention of $\inf \emptyset := 0$ applies. The mean commute time between $x$ and $y$ of $X$ are given by
	\begin{align}\label{eq:varcommute}
		\mathbb{E}_x(\tau_y(P)) + \mathbb{E}_y(\tau_x(P)) &= \sup_{0 \leq f \leq 1;~f(x)=1,f(y)=0} \dfrac{1}{\mathcal{D}_{\pi}(f,f)}.
	\end{align}
	Analogously one can define the mean commute time between $x^{(S)}$ and $y^{(S)}$ of $X^{(S)}$ to be
	$$\mathbb{E}_{x^{(S)}}(\tau_{y^{(S)}}(P^{(S)})) + \mathbb{E}_{y^{(S)}}(\tau_{x^{(S)}}(P^{(S)})).$$
	The maximal mean commute time $t_c$ is defined to be
	\begin{align*}
		t_c(P) &:= \max_{x,y \in \mathcal{X}} \mathbb{E}_x(\tau_y(P)) + \mathbb{E}_y(\tau_x(P)), \\
		t_c(P^{(S)}) &= \max_{x^{(S)},y^{(S)} \in \mathcal{X}^{(S)}} \mathbb{E}_{x^{(S)}}(\tau_{y^{(S)}}(P^{(S)})) + \mathbb{E}_{y^{(S)}}(\tau_{x^{(S)}}(P^{(S)})).
	\end{align*}
	The average hitting time $t_{av}$ is defined to be
	\begin{align*}
		t_{av}(P) &:= \sum_{x,y \in \mathcal{X}} \pi(x)\pi(y) \mathbb{E}_x(\tau_y(P)), \\
		t_{av}(P^{(S)}) &= \sum_{x^{(S)},y^{(S)} \in \mathcal{X}^{(S)}} \pi^{(S)}(x^{(S)})\pi^{(S)}(y^{(S)}) \mathbb{E}_{x^{(S)}}(\tau_{y^{(S)}}(P^{(S)})),
	\end{align*}
	
\end{itemize}

We proceed to develop results to compare these parameters between $P$ and its information projections. For instance, in the case of the right spectral gap, we are interested in bounding $\gamma(P)$ with $\gamma(P^{(S)})$ or $\gamma(P^{(S)} \otimes P^{(-S)})$.

The main result of this subsection recalls a contraction principle in \cite[Proposition $4.44$]{AF02}: the hitting and mixing time parameters such as $\gamma, \alpha, \Phi, t_c, t_{av}$ are at least ``faster" for $P^{(S)}$ than the original chain $P$. We also establish new monotonicity results for the parameters $S \mapsto \mathbb{I}^{\pi^{(S)}}(P^{(S)})$ and $S \mapsto D_{KL}^{\pi^{(S)}}(P^{(S)} \| \Pi^{(S)})$ via the partition lemma presented in Theorem \ref{thm:partitionMC}.

\begin{corollary}[Contraction principle and monotonicity]\label{cor:contraction}
	Let $\emptyset \neq T \subseteq S \subseteq \llbracket d \rrbracket$ and $P \in \mathcal{L}(\pi)$ be an ergodic transition matrix. We have
	\begin{align*}
		\gamma(P) &\leq \gamma(P^{(S)}) \leq \gamma(P^{(T)}), \\
		\alpha(P) &\leq \alpha(P^{(S)}) \leq \alpha(P^{(T)}), \\
		\Phi(P) &\leq \Phi(P^{(S)}) \leq \Phi(P^{(T)}), \\
		t_c(P) &\geq t_c(P^{(S)}) \geq t_c(P^{(T)}), \\
		t_{av}(P) &\geq t_{av}(P^{(S)}) \geq t_{av}(P^{(T)}), \\
		\mathbb{I}^{\pi}(P) &\geq \mathbb{I}^{\pi^{(S)}}(P^{(S)}) \geq \mathbb{I}^{\pi^{(T)}}(P^{(T)}), \\
		D_{KL}^{\pi}(P \| \Pi) &\geq D_{KL}^{\pi^{(S)}}(P^{(S)} \| \Pi^{(S)}) \geq D_{KL}^{\pi^{(T)}}(P^{(T)} \| \Pi^{(T)}),
	\end{align*}
		%
	where $\Pi \in \mathcal{L}(\pi)$ is the transition matrix with each row being $\pi$. All the above equalities hold if $T = \llbracket d \rrbracket$. The last two inequalities also hold for general $\pi$-stationary $P$ \textcolor{black}{without reversibility}.
\end{corollary}

\begin{remark}
	We note that it is perhaps possible to obtain tighter bounds in Corollary \ref{cor:contraction} by considering the corresponding hitting or mixing time parameters of the ``restriction chains" as in \cite{JSTV04}, see Section \ref{subsubsec:projectionsampler}.
\end{remark}


\begin{proof}[Proof of Corollary \ref{cor:contraction}]
	Once we have obtained the inequalities governing $P$ and $P^{(S)}$, we replace $P$ by $P^{(S)}$ and note that $\left(P^{(S)}\right)^{(T)} = P^{(T)}$ to reach at the inequalities governing $P^{(S)}$ and $P^{(T)}$. \textcolor{black}{By recalling Remark \ref{rk:condite}, $P^{(S)}$ can be written as 
    \begin{equation*}
        P^{(S)}(x^{(S)}, y^{(S)})=\sum_{y^{(-S)}}\mathbb E_{x^{(-S)}\sim \pi(\cdot|x^{(S)})}\left[P\left((x^{(S)}, x^{(-S)}),(y^{(S)}, y^{(-S)})\right)\right],
    \end{equation*}
    hence
    \begin{align*}
        &\quad \left(P^{(S)}\right)^{(T)}(x^{(T)},y^{(T)})\\
        &=\sum_{y^{(S\setminus T)}}\mathbb E_{x^{(S\setminus T)}\sim \pi^{(S)}(\cdot|x^{(T)})}\left[P^{(S)}\left((x^{(T)}, x^{(S\setminus T)}),(y^{(T)}, y^{(S\setminus T)})\right)\right]\\
        &=\sum_{y^{(S\setminus T)}}\mathbb E_{x^{(S\setminus T)}\sim \pi^{(S)}(\cdot|x^{(T)})}\left[\sum_{y^{(-S)}}\mathbb E_{x^{(-S)}\sim \pi(\cdot|x^{(S)})}\left[P\left((x^{(T)}, x^{(S\setminus T)}, x^{(-S)}),(y^{(T)}, y^{(S\setminus T)}, y^{(-S)})\right)\right]\right]\\
        &=\sum_{y^{(-T)}}\mathbb E_{x^{(-T)}\sim \pi(\cdot|x^{(T)})}\left[P\left((x^{(T)}, x^{(-T)}),(y^{(T)}, y^{(-T)})\right)\right]\\
        &=P^{(T)}(x^{(T)}, y^{(T)}).
    \end{align*}
    }
    As a result, it suffices to prove only the inequalities between $P$ and $P^{(S)}$.
	
	The case of $\gamma, \alpha, \Phi, t_c, t_{av}$ have already been analyzed in \cite[Proposition $4.44$]{AF02}.

	Using the partition lemma in Theorem \ref{thm:partitionMC}, we see that
	\begin{align*}
		\mathbb{I}^{\pi}(P) = D^\pi_{KL}(P \| \otimes_{i=1}^d P^{(i)}_\pi) \geq D^{\pi^{(S)}}_{KL}(P^{(S)}_\pi \| \otimes_{i \in S} P^{(i)}_\pi) = \mathbb{I}^{\pi^{(S)}}(P^{(S)}).
	\end{align*}
	Replacing $P$ with $P^{(S)}$ and $S$ with $T$ in the above equations give the second desired inequality.
	
	Finally, using again the partition lemma twice as in Theorem \ref{thm:partitionMC} leads to
	\begin{align*}
		D_{KL}^{\pi}(P \| \Pi) \geq D_{KL}^{\pi^{(S)}}(P^{(S)} \| \Pi^{(S)}) \geq D_{KL}^{\pi^{(T)}}(P^{(T)} \| \Pi^{(T)}).
	\end{align*}
\end{proof}


\textcolor{black}{
For non-reversible Markov chains, these above quantities can be far from sharply bounding mixing times. A good alternative is multiplicative spectral gap \cite[Chapter 1]{Montenegro2006}, which is the gap between $1$ and second largest singular value. For a finite Markov chain with transition matrix $P$, the multiplicative spectral gap is defined as 
\begin{equation*}
    \gamma_M(P):=\gamma(\sqrt{PP^*})=1-\left\|P\right\|_{\ell_0^2(\pi)\to \ell_0^2(\pi)},
\end{equation*}
where 
\begin{equation*}
    \left\|P\right\|_{\ell_0^2(\pi)\to \ell_0^2(\pi)}:=\sup_{f\in \ell_0^2(\pi)}\frac{\left\|Pf\right\|_2}{\left\|f\right\|_2}, \quad \ell_0^2(\pi)=\left\{f\in \ell^2(\pi):\pi(f)=0\right\}.
\end{equation*}
Next, we provide the contraction principle based on the multiplicative spectral gap for non-reversible chains.
\begin{corollary}\label{cor:multiplicative spectral gap}
    Let $\emptyset \neq T \subseteq S \subseteq \llbracket d \rrbracket$ and $P$ be an ergodic transition matrix. Without assuming the reversibility of $P$, we have 
    \begin{equation*}
        \gamma_M(P) \leq \gamma_M(P^{(S)}) \leq \gamma_M(P^{(T)}).
    \end{equation*}
\end{corollary}
\begin{proof}
    It suffices to prove the first inequality. According to Remark \ref{rk:condite}, the projected chain $P^{(S)}$ can be used to characterize the following movement on $\mathcal{X}^{(S)}$:
    \begin{enumerate}[label=(\roman*)]
        \item Starting from $x^{(S)}$, draw $x^{(-S)}\sim \pi(\cdot|x^{(S)})$;
        \item Draw $(y^{(S)}, y^{(-S)})\sim P\left((x^{(S)}, x^{(-S)}),\cdot\right)$;
        \item Update $x^{(S)}\leftarrow y^{(S)}$. 
    \end{enumerate}
    Next, we define $K_S: \ell^2(\pi)\to \ell^2(\pi^{(S)})$ and $J_S: \ell^2(\pi^{(S)})\to \ell^2(\pi)$ as 
    \begin{gather*}
        K_Sf(x^{(S)}):=\mathbb E_{x^{(-S)}\sim \pi(\cdot|x^{(S)})}\left[f(x^{(S)}, x^{(-S)})\right], \quad f\in \ell^2(\pi),\\
        J_Sg(x^{(S)}, x^{(-S)}):= g(x^{(S)}), \quad g\in \ell^2(\pi^{(S)}),
    \end{gather*}
    then it is easy to see that $P^{(S)}=K_SPJ_S$. Moreover, $K_S$ and $J_S$ are adjoint operators, since for any $f\in \ell^2(\pi)$ and $g\in \ell^2(\pi^{(S)})$, 
    \begin{align*}
        \left\langle K_S f, g\right\rangle_{\pi^{(S)}}&=\sum_{x^{(S)}}g(x^{(S)})\pi^{(S)}(x^{(S)})\sum_{x^{(-S)}}f(x^{(S)}, x^{(-S)})\pi(x^{(-S)}|x^{(S)})\\
        &=\sum_{x}f(x)g(x^{(S)})\pi(x)\\
        &=\left\langle f, J_Sg\right\rangle_\pi.
    \end{align*}
    Observing that $J_S$ is an isometric embedding with $\left\|J_S\right\|_{\ell_0^2(\pi^{(S)})\to \ell_0^2(\pi)}=1$, we have 
    \begin{align*}
        \left\|P^{(S)}\right\|_{\ell_0^2(\pi^{(S)})\to \ell_0^2(\pi^{(S)})}&=\left\|K_SPJ_S\right\|_{\ell_0^2(\pi^{(S)})\to \ell_0^2(\pi^{(S)})}\\
        &\leq \left\|K_S\right\|_{\ell_0^2(\pi)\to \ell_0^2(\pi^{(S)})}\cdot \left\|P\right\|_{\ell_0^2(\pi)\to \ell_0^2(\pi)}\cdot \left\|J_S\right\|_{\ell_0^2(\pi^{(S)})\to \ell_0^2(\pi)}\\
        &=\left\|J_S\right\|_{\ell_0^2(\pi^{(S)})\to \ell_0^2(\pi)}\cdot \left\|P\right\|_{\ell_0^2(\pi)\to \ell_0^2(\pi)}\cdot \left\|J_S\right\|_{\ell_0^2(\pi^{(S)})\to \ell_0^2(\pi)}\\
        &=\left\|P\right\|_{\ell_0^2(\pi)\to \ell_0^2(\pi)},
    \end{align*}
    where we have used the fact that adjoint operators share the same norm. Then the result follows.
\end{proof}
}

The next result compares \textcolor{black}{$\gamma_M(P)$} and \textcolor{black}{$\gamma_M(P^{(S)} \otimes P^{(-S)})$} to obtain

\textcolor{black}{
\begin{corollary}\label{cor:mspectralgap}
	Let $\emptyset \neq S \subseteq \llbracket d \rrbracket$. Let $P$ be an ergodic and $\pi$-stationary transition matrix. We have 
    \begin{equation*}
        \gamma_M(P) \leq \gamma_M(P^{(S)} \otimes P^{(-S)}).
    \end{equation*}
    If we further assume that $P$ is lazy and reversible, then we have
	\begin{align*}
		\gamma(P) \leq \gamma(P^{(S)} \otimes P^{(-S)}).
	\end{align*}
	In particular, this yields $t_{rel}(P) \geq t_{rel}(P^{(S)} \otimes P^{(-S)})$. 
\end{corollary}
}
\begin{proof}
    \textcolor{black}{It is easy to verify that $\left(P^{(S)} \otimes P^{(-S)}\right)^*=\left(P^{(S)}\right)^*\otimes \left(P^{(-S)}\right)^*$ and 
    \begin{equation*}
        \left(P^{(S)} \otimes P^{(-S)}\right) \left(P^{(S)} \otimes P^{(-S)}\right)^*=P^{(S)}\left(P^{(S)}\right)^*\otimes P^{(-S)}\left(P^{(-S)}\right)^*.
    \end{equation*}
    Recalling for two reversible transition matrices $Q_1\in \mathcal{L}(\pi_1)$ and $Q_2\in \mathcal{L}(\pi_2)$ with non-negative eigenvalues, it is well known that 
    \begin{equation*}
        \lambda_2(Q_1\otimes Q_2)=\max \left\{\lambda_2(Q_1), \lambda_2(Q_2)\right\},
    \end{equation*}
    then the first result comes from Corollary \ref{cor:multiplicative spectral gap}. If $P$ is further assumed to be lazy and reversible, both $P^{(S)}, P^{(-S)}$ are also lazy by Proposition \ref{prop:PSergodic}, and hence $\lambda_2(P^{(S)} \otimes P^{(-S)}) = \max\{\lambda_2(P^{(S)}), \lambda_2(P^{(-S)})\}$. The desired result follows from Corollary \ref{cor:contraction}.
    }
\end{proof}

\subsection{Some submodular functions arising in the information theory of multivariate Markov chains}\label{subsec:submodular}

This subsection is devoted to prove that the mappings $S \mapsto H(P^{(S)})$ and $S \mapsto D(P \| P^{(S)} \otimes P^{(-S)})$ are submodular in $S$. These two properties are analogous to the counterpart properties of the Shannon entropy and the mutual information of random variables: they are respectively monotonically non-decreasing submodular and submodular functions (see for example \cite[Chapter $1.4, 1.5$]{polyanskiy2022information}). Note that in the i.i.d. case, the submodularity of Shannon entropy implies the Han's inequality via the notion of self-bounding function, see for example \cite[Corollary~2]{sason2022information}.

\begin{proposition}
	Let $S \subseteq \llbracket d \rrbracket$. Let $P \in \calL(\calX)$ with 
	stationary distribution $\pi$. We have
	\begin{enumerate}
		
		\item(Submodularity of the entropy rate of $P$)\label{it:submentropy}  The mapping $S \mapsto H(P^{(S)})$ is submodular.
		
		\item(Submodularity of the distance to $(S, \llbracket d \rrbracket \backslash S)$-factorizability of $P$)\label{it:submfact} The mapping $S \mapsto D(P \| P^{(S)} \otimes P^{(-S)})$ is submodular, where we recall that $D(P \| P^{(S)} \otimes P^{(-S)})$ is the distance to $(S, \llbracket d \rrbracket \backslash S)$-factorizability of $P$ in Section \ref{subsec:partitionfactor}.
		
		\item(Supermodularity and monotonicity of the distance to independence)\label{it:supermd} The mapping $S \mapsto \mathbb{I}(P^{(S)})$ is monotonically non-decreasing and supermodular.
		
	\end{enumerate}
\end{proposition}

\begin{proof}
	We first prove item \eqref{it:submentropy}. Let $S \subseteq T \subseteq \llbracket d \rrbracket$ and suppose $i \in \llbracket d \rrbracket \backslash T$. By the definition of submodularity, it suffices to show that
	\begin{align*}
		H(P^{(S\cup \{i\})}) - H(P^{(S)}) \geq H(P^{(T\cup \{i\})}) - H(P^{(T)}).
	\end{align*}
	Rearranging these terms implies that it is equivalent to show that
	\begin{align*}
		D^{\pi^{(T\cup \{i\})}}_{KL}(P^{(T\cup \{i\})} \| P^{(T)} \otimes P^{(i)}) \geq D^{\pi^{(S \cup \{i\})}}_{KL}(P^{(S \cup \{i\})} \| P^{(S)} \otimes P^{(i)}).
	\end{align*}
	This holds owing to the partition lemma in Theorem \ref{thm:partitionMC}. 
	
	
	Next, we prove item \eqref{it:submfact}. By property of submodular function, the mapping $S \mapsto H(P^{(-S)})$ is also submodular, and hence
	$$D(P \| P^{(S)} \otimes P^{(-S)}) = H(P^{(S)}) + H(P^{(-S)})- H(P)$$
	is submodular since the sum of two submodular functions is submodular.
	
	Finally, we prove item \eqref{it:supermd}. The monotonicity is shown in Corollary \ref{cor:contraction}. We also note that
	\begin{align*}
		\mathbb{I}(P^{(S)}) = \sum_{i \in S} H(P^{(i)}) - H(P^{(S)}),
	\end{align*}
	which is a sum of a modular function and a supermodular function $- H(P^{(S)})$. As the sum of two supermodular functions is supermodular, $\mathbb{I}(P^{(S)})$ is supermodular.

\end{proof}

\section{Applications of projection chains to the design of MCMC samplers}\label{sec:projectionsamplers}

The main aim of this section is to concretely illustrate and apply the notion of projection chains (i.e. keep-$S$-in or leave-$S$-out chains) to design MCMC samplers. As this part is of independent interest, we have decided to single it out as an individual section. In particular, in Section \ref{subsubsec:projectionsampler} and Section \ref{subsec:ltemp} we shall design an improved projection sampler over the original swapping algorithm and analyze its mixing time.

We stress that the technique of projection chains is not limited to the swapping algorithm. In view of Corollary \ref{cor:contraction}, the technique is broadly applicable to speedup mixing of multivariate Markov chains (such as particle-based MCMC) with stationary distribution $\pi$ and we are only interested in sampling from $\pi^{(S)}$, under the assumption that certain conditional distributions can be sampled from (recall Remark \ref{rk:condite}). In this section, we focus on a particular reversible multivariate Markov chain (swapping algorithm), as finer or more precise mixing time results can usually be obtained for reversible chains.

\subsection{An improved projection sampler based on the swapping algorithm and its mixing time analysis}\label{subsubsec:projectionsampler}

We first consider the special case of $d = 2$-temperature swapping algorithm. The main results of this section (\textcolor{black}{Corollaries} \ref{cor:projectionsamplerN} and \ref{cor:projectionsamplermain} below) state that various mixing time parameters of the keep-$\{2\}$-in (or leave-$\{1\}$-out) chain based on a two-temperature swapping algorithm are improved by at least a factor of the dimension of the underlying state space than the original swapping algorithm. We also offer an intuitive explanation on why this projection sampler accelerates over the original algorithm in Remark \ref{rk:intuitive2}.

To this end, let us fix the notations and briefly recall the dynamics of the swapping algorithm. Much of our exposition in this example follows that in \cite{BR16}. Let $P_0$ be an ergodic and reversible chain with stationary distribution being the discrete uniform $\pi_0$ on $\mathcal{X}$. Denote the Boltzmann-Gibbs distribution associated with energy function $\mathcal{H}: \mathcal{X} \to \mathbb{R}$ at inverse temperature $\beta \geq 0$ to be
\begin{align*}
	\pi_{\beta}(x) \propto e^{-\beta \mathcal{H}(x)}.
\end{align*}
Let $0=: \beta_1 < \ldots < \beta_d := \beta$ be a sequence of inverse temperatures with $d \geq 2$, and we denote by $\mathcal{X}_{sw} := \mathcal{X}^{d}$, the $d$ products of the original state space $\mathcal{X}$, to be the state space of the swapping chain. Let $P_{sw}$ be the transition matrix of the swapping chain, whose stationary distribution $\pi_{sw}$ is of product form with
$$\pi_{sw} := {\otimes}_{i=1}^d \pi_{\beta_i}.$$

At each step, the swapping chain chooses uniformly at random between the following two moves:
\begin{itemize}
	\item(Level move): In a level move, the swapping chain chooses an inverse temperature $\beta_i$ uniformly at random. The swapping chain moves the $i$th coordinate according to a Metropolis-Hastings chain at inverse temperature $\beta_i$, while the remaining coordinates are kept fixed.
	
	\item(Swap move): In a swap move, the swapping chain chooses an index $i \in \llbracket d-1 \rrbracket$ and swaps the coordinate $x_{i}$ and $x_{i+1}$ with a suitable acceptance probability. Precisely, we have for all $i \in \llbracket d-1 \rrbracket$, $x = (x_1,\ldots,x_d), y = (x_1,\ldots,x_{i+1},x_{i},\ldots,x_d)$,
	\begin{align*}
		P_{sw}(x,y) &= \dfrac{1}{2(d-1)} \min\bigg\{1,\dfrac{\pi_{sw}(y)}{\pi_{sw}(x)}\bigg\} \\
		&= \dfrac{1}{2(d-1)} \min\bigg\{1,\dfrac{\pi_{\beta_i}(x_{i+1})\pi_{\beta_{i+1}}(x_{i})}{\pi_{\beta_i}(x_{i})\pi_{\beta_{i+1}}(x_{i+1})}\bigg\} \\
		&= \dfrac{1}{2(d-1)} e^{-(\beta_{i+1}-\beta_i)(\mathcal{H}(x_{i})-\mathcal{H}(x_{i+1}))_+}.
	\end{align*}
\end{itemize}

In the special case of $d = 2$, the swapping algorithm amounts to running two Markov chains (aka particles) simultaneously with one at inverse temperature $0$ and the other one at the target $\beta$ coupled with swapping moves between the states of these two chains. With these choices, $\pi_{sw} = \pi_0 \otimes \pi_{\beta}$, the product distribution of $\pi_{0}$ and $\pi_{\beta}$. The keep-$\{2\}$-in (or $2$nd marginal as in Definition \ref{def:Pi}) Markov chain with transition matrix $P_{sw}^{(2)}$ can be written as, for $x^2, y^2 \in \mathcal{X}$,
\begin{align*}
	P_{sw}^{(2)}(x^2,y^2) = \sum_{x^1,y^1 \in \mathcal{X}} \pi_0(x^1) P_{sw}((x^1,x^2),(y^1,y^2)).
\end{align*}
Thus, to simulate one step of the keep-$\{2\}$-in chain $P_{sw}^{(2)}$ with a starting state $x^2$, we first draw a random state $x^1$ according to the discrete uniform $\pi_0$, then the Markov chain is evolved according to the swapping chain $P_{sw}$ from $(x^1,x^2)$ to $(y^1,y^2)$. In view of Remark \ref{rk:condite}, note that we implicitly assume we are able to sample from $\pi_0$, which is possible for instance in Ising models where the state space can be of the form $\mathcal{X} = \{0,1\}^N$. This assumption however can be unrealistic when the state space is more complicated such that one may not be able to sample from $\mathcal{X}$ uniformly.

Observe that the state space of this two-temperature swapping chain can be partitioned as disjoint unions of $\Omega_x$ given by 
\begin{align*}
	\mathcal{X}^2 &= \cup_{x \in \mathcal{X}} \Omega_x, \quad \Omega_x := \mathcal{X} \times \{x\}.
\end{align*}
In the terminologies of \cite{JSTV04}, the projection chain (i.e. the notation $\overline{P}$ therein) of $P_{sw}$ with respect to the above partition is $P_{sw}^{(2)}$, while the restriction chains (i.e. the notation $P_i$ therein) on the state space $\Omega_x$ can be shown to be, for each $x \in \mathcal{X}$,
$$(P_{sw})_x := \dfrac{1}{4} P_0 + \dfrac{3}{4} I.$$
Note that the right hand side of the above expression does not depend on $x$. Thus, the right spectral gap and log-Sobolev constant of $(P_{sw})_x$ are 
\begin{align}\label{eq:restrictiongamma}
	\gamma((P_{sw})_x) = \dfrac{1}{4} \gamma(P_0), \quad \alpha((P_{sw})_x) = \dfrac{1}{4} \alpha(P_0).
\end{align}

Since the projection chain $P_{sw}^{(2)}$ and the restriction chains $(P_{sw})_x$ are ergodic, Theorem $1$ and Theorem $4$ in \cite{JSTV04} can be applied in this setting to yield

\begin{corollary}[Relaxation time and log-Sobolev time of $P_{sw}^{(2)}$ are at least three times faster than that of $P_{sw}$]
	\begin{align*}
		\gamma(P_{sw}) &\leq \dfrac{1}{3} \gamma(P_{sw}^{(2)}), \quad
		\alpha(P_{sw}) \leq \dfrac{1}{3} \alpha(P_{sw}^{(2)}).
	\end{align*}
\end{corollary}

In view of the above result, it is thus advantageous to use the projection sampler $P_{sw}^{(2)}$ rather than the original $P_{sw}$ to sample from $\pi_{\beta}$, as the former is at least three times faster than the latter in terms of relaxation or log-Sobolev time. Note that this result is also an improvement towards the general contraction principle presented in Corollary \ref{cor:contraction}.

Let us now specialize into $\mathcal{X} = \{0,1\}^N$ with $N \in \mathbb{N}$, and consider $P_0$ being the transition matrix of the simple random walk on the hypercube $\mathcal{X}$. In other words, at each time a coordinate is chosen uniformly at random, and the entry of the chosen coordinate is flipped to the other with probability $1$ while keeping all other coordinates unchanged. It is well-known (see e.g. \cite[Section $4.5$]{JSTV04}) that $\alpha(P_0) = \gamma(P_0)/2 = \frac{1}{N+1}$, and hence, in view of \eqref{eq:restrictiongamma}, we arrive at
\begin{align}\label{eq:restrictiongamma2}
	\gamma((P_{sw})_x) = \dfrac{1}{2} \dfrac{1}{N+1}, \quad \alpha((P_{sw})_x) = \dfrac{1}{4} \dfrac{1}{N+1}.
\end{align}

Denote the parameter $\Gamma$ (i.e. the notation $\gamma$ in \cite{JSTV04}) to be
\begin{align}\label{eq:Gamma}
	\Gamma := \max_{x \in \mathcal{X}} \max_{y \in \mathcal{X}} \left(1 - \sum_{z \in \mathcal{X}} P_{sw}((y,x),(z,x))\right).
\end{align}
This parameter $\Gamma$ measures the probability of escaping from one block of partition $\Omega_x$ maximized over all states $x$. Let $x^* = \argmax \mathcal{H}(x)$ and $y^*$ be chosen such that $y^* \neq x^*$, then it can readily be seen that $\Gamma$ is attained with these choices and
\begin{align*}
	\Gamma = 1 - \sum_{z \in \mathcal{X}} P_{sw}((y^*,x^*),(z,x^*)) = 1 - \dfrac{1}{4} = \dfrac{3}{4}.
\end{align*}

Using again Theorem $1$ and Theorem $4$ in \cite{JSTV04} leads to

\begin{corollary}[Relaxation time and log-Sobolev time of $P_{sw}^{(2)}$ are at least $N$ times faster than that of $P_{sw}$]\label{cor:projectionsamplerN}
	On the state space $\mathcal{X} = \{0,1\}^N$ with $P_0$ being the simple random walk on $\mathcal{X}$, we have
	\begin{align*}
		\dfrac{1}{\gamma(P_{sw})} &= 2(N+1) + \dfrac{9}{2}(N+1) \dfrac{1}{\gamma(P_{sw}^{(2)})}, \\
		\dfrac{1}{\alpha(P_{sw})} &= 4(N+1) + 9(N+1) \dfrac{1}{\alpha(P_{sw}^{(2)})}.
	\end{align*}
\end{corollary}

From the viewpoint of MCMC, using the projection sampler $P_{sw}^{(2)}$ can save a factor of $N$, the dimension of $\mathcal{X}$, when compared with the original swapping chain $P_{sw}$.

We proceed to compare the (worst-case $L^2$) mixing time of the continuized chain of $P_{sw}$ and $P_{sw}^{(2)}$. For $t \geq 0$, define the heat kernel of a $P \in \mathcal{L}(\mathcal{X})$ to be
$$\mathbf{H}_t(P) := e^{t(P - I)}.$$
If $P$ is $\pi$-stationary, the mixing time of the continuized chain of $P$ is defined to be
\begin{align*}
	T_{mix}(P,\varepsilon) := \inf\bigg\{t \geq 0;~ \max_{x \in \mathcal{X}} \sqrt{\sum_{y \in \mathcal{X}} \pi(y)\left(\dfrac{\mathbf{H}_t(P)(x,y)}{\pi(y)} -1\right)^2} < \varepsilon\bigg\}.
\end{align*}
A celebrated result \cite[Page $697$]{DSC96} gives that
\begin{align}\label{eq:logsobolevmix}
	\dfrac{1}{2 \alpha(P)}\leq T_{mix}(P, 1/e) \leq \dfrac{4 + \log \log (1/\min_x \pi(x))}{\alpha(P)}.
\end{align}
Using \eqref{eq:logsobolevmix} together with Corollary \ref{cor:projectionsamplerN} gives

\begin{corollary}\label{cor:projectionsamplermain}
	In the setting of Corollary \ref{cor:projectionsamplerN}, let $\mathrm{Osc}(\mathcal{H}) = \max_x \mathcal{H}(x) - \min_x \mathcal{H}(x)$ be the oscillation of the function $\mathcal{H}$. We have
	\begin{align*}
		T_{mix}(P_{sw}, 1/e) &\geq \dfrac{9(N+1)}{2\alpha(P_{sw}^{(2)})} = \Omega\left(\dfrac{N}{\alpha(P_{sw}^{(2)})}\right), \\
		T_{mix}(P_{sw}^{(2)}, 1/e) &\leq \dfrac{4+\log(\beta \mathrm{Osc}(\mathcal{H}) + N \log 2)}{\alpha(P_{sw}^{(2)})}. 
	\end{align*}
	In particular, if $\mathrm{Osc}(H) = \mathcal{O}(N^k)$ for some $k > 0$, then for large enough $N$ we have
	$$T_{mix}(P_{sw}^{(2)}, 1/e) = \mathcal{O}\left(\dfrac{\log(\beta N)}{\alpha(P_{sw}^{(2)})}\right).$$
\end{corollary}
Many models in statistical physics satisfy a polynomial in $N$ oscillation of $\mathcal{H}$ with $\mathrm{Osc}(\mathcal{H}) = \mathcal{O}(N^k)$ for some positive integers $k$, for instance the Curie-Weiss model on a complete graph or the Ising model on finite grid \cite{NZ19}. From the viewpoint of MCMC again, the above results indicate that at times it is advantageous to simulate the keep-$\{2\}$-in $P_{sw}^{(2)}$ over $P_{sw}$ with a speedup of at least a factor of $N/\log (\beta N)$.

\subsection{Generalization to $d$-temperature swapping algorithm with $d \geq 2$}\label{subsec:ltemp}

The discussion so far in this section can be generalized to design a projection sampler based on the $d$-temperature swapping algorithm with $\mathbb{N} \ni d \geq 2$. The main results of this section (Corollary \ref{cor:projectionsamplerlN} and \ref{cor:projectionsamplermainl} below) state that various mixing time parameters of the leave-$\{1\}$-out chain based on a $d$-temperature swapping algorithm can be improved by at least a factor of the dimension of the underlying state space times the number of temperatures $d$. We also offer an intuitive explanation on the acceleration effect in Remark \ref{rk:intuitive2}.

Precisely, the projected leave-$\{1\}$-out Markov chain with transition matrix $P_{sw}^{(-1)}$ can be written as, for $x^{(-1)}, y^{(-1)} \in \mathcal{X}^{d-1}$,
\begin{align*}
	P_{sw}^{(-1)}(x^{(-1)},y^{(-1)}) = \sum_{x^1,y^1 \in \mathcal{X}} \pi_0(x^1) P_{sw}((x^1,x^{(-1)}),(y^1,y^{(-1)})).
\end{align*}
To simulate one step from the transition matrix $P_{sw}^{(-1)}$ with a starting state $x^{(-1)}$, we first draw a random state $x^1$ according to the discrete uniform $\pi_0$, then the Markov chain is evolved according to the swapping chain $P_{sw}$ from $(x^1,x^{(-1)})$ to $(y^1,y^{(-1)})$. Note that again we implicitly assume we are able to sample from $\pi_0$.

The state space of the $d$-temperature swapping chain can be decomposed into disjoint unions of $\Omega_{x^{(-1)}}$, that is, 
\begin{align*}
	\mathcal{X}_{sw} = \mathcal{X}^d &= \cup_{x^{(-1)} \in \mathcal{X}^{d-1}} \Omega_{x^{(-1)}}, \quad \Omega_{x^{(-1)}} := \mathcal{X} \times \{x^{(-1)}\}.
\end{align*}
In the terminologies of \cite{JSTV04}, the projection chain (i.e. the notation $\overline{P}$ therein) of $P_{sw}$ with respect to the above partition is $P_{sw}^{(-1)}$, while the restriction chains (i.e. the notation $P_i$ therein) on the state space $\Omega_{x^{(-1)}}$ can be shown to be, for each $x^{(-1)} \in \mathcal{X}^{d-1}$,
$$(P_{sw})_{x^{(-1)}} := \dfrac{1}{2d} P_0 + \left(1-\dfrac{1}{2d}\right) I.$$
Note that the right hand side of the above expression does not depend on $x^{(-1)}$. Thus, the right spectral gap and log-Sobolev constant of $(P_{sw})_{x^{(-1)}}$ are 
\begin{align}\label{eq:restrictiongammal}
	\gamma((P_{sw})_{x^{(-1)}}) = \dfrac{1}{2d} \gamma(P_0), \quad \alpha((P_{sw})_{x^{(-1)}}) = \dfrac{1}{2d} \alpha(P_0).
\end{align}

We now consider $\mathcal{X} = \{0,1\}^N$ with $N \in \mathbb{N}$, and take $P_0$ to be the transition matrix of the simple random walk on the hypercube $\mathcal{X}$. As $\alpha(P_0) = \gamma(P_0)/2 = \frac{1}{N+1}$, using \eqref{eq:restrictiongammal} we see that
\begin{align}\label{eq:restrictiongammal2}
	\gamma((P_{sw})_{x^{(-1)}}) = \dfrac{1}{d} \dfrac{1}{N+1}, \quad \alpha((P_{sw})_{x^{(-1)}}) = \dfrac{1}{2d} \dfrac{1}{N+1}.
\end{align}

We now recall the parameter $\Gamma$ (i.e. the notation $\gamma$ in \cite{JSTV04}) introduced earlier in \eqref{eq:Gamma}. Let $x^* = \argmax \mathcal{H}(x)$ and $y^*$ be chosen such that $y^* \neq x^*$. Let $\mathbf{x}^* \in \mathcal{X}^{d-1}$ be a $(d-1)$-dimensional vector with all entries equal to $x^*$. It can then readily be seen that 
\begin{align*}
	1 \geq \Gamma \geq 1 - \sum_{z \in \mathcal{X}} P_{sw}((y^*,\mathbf{x}^*),(z,\mathbf{x}^*)) = 1 - \dfrac{1}{2d} - \dfrac{1}{2}\dfrac{d-2}{d-1}.
\end{align*}

Using again Theorem $1$ and Theorem $4$ in \cite{JSTV04} leads to

\begin{corollary}[Relaxation time and log-Sobolev time of $P_{sw}^{(-1)}$ are at least $d N$ times faster than that of $P_{sw}$]\label{cor:projectionsamplerlN}
	On the state space $\mathcal{X} = \{0,1\}^N$ with $P_0$ being the simple random walk on $\mathcal{X}$, we have
	\begin{align*}
		\dfrac{1}{\gamma(P_{sw})} &\geq d(N+1) + 3\left(1 - \dfrac{1}{2d} - \dfrac{1}{2}\dfrac{d-2}{d-1}\right)d(N+1) \dfrac{1}{\gamma(P_{sw}^{(-1)})}, \\
		\dfrac{1}{\alpha(P_{sw})} &\geq 2d(N+1) + 6\left(1 - \dfrac{1}{2d} - \dfrac{1}{2}\dfrac{d-2}{d-1}\right)d(N+1) \dfrac{1}{\alpha(P_{sw}^{(-1)})}.
	\end{align*}
\end{corollary}

\begin{remark}[An intuitive justification on the speedup of $P_{sw}^{(-1)}$ over $P_{sw}$]\label{rk:intuitive2}
	In simulating $P_{sw}^{(-1)}$, we first sample according to the stationary distribution of the first coordinate $\pi_0$, followed by a step in the swapping algorithm. As the first coordinate is at stationarity, the swapping algorithm only needs to equilibrate the remaining $d-1$ coordinates.
	
	On the other hand, in simulating the original swapping algorithm $P_{sw}$, efforts are required to equilibrate all $d$ coordinates simultaneously.	
\end{remark}

From the perspective of MCMC, using the projection sampler $P_{sw}^{(-1)}$ can save a factor of $dN$, the number of Markov chains (or temperatures) times the dimension of $\mathcal{X}$, when compared with the original swapping chain $P_{sw}$.

In addition to the speedup, the projection sampler can be interpreted as a randomized swapping algorithm: at each step, the first coordinate is refreshed or randomly resampled from $\pi_0$. This randomized feature allows the projection sampler to start fresh at times and discard or throw away local modes, which is not possible in the original swapping algorithm. To illustrate, consider a $d=3$-temperature swapping algorithm where the current third coordinate $x^3$ is a local mode of $\pi_{\beta}$. In the original swapping algorithm, there is a positive probability that $x^3$ is swapped from the third to second to first back to second and third coordinate, which is not ideal. On the other hand, in the proposed projection sampler, once $x^3$ is swapped to second and then to the first coordinate, $x^3$ will be discarded and the algorithm starts fresh. However, we should note that if $x^3$ is a global mode of $\pi_{\beta}$, then this discarding feature may not be ideal as well.

Utilizing \eqref{eq:logsobolevmix} together with Corollary \ref{cor:projectionsamplerlN} leads us to

\begin{corollary}\label{cor:projectionsamplermainl}
	In the setting of Corollary \ref{cor:projectionsamplerlN}, let $\mathrm{Osc}(\mathcal{H}) = \max_x \mathcal{H}(x) - \min_x \mathcal{H}(x)$ be the oscillation of the function $\mathcal{H}$. We have
	\begin{align*}
		T_{mix}(P_{sw}, 1/e) &= \Omega\left(\dfrac{dN}{\alpha(P_{sw}^{(-1)})}\right), \\
		T_{mix}(P_{sw}^{(-1)}, 1/e) &\leq \dfrac{4 + \log(\beta d\mathrm{Osc}(\mathcal{H}) + N \log 2)}{\alpha(P_{sw}^{(-1)})}. 
	\end{align*}
	In particular, if $\mathrm{Osc}(\mathcal{H}) = \mathcal{O}(N^k)$ for some $k > 0$, then for large enough $N$ we have
	$$T_{mix}(P_{sw}^{(-1)}, 1/e) = \mathcal{O}\left(\dfrac{\log(\beta d N)}{\alpha(P_{sw}^{(-1)})}\right).$$
\end{corollary}

In the literature, it is noted in \cite{MZ03} that a common choice is to set $d$ proportional to $N$. As a result, if we choose $d = N$, then the leave-$\{1\}$-out projection chain enjoys at least a speedup of a multiplicative factor of $N^2/\log(N)$ when compared with the original swapping algorithm in terms of the worst-case $L^2$ mixing time.

The analysis in this section can be generalized to the case where the highest temperature (or smallest inverse temperature) of the swapping algorithm is more generally $\beta_1 \geq 0$, and under the assumption that we can sample from $\pi_{\beta_1}$. While in practice it may not be possible to do so, often we have rapidly mixing Markov chains at high temperatures, which can be used for the resampling step as a surrogate for sampling from $\pi_{\beta_1}$.

\subsection{Numerical experiments}\label{subsec:num}

In this section, we present a simple bimodal example as in Section \ref{subsubsec:numlifted}, where the last coordinate of $P_{sw}^{(-1)}$ mixes well while that of $P_{sw}$ is stuck at the region around one mode.

In view of the setting and notations in this section and Section \ref{subsubsec:numlifted}, suppose the state space is $\mathcal{X} = \llbracket -n,n \rrbracket$ on which the target distribution is, for $x \in \mathcal{X}$,
\begin{align*}
	\pi_\beta(x) \propto 2^{|x|} = e^{-\beta \mathcal{H}(x)},
\end{align*}
where we take $\beta = \ln 2$ and $\mathcal{H}(x) = -|x|$. There are two modes of this distribution at $\pm n$ respectively. In this context, $\pi_0$ is simply the discrete uniform distribution on $\mathcal{X}$ which can be sampled at ease.

For reproducibility, the code used in our experiments is available at \url{https://github.com/mchchoi/factorization/tree/main}. We now state the parameters used in the experiments:
\begin{itemize}
	\item $n = 100$.
	\item $d = 3$ with temperature ladder $(\beta_1,\beta_2,\beta_3) = (0,\frac{\ln2}{2}, \ln 2)$.
	\item The proposal chain of Metropolis-Hastings during the level move of the swapping algorithm moves from $x$ to $\min\{x+1,n\}$ and $\max\{x-1,-n\}$ with probability $1/2$, and $0$ otherwise.
	\item All samplers $P_{sw}^{(-1)}, P_{sw}$ are initialized at $-100$, the mode on the left, and are simulated for $100,000$ steps.
\end{itemize}

The results are summarized and presented in Figure \ref{fig:Vshapeswap}, Table \ref{tab:firstswap} and \ref{tab:secondswap}.

First, we note that $P_{sw}$ does not exhibit mixing: from the traceplot, histogram and empirical mean, it only explores the basin around the left mode at $-100$ and do not traverse to the right mode at $100$ in the experiment.

Second, from the traceplots and histograms we see that $P_{sw}^{(-1)}$ is able to hop between the two modes effectively. From Table \ref{tab:first} the empirical distribution generated by $P_{sw}^{(-1)}$ is notably closer to the ground truth $\pi_\beta$ than that generated by $P_{sw}$. From Table \ref{tab:secondswap} the empirical mean and second moment generated by $P_{sw}^{(-1)}$ are also closer to the respective ground truth than that generated by $P_{sw}$. These results give empirical evidence that it is advantageous to use the projection sampler $P_{sw}^{(1)}$ over $P_{sw}$ to sample from $\pi_\beta$. \textcolor{black}{Thus, for the purpose of estimating $\pi^{(1)}(g)$ where $g$ is either $g(x) = x$ or $g(x) = x^2$, it is preferable to use samples generated by $P_{sw}^{(-1)}$ over that of $P_{sw}$.}

\begin{figure}[htbp]
	\centering
	
	\begin{subfigure}{\textwidth}
		\centering
		\includegraphics[width=0.45\textwidth]{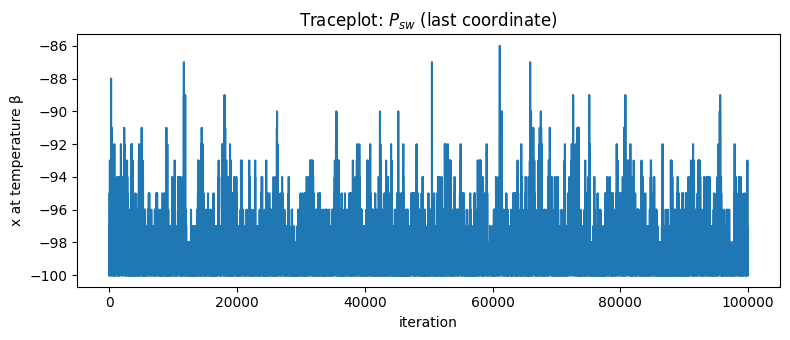}
		\includegraphics[width=0.45\textwidth]{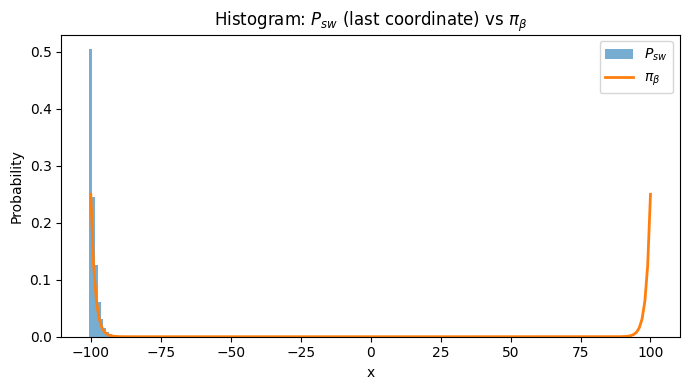}
		\caption{Traceplot and histogram of the trajectories of the last coordinate of $P_{sw}$.}
	\end{subfigure}
	
	\begin{subfigure}{\textwidth}
		\centering
		\includegraphics[width=0.45\textwidth]{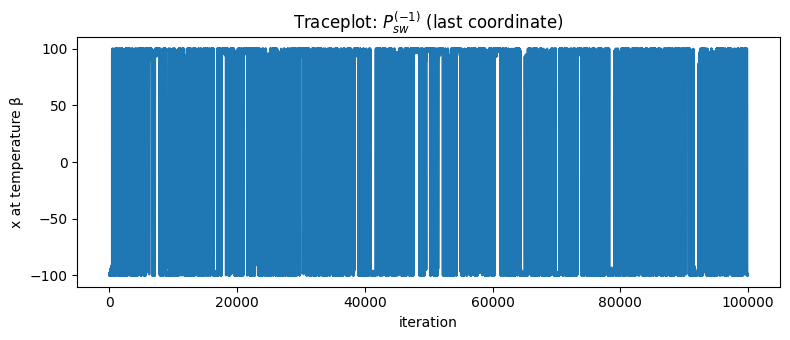}
		\includegraphics[width=0.45\textwidth]{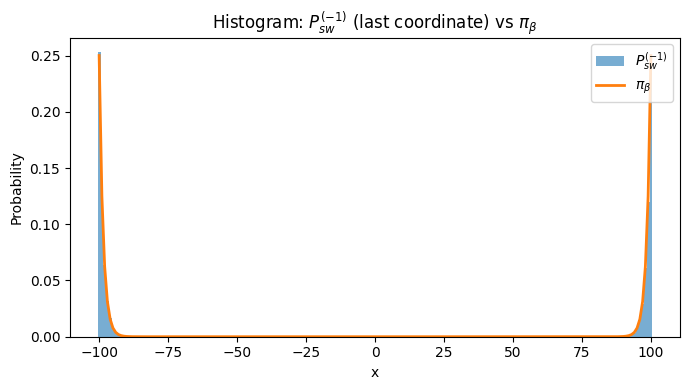}
		\caption{Traceplot and histogram of the trajectories of the last coordinate of $P_{sw}^{(-1)}$.}
	\end{subfigure}
	
	\caption{Numerical experiments comparing the two samplers $P_{sw}, P_{sw}^{(-1)}$ with target distribution being the V-shaped $\pi_\beta(x) \propto 2^{|x|}$.}
	\label{fig:Vshapeswap}
\end{figure}

\begin{table}[H]
	\centering
	\begin{tabular}{lcc}
		\toprule
		Sampler & $\widetilde{D}_{TV}(\widehat{\pi}_\beta,\pi_\beta)$ & $\widetilde{D}_{KL}(\widehat{\pi}_\beta \| \pi_\beta)$ \\
		\midrule
		$P_{sw}$ (last coordinate) & 0.50 & 0.69 \\
		$P_{sw}^{(-1)}$ (last coordinate) & 0.01 & 0.00 \\
		\bottomrule
	\end{tabular}
	\caption{Comparison of total variation distance and KL divergence between $\widehat{\pi}_\beta$ and the ground truth $\pi_\beta$, where $\widehat{\pi}_\beta$ is the empirical distribution formed by the trajectories of the samplers. }
	\label{tab:firstswap}
\end{table}

\begin{table}[H]
	\centering
	\begin{tabular}{lcc}
		\toprule
		Sampler & Mean & Second moment \\
		\midrule
		$P_{sw}$ (last coordinate) & -99.01 & 9804.85 \\
		$P_{sw}^{(-1)}$ (last coordinate) & -2.01 & 9804.27 \\
		Truth $\pi^{(1)}$ & 0 & 9803.00 \\
		\bottomrule
	\end{tabular}
	\caption{Comparison of the first and second moment between the samplers and the ground truth $\pi_\beta$.}
	\label{tab:secondswap}
\end{table}

\subsection{$\pi$-stationarity of the lifted sampler}

\textcolor{black}{In this subsection, we prove that the lifted sampler $P$ introduced in Section~\ref{subsec:liftedMCMC} is indeed $\pi$-stationary. We adopt the notations and the setting therein.}

{\color{black}
	\begin{proposition}\label{prop:pistationaritylifted}
		Suppose that $Q$ is $\pi^{(1)}$-stationary and that $a,b,c>0$ with $a+b+c=1$. Then the transition matrix $P$ defined in Section \ref{subsec:liftedMCMC} is $\pi=\pi^{(1)}\otimes \mathcal{U}(\{-1,+1\})$-stationary.
	\end{proposition}
	
	\begin{proof}
		In this proof, we write $p_x:=\pi^{(1)}(x)$, so that
		\[
		\pi(x,v)=\frac{1}{2}p_x,
		\qquad x\in \llbracket -n,n\rrbracket,\quad v\in\{-1,+1\}.
		\]
		Let $P_Q$, $P_{\mathrm{dir}}$ and $P_{\mathrm{flip}}$ denote respectively the transition matrices corresponding to the position move, the directed Metropolis move and the direction flip. Then
		\[
		P=aP_Q+bP_{\mathrm{dir}}+cP_{\mathrm{flip}}.
		\]
		It is enough to show that each of these three transition matrices is $\pi$-stationary.
		
		First, since $Q$ is $\pi^{(1)}$-stationary, for any $(y,w)\in \mathcal{X}$ we have
		\[
		\sum_{(x,v)\in \mathcal{X}} \pi(x,v)P_Q((x,v),(y,w))
		=
		\frac{1}{2}\sum_{x\in \llbracket -n,n\rrbracket} p_x Q(x,y)
		=
		\frac{1}{2}p_y
		=
		\pi(y,w).
		\]
		Hence $P_Q$ is $\pi$-stationary.
		
		Second, for the direction flip,
		\[
		P_{\mathrm{flip}}((x,v),(y,w))
		=
		\mathbf{1}_{\{y=x,\;w=-v\}}.
		\]
		Therefore
		\[
		\sum_{(x,v)\in \mathcal{X}} \pi(x,v)P_{\mathrm{flip}}((x,v),(y,w))
		=
		\pi(y,-w)
		=
		\frac{1}{2}p_y
		=
		\pi(y,w),
		\]
		and so $P_{\mathrm{flip}}$ is also $\pi$-stationary.
		
		It remains to check the directed Metropolis move. Fix $(y,w)\in \mathcal{X}$. The chain can arrive at $(y,w)$ under $P_{\mathrm{dir}}$ in two ways: either from $(y-w,w)$ by an accepted directed proposal, or from $(y,-w)$ by a rejected directed proposal. Hence
		\[
		\sum_{(x,v)\in \mathcal{X}} \pi(x,v)P_{\mathrm{dir}}((x,v),(y,w))
		=
		\frac{1}{2}p_{y-w}\alpha_w(y-w)
		+
		\frac{1}{2}p_y\bigl(1-\alpha_{-w}(y)\bigr),
		\]
		where the first term is interpreted as zero if $y-w\notin \llbracket -n,n\rrbracket$.
		By the definition of the Metropolis acceptance probability,
		\[
		p_{y-w}\alpha_w(y-w)
		=
		p_y\alpha_{-w}(y).
		\]
		Indeed, if both $y-w$ and $y$ belong to $\llbracket -n,n\rrbracket$, this is the usual identity
		\[
		p_{y-w}\min\left\{1,\frac{p_y}{p_{y-w}}\right\}
		=
		p_y\min\left\{1,\frac{p_{y-w}}{p_y}\right\}.
		\]
		If $y-w\notin \llbracket -n,n\rrbracket$, then both sides are zero by convention and by the boundary definition of $\alpha_{-w}(y)$. Therefore
		\[
		\sum_{(x,v)\in \mathcal{X}} \pi(x,v)P_{\mathrm{dir}}((x,v),(y,w))
		=
		\frac{1}{2}p_y\alpha_{-w}(y)
		+
		\frac{1}{2}p_y\bigl(1-\alpha_{-w}(y)\bigr)
		=
		\frac{1}{2}p_y
		=
		\pi(y,w).
		\]
		Thus $P_{\mathrm{dir}}$ is $\pi$-stationary.
		
		Combining the three parts gives
		\[
		\pi P
		=
		a\pi P_Q+b\pi P_{\mathrm{dir}}+c\pi P_{\mathrm{flip}}
		=
		a\pi+b\pi+c\pi
		=
		\pi.
		\]
		This proves the claim.
\end{proof}}

\section{Applications to KL-projected factored filtering in an Ising hidden Markov model (HMM)}\label{sec:factorfilter}

{\color{black}In earlier sections we have seen how the projected Markov chains can serve as improved samplers in lifted Markov chains and swapping algorithm. In addition to its applications in sampling, the projected chains can also be applied for improved approximate inference. In this section, we further expand along this direction and propose a factored filtering scheme in which the prediction step uses a product-form approximation of the original Markov transition. This yields a substantial computational gain: the per-step cost scales linearly with the dimension, whereas the exact filter scales exponentially, at the price of an approximation bias. Our numerical results further show that the distance to independence in \eqref{def:IfP} provides a meaningful quantitative proxy for this induced approximation error.

\paragraph{Latent Ising HMM}
Let $\mathcal{X} = \{-1,+1\}^d$ be a $d$-dimensional state space of an $L\times L$ grid where $d=L^2$ with neighbor sets $\mathcal{N}(i)$ at coordinate $i \in \llbracket d \rrbracket$ along with open boundary conditions. The latent dynamics $P$ is the \emph{random-scan single-site Glauber} chain $(X_t)_{t \in \mathbb{N}} = (X_t^{1},X_t^{2},\ldots,X_t^{d})_{t \in \mathbb{N}}$:
given $X_t=x$, choose $I_t\sim\mathcal{U}(\llbracket d \rrbracket)$ and resample only coordinate $I_t$ via
\[
\mathbb P(X_{t+1}^{i}=+1\mid X_t=x,\, I_t=i)=\sigma\!\Big(2\beta\sum_{j\in \mathcal{N}(i)}x_j\Big),
\qquad
X_{t+1}^{-i}=x^{(-i)},
\]
where $\sigma(z)=1/(1+e^{-z})$, $\beta\ge 0$ is the inverse temperature and $X_{t+1}^{-i}$ denotes the coordinates excluding $i$ at time $t+1$. Observations follow
\emph{node-wise flip noise}:
\[
\mathbb P(Y_t^i=X_t^i)=1-\varepsilon,\qquad \mathbb P(Y_t^i=-X_t^i)=\varepsilon,
\]
so the likelihood factorizes as $g(y_t\mid x_t)=\prod_{i=1}^d g_i(y_t^i\mid x_t^i)$ where
\[
g_i(y^i\mid x^i) :=\mathbb P(Y_t^i=y^i\mid X_t^i=x^i)
=\begin{cases}
	1-\varepsilon, & y^i=x^i,\\
	\varepsilon, & y^i=-x^i,
\end{cases}
\qquad x^i,y^i\in\{-1,+1\}.
\]
The filtering distribution is $\pi_t(x)=\mathbb P(X_t=x\mid Y_0, Y_1, \ldots Y_t)$.

\paragraph{KL-projected factored filtering}
We initialize $q_0$ to be the discrete uniform distribution on $\mathcal{X}$. We maintain a product approximation $q_t(x)=\prod_{i=1}^d q_t^{(i)}(x^i)$. At time $t$, we replace the coupled kernel $P$ by its KL-information projection onto product kernels given by
$$\widehat P_{q_t} := \otimes_{i=1}^d P^{(i)}_{q_t},$$
where each $P^{(i)}_{q_t}$ is an effective two-state transition obtained by averaging the Ising update over neighbors drawn from
$q_t$.  Suppose that $X^i_{t} = s \in \{-1,+1\}$. To simulate one step of $P^{(i)}_{q_t}$, we note that
\begin{enumerate}[label=(\roman*)]
	\item Draw the neighbor configuration $(x^j)_{j\in \mathcal{N}(i)}$ independently according to the current product belief $q_t$, i.e.,
	for each $j\in \mathcal{N}(i)$ sample $x^j\sim q_t^{(j)}(\cdot)$.
	\item Compute the probability
	\[
	p_i \;:=\; \sigma\!\left(2\beta\sum_{j\in \mathcal{N}(i)} x^j\right).
	\]
	\item Draw the updated spin $s'\in\{-1,+1\}$ by setting $s'=+1$ with probability $p_i$ and $s'=-1$ with probability $1-p_i$.
	\item With probability $1-\frac{1}{d}$ set $X_{t+1}^i \leftarrow s$ (site $i$ not selected), and with probability $\frac{1}{d}$ set $X_{t+1}^i \leftarrow s'$ (site $i$ selected).
\end{enumerate}

The filtering recursion of this factored filter is:
\[
\text{Factored filter}: \text{(predict)}\;\; \tilde q_{t+1}=q_t\,\widehat P_{q_t},\qquad
\text{(update)}\;\; q_{t+1}(x)\propto \tilde q_{t+1}(x)\,g(y_{t+1}\mid x).
\]
Because $g$ factorizes, the update is coordinate-wise and $q_{t+1}$ remains a product. For comparison, the exact (original) filter uses the full kernel $P$:
\[
\text{Exact filter}: \text{(predict)}\;\; \tilde\pi_{t+1}=\pi_t\,P,\qquad
\text{(update)}\;\; \pi_{t+1}(x)\propto \tilde\pi_{t+1}(x)\,g(y_{t+1}\mid x),
\]
which is typically intractable when $d$ is large since $\pi_t$ is a distribution on $\mathcal{X}$ of size $|\mathcal{X}|=2^d$.

\paragraph{Distance to independence}
We compute the distance to independence of $P$ at time $t$ by
\[
\mathbb{I}^{q_t}(P)=D^{q_t}_{KL}(P \| \otimes_{i=1}^d P^{(i)}_{q_t}) ,
\]
which quantifies the information loss incurred in the prediction step by enforcing independent
dynamics. In the random-scan Glauber chain, for each configuration $x$ the kernel $P(x,\cdot)$
is supported on at most $d+1$ states (self-loop plus single-site flips), so $\mathbb{I}^{q_t}(P)$ can be
estimated by Monte Carlo sampling using $q_t$ without enumerating $\{-1,+1\}^d$.

The following two experiments were run on Google Colab using a CPU backend (Intel Xeon, 2 vCPUs at $2.30$\,GHz) with $12$\,GB system RAM. The code is available at \url{https://github.com/mchchoi/factorization/tree/main}.

\subsection*{Experiment: comparison between factored filtering and exact filtering for $L = 4$}
We take $L=4$ (and hence $d = 16$) so the state space has size $2^{16}$ and exact filtering is still computationally feasible.
For fixed $(\beta,\varepsilon) = (0.6,0.1)$, we simulate one trajectory $(X_t,Y_t)_{t=0}^T$ with $T = 1000$ and run:
(i) the exact filter, and (ii) the KL-projected factored filter above. We report
\[
\mathrm{Err}_{\mathrm{marg}}(t):=\frac1d\sum_{i=1}^d\big|q_t(X^i=+1)-\pi_t(X^i=+1)\big|
\]
together with $\mathbb{I}^{q_t}(P)$. In addition, since $d$ is small we can explicitly form the product joint
$q_t(x)=\prod_{i=1}^d q_t^{(i)}(x^i)$ and compute the joint total variation discrepancy
\[
\widetilde{D}_{TV}(q_t,\pi_t)=\frac12\sum_{x\in\mathcal{X}}\big|q_t(x)-\pi_t(x)\big|.
\]
Plotting $\mathrm{Err}_{\mathrm{marg}}(t)$, $\widetilde{D}_{TV}(q_t,\pi_t)$ and $\mathbb{I}^{q_t}(P)$ versus $t$ in Figure~\ref{fig:pfexp1} illustrates that the distance to independence can serve as a diagnostic for when the independence-based approximation is accurate, both at the marginal and joint levels. In our run, the sample correlation between $\mathrm{Err}_{\mathrm{marg}}(t)$ and $\mathbb{I}^{q_t}(P)$ is $0.52$, while the sample correlation between $\widetilde{D}_{TV}(q_t,\pi_t)$ and $\mathbb{I}^{q_t}(P)$ is $0.54$; in both cases the Pearson correlation test yields a p-value close to zero, indicating statistically significant associations.}

\begin{figure}[h]
	\centering
	\includegraphics[width=0.36\textwidth]{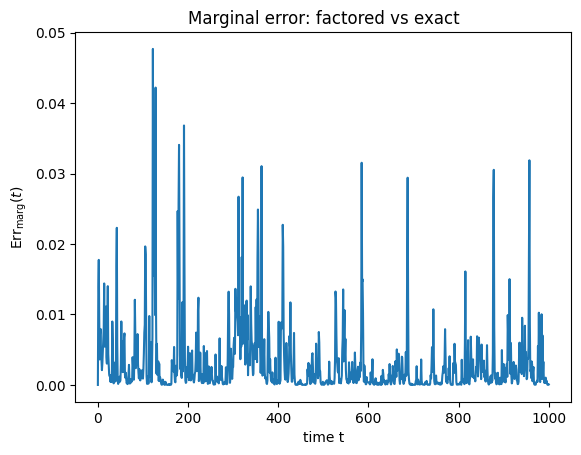}
	\includegraphics[width=0.36\textwidth]{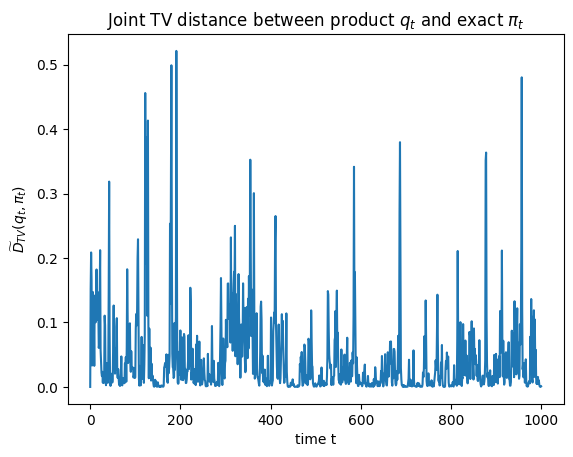}
	\includegraphics[width=0.36\textwidth]{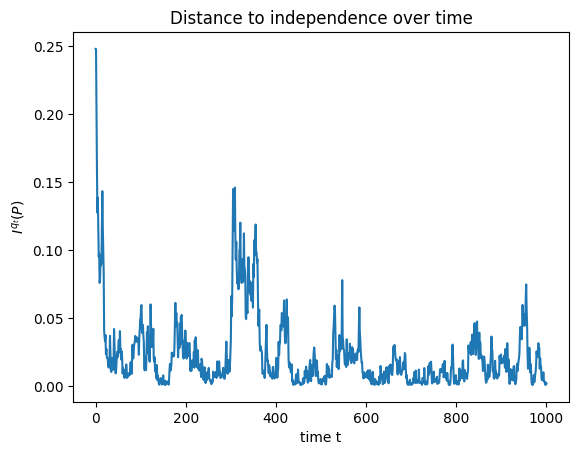}
%
	\caption{Time series of the marginal error $\mathrm{Err}_{\mathrm{marg}}(t)$, the joint total variation discrepancy $\widetilde{D}_{TV}(q_t,\pi_t)$, and the distance to independence $\mathbb{I}^{q_t}(P)$ versus time $t$. The time-averaged marginal error is $0.003$, time-averaged joint TV discrepancy is $0.043$, and time-averaged distance to independence is $0.023$. The sample correlations are $0.52$ between $\mathrm{Err}_{\mathrm{marg}}(t)$ and $\mathbb{I}^{q_t}(P)$, and $0.54$ between $\widetilde{D}_{TV}(q_t,\pi_t)$ and $\mathbb{I}^{q_t}(P)$. Parameters: $L=4$, $\beta=0.6$, $\varepsilon=0.1$, $T=1000$.}
	\label{fig:pfexp1}
\end{figure}

\subsection*{Scalability experiment: runtime scaling with $d$}
The exact filter represents $\pi_t$ on the full state space $\{-1,+1\}^d$, hence requires
$\Omega(2^d)$ memory and at least $\Omega(2^d)$ work per step to compute $\tilde\pi_{t+1}$. In contrast, the factored filter stores only $d$ marginals and performs local updates based on neighborhoods of size $\le 4$, yielding $\mathcal{O}(d)$ memory and $\mathcal{O}(d)$ time per filtering step (up to a small constant
from enumerating $2^{|\mathcal{N}(i)|} \leq 2^4$ neighbor configurations). This makes approximate inference feasible for large grids where exact filtering is impossible.

\smallskip
To demonstrate scalability, we benchmark wall-clock runtime per filtering step (prediction and update)
as a function of $d=L^2$. We report both methods where feasible (e.g. $L \in \{2,3,4\}$ for the exact filter) and run the factored filter for substantially larger $L$ (e.g. $L \in \{10,50,100\}$). For fixed $(\beta,\varepsilon) = (0.6,0.1)$, we simulate three trajectories $(X_t,Y_t)_{t=0}^T$ with $T = 30$ and report the average runtime of these three runs. Figure \ref{fig:pfexp2} plots seconds-per-step versus $d$: the exact filter rapidly becomes intractable when $L$ increases beyond $4$, while the factored filter exhibits linear growth in $d$.

\begin{figure}[htbp]
	\centering
	\includegraphics[width=0.5\textwidth]{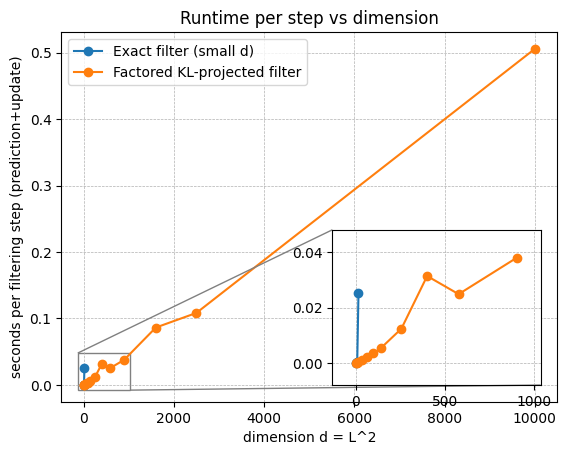}
	
%
	
	\caption{Average wall-clock runtime per filtering step (prediction and update) of the exact and factored filter over three runs in the Ising HMM model with varying $L \in \{2, 3, 4, 6, 8, 10, 12, 16, 20, 24, 30, 40, 50, 100\}$, $\beta = 0.6$, $\varepsilon = 0.1$, $T = 30$.}
	\label{fig:pfexp2}
\end{figure}

\paragraph{Advantages of the factored filter}
The KL-projected factored filter is an independence-based approximation tailored to the current
belief: it replaces the coupled transition $P$ by its KL-optimal product surrogate $\widehat P_{q_t}$, preserves tractability of both prediction and update under factorized likelihoods, and scales to high-dimensional latent chains, at the tradeoff of introducing approximation error. The diagnostic $\mathbb{I}^{q_t}(P)$ provides a quantitative proxy of approximation error in the dynamics.

\section*{Acknowledgments}
The authors gratefully acknowledge the constructive comments from the associate editor and two reviewers. Michael Choi is grateful for helpful discussions on large deviations with Pierre Del Moral, and on information theory with Cheuk Ting Li and Lei Yu. Michael Choi acknowledges the financial support of the projects A-8001061-00-00, NUSREC-HPC-00001, NUSREC-CLD-00001, A-0000178-01-00, A-0000178-02-00 and A-8003574-00-00 at National University of Singapore. Youjia Wang gratefully acknowledges the financial support from National University of Singapore via the Presidential Graduate Fellowship. GW expresses gratitude to Shun Watanabe for insightful discussions.
The work of GW was supported in part by the Special Postdoctoral Researcher Program (SPDR) of RIKEN and by the Japan Society for the Promotion of Science KAKENHI under Grant 23K13024.

\bibliographystyle{siamplain}
\bibliography{ref}
\end{document}